%% file: local_rewriting.tex
\documentclass[11pt]{amsart}
\usepackage[utf8]{inputenc}
\usepackage[T1]{fontenc}
\usepackage{graphicx}
\usepackage{grffile}
\usepackage{longtable}
\usepackage{wrapfig}
\usepackage{rotating}
\usepackage[normalem]{ulem}
\usepackage{amsmath}
\usepackage{textcomp}
\usepackage{amssymb}
\usepackage{capt-of}
\usepackage{hyperref}
\input{header}
\input{tikz_figures}

\usepackage[margin=1in]{geometry}
\date{}
\title{\textsc{GraphSAT} -- a decision problem connecting satisfiability and graph theory}
\hypersetup{
 pdfauthor={Vaibhav Karve and Anil N.~Hirani},
 pdftitle={\textsc{GraphSAT} -- a decision problem connecting satisfiability and graph theory},
 pdfkeywords={},
 pdfsubject={},
 pdfcreator={Emacs 26.3 (Org mode 9.1.9)}, 
 pdflang={English}}
\begin{document}

\title{\textsc{GraphSAT} -- a decision problem connecting satisfiability and graph theory}
\author{Vaibhav Karve}
\email{vkarve2@illinois.edu}
\address{Department of Mathematics, University of Illinois at Urbana-Champaign, 1409 West Green Street, Urbana, IL 61801}
\author{Anil N. Hirani}
\email{hirani@illinois.edu}
\address{Department of Mathematics, University of Illinois at Urbana-Champaign, 1409 West Green Street, Urbana, IL 61801}
\maketitle

\begin{abstract}
Satisfiability of boolean formulae (SAT) has been a topic of research in
logic and computer science for a long time. In this paper we are interested
in understanding the structure of satisfiable and unsatisfiable sentences.
In previous work we initiated a new approach to SAT by formulating a
mapping from propositional logic sentences to graphs, allowing us to find
structural obstructions to 2SAT (clauses with exactly 2 literals) in terms
of graphs. Here we generalize these ideas to multi-hypergraphs in which the
edges can have more than 2 vertices and can have multiplicity. This is
needed for understanding the structure of SAT for sentences made of clauses
with 3 or more literals (3SAT), which is a building block of
NP-completeness theory. We introduce a decision problem that we call
GraphSAT, as a first step towards a structural view of SAT. Each
propositional logic sentence can be mapped to a multi-hypergraph by
associating each variable with a vertex (ignoring the negations) and each
clause with a hyperedge. Such a graph then becomes a representative of a
collection of possible sentences and we can then formulate the notion of
satisfiability of such a graph. With this coarse representation of classes
of sentences one can then investigate structural obstructions to SAT. To
make the problem tractable, we prove a local graph rewriting theorem which
allows us to simplify the neighborhood of a vertex without knowing the rest
of the graph. We use this to deduce several reduction rules, allowing us
to modify a graph without changing its satisfiability status which can then
be used in a program to simplify graphs. We study a subclass of 3SAT by
examining sentences living on triangulations of surfaces and show that for
any compact surface there exists a triangulation that can support
unsatisfiable sentences, giving specific examples of such triangulations
for various surfaces.
\end{abstract}

\section{Introduction}
\label{sec:org67413fc}
We introduce and analyze a novel graph decision problem that we call
\GraphSAT. Using the tools of graph theory, this new variant builds upon
the classical logic and computer science problem of boolean satisfiability
(\(k\)\SAT). \(k\)\SAT asks if there exists a truth assignment that
satisfies a given boolean formula. Our variant deals with multi-hypergraphs
instead of boolean formulae and uses truth assignments on vertices instead
of variables. This graph-theoretic picture helps us explore and exploit
patterns in unsatisfiable instances of \(k\)\SAT, which in turn helps us
identify minimal obstruction sets to graph satisfiability.

In Theorem \ref{thm-local_rewriting} (the local rewriting theorem) we prove
invariance of graph satisfiability under a graph rewriting system that we
introduce. This theorem is the main conceptual and computational tool that
allows us to replace the question of satisfiability of a graph by those for
smaller graphs. This theorem makes possible a computational search for
patterns in satisfiable and unsatisfiable graphs without knowing the
entirety of the graph. The computational results in this paper, for example
the search for unsatisfiable instances of \GraphSAT, were obtained using a
Python package (called \texttt{graphsat}) that we created to handle
multi-hypergraph instances and local rewriting of graphs. This software
will be described elsewhere. Our goal in this paper is to introduce
\GraphSAT and logical operations on graphs and to prove the local graph
rewriting theorem as well as to detail consequences of this theorem.

This paper is a successor to \cite{KaHi2020}, which looks at \GraphSAT
restricted to multi-graphs and \(2\)\SAT. The key results in that paper are
that graph homeomorphisms preserve satisfiability status and a complete set
of minimal unsatisfiable simple graphs. Specifically, we showed that there
are exactly four obstructions to satisfiability in simple graphs and
exactly four obstructions to satisfiability in looped-multi-graphs.

\begin{theorem*}[Theorem 18 and Remark 19 in \cite{KaHi2020}]
  A simple graph is unsatisfiable if and only if it contains an element of
  the set \[{\raisebox{2.5pt}{\Big\{}\;\Butterfly,\;\;\Bowtiesmall,\;\;
  \Kfour,\;\;\Book\raisebox{2.5pt}{\Big\}}}\] as a topological minor.
  A looped-multi-graph is unsatisfiable if and only if it contains an
  element of the set \[{\raisebox{1ex}{\Big\{}\;\FlexibleButterfly{},\;
  \;\FlexibleBowtie,\; \; \FlexibleKfour,\;
  \;\FlexibleBook\raisebox{1ex}{\Big\}}}\] as a topological minor.
\end{theorem*}

This current paper is the first step towards generalizing the \(2\)\SAT
results of \cite{KaHi2020} to \(k\)\SAT. By studying properties of
satisfiable and unsatisfiable graphs we are studying entire classes of
boolean formulae. This results in a coarsening compared to studying Cnfs
directly. Nevertheless, the translation from Cnfs to graphs capture the
essential features of any formula and offers a complete picture of the
structure of its boolean constraints. For example, we show that this notion
of satisfiability is closed under subgraphing as well as under the
relations of ``shaved graph versions'' and topological-minoring. The family
of satisfiable multi-hypergraphs is also closed under the action of a
variety of higher-dimensional analogues of well-known graph operations. We
explore these analogous graph operations and analyze their effects
vis-à-vis the structural properties of multi-hypergraphs.

Robertson and Seymour, in their seminal papers \cite{RoSe2004} proved that
a graph family that is closed under the graph minor operation always has a
finite obstruction set. This theorem also does not apply to the case of
graphs for \(2\)\SAT considered in \cite{KaHi2020}. Despite this, in
\cite{KaHi2020}, we were able to obtain a very short list of obstructions
listed in the Theorem above. The case of \(k\)\SAT for \(k>2\) considered
here is considerably harder. There is no known finite obstruction
Robertson-Seymour theorem for multi-hypergraphs. Nevertheless, we obtain
several hundred obstructions to satisfiability for multi-hypergraphs. We
enumerate a part of this list in Appendix \ref{sec:org9d7b11d}.

Section \ref{sec:org93a5019} introduces definitions and notations used
in this paper. It introduces some type-theoretic notation that streamlines
the mathematical exposition. It also includes definitions for the two
halves of \GraphSAT, i.e. boolean formulae and graphs.

In Section \ref{sec:org7839d0f} we state and prove the local graph
rewriting theorem (Theorem \ref{thm-local_rewriting}), that enables rewriting of
graphs at a vertex, while preserving its satisfiability status. In Section
\ref{sec:org60b361b} we list some graph rewrite rules that leave graph
satisfiability invariant. In \ref{sec:org34ae191} we give a result that
is true for edges of any size. Thereafter, we restrict our attention to
edges of size at most three. In Sections \ref{sec:orga404456}, \ref{sec:org5b66ec3} and \ref{sec:org0d3ace9} we consider the satisfiability of mixed hypergraphs,
triangulations, and infinite graphs respectively. In these sections we
summarize some of the key findings enabled by the local rewriting theorem
(Theorem \ref{thm-local_rewriting}), and by our \texttt{graphsat} Python package. This
software package along with the local rewriting theorem allows us to
conduct experiments combining programming and proofs to classify graphs,
hypergraphs, and infinite graphs as totally satisfiable or unsatisfiable.
Section \ref{sec:org23f1643} provides a conclusion to this
study by summarizing key results as well as future directions and
conjectures.

\section{Definitions and notation}
\label{sec:org93a5019}
We start by inductively defining boolean formulae in conjunctive normal
form in \S \ref{sec:org1618a2f} and graphs (in fact multi-hypergraphs) in \S
\ref{sec:org2ef3df6}. We introduce a way to view graphs as sets of
Cnfs in \S \ref{sec:org2441f87}. After translating Cnfs to graphs, we
also define a notion of satisfiability for graphs in \S \ref{sec:orgcbfe531}. We also define notation and conventions aimed at
making the connection between Cnfs and graphs more intuitive. These are
summarized in Table \ref{table-operator-summary}.

\subsection{Type theory annotations}
\label{sec:org9892acd}
Throughout this section, we add various annotations to our terms in order
to aid the reader in parsing and understanding the mathematics presented
herein. These annotations are inspired from the field of type theory and
can be thought of as representing the ``category'' or the ``type'' of a term.

For example, we write \(c : \Clause = x_1 \vee x_2\) to mean \(c\) is of
type Clause and is equal to \(x_1 \vee x_2\). The type annotations help
with clarity without changing the mathematical content. For this reason, we
use them wherever they can add to the exposition and avoid them wherever
they might be unnecessarily verbose.

We will use a fixed list of types in this section, collected in Table
\ref{table-type-summary}. Full definitions for each type can be found in the
sections that follow. Let \(V\) denote an arbitrary type, used to
parameterize the other types. Elements of \(V\) will be called variables.
For example, we define the type \(\Cnf\;V\) to be the type of all Cnfs on
the variable set \(V\). In practice, we will avoid mentioning \(V\)
explicitly and simply write type judgments like \((x : \Cnf)\) instead of
\((x : \Cnf\;V)\).

In Table \ref{table-type-summary}, we use the notation \(A\equiv B\) to be mean
that the two types have exactly the same terms. We write \(A\sqsubset B\)
to mean that \(A\) is a subtype of \(B\). This indicates that there are
some extra restrictions placed on \(A\). For example, we have \(\Clause
\sqsubset \Set \Literal\) because every clause can also be viewed as a set
of literals. However, not every set of literals is a clause because we
require that clauses be nonempty. We also use \(\oplus\) and \(\times\) to
denote the disjoint-sum and Cartesian-product of types respectively.

\begin{table}
\footnotesize
\caption{Summary of all the types defined and used in this section.}
\begin{center}
\begin{tabular}{p{2.1cm}  p{3.9cm} l}
\hline
Type & Relation to other types & Description\\
\hline
Variable & Variable \(\equiv V\) & variables; an alias of \(V\)\\
Literal \(V\) & Literal \(\equiv\) \(V\oplus V \oplus\) Bool & literals of variable type \(V\)\\
Clause \(V\) & Clause \(\sqsubset\) Set Literal & clauses of variable type \(V\)\\
Cnf \(V\) & Cnf \(\sqsubset\) Set Clause & Cnfs of variable type \(V\)\\
Assignment \(V\) & Assignment \(\sqsubset\) Set Literal & assignment of variable type \(V\)\\
 &  & \\
Vertex & Vertex \(\equiv V\) & vertices; an alias of \(V\)\\
Edge \(V\) & Edge \(\sqsubset\) Set Vertex & hyperedges of vertex type \(V\)\\
Graph \(V\) & Graph \(\sqsubset\) Multiset Edge & multi-hypergraphs of vertex type \(V\)\\
 &  & \\
Bool &  & the type of boolean values (true and false)\\
\(\mathbb{N}\) &  & the type of natural numbers\\
Set \(V\) &  & homogeneous sets of elements of type \(V\)\\
Multiset \(V\) & Multiset \(V\equiv\) \(\mathbb{N}\times \Set V\) & homogeneous multi-sets of elements of type \(V\)\\
\hline
\end{tabular}
\end{center}
\label{table-type-summary}
\end{table}

\subsection{Boolean formulae}
\label{sec:org1618a2f}
In this section, we inductively define variables, literals, clauses, Cnfs
and assignments. We then define satisfiability and equi-satisfiability for
Cnfs.

We fix an arbitrary countable set and call its elements \emph{variables}.
Variables will be denoted by \(x_1, x_2, y_1, y_2,\) etc. We denote the
type of variables as Variable and thus can write \((x_2 : \Variable)\) to
mean that \(x_2\) is a variable.

A \emph{literal} is either a variable, denoted by the same symbol as the
variable; or its negation, denoted by \(\neg x_i\) or \(\overline{x_i}\).
We also declare that there are two additional literals called true (denoted
by \(\top\)) and false (denoted by \(\bot\)).

We denote the type of literals as Literal and can thus write
\(\overline{x_2} : \Literal\) to mean that \(\overline{2}\) is a literal.
The Literal type is thus composed of two copies of \(V\) and one copy of
true and false. We can therefore write Literal \(\equiv\) \(V\oplus V\oplus
\{\top, \bot\}\).

A \emph{clause} is a disjunction of one or more literals. A clause made of a
single literal \(x_i\) is also denoted by \(x_i\). We denote the type of
clauses as Clause. Thus, the following are all valid examples of clauses ---
\((\top : \Clause),\; (\bot : \Clause),\; (x_1 : \Clause),\;
(x_1\vee\overline{x_2} : \Clause)\).

A \emph{Cnf} is a boolean formula in \emph{conjunctive normal form}, i.e. it is a
formula made of conjunction of one or more clauses. A Cnf containing a
single clause has the same representation as the clause itself. Thus, the
following are all valid Cnfs --- \((\top : \Cnf),\; (\bot : \Cnf),\; (x_1 :
\Cnf),\; (x_1\vee\overline{x_2} : \Cnf),\; ((x_1\vee\overline{x_2})\wedge
x_2\vee\overline{x_3}) : \Cnf)\).

\subsection{Assignments}
\label{sec:org5ecd0fd}
A \emph{truth assignment} (or simply an \emph{assignment}) is a set of literals with
the additional condition that a literal and its negation cannot both belong
to an assignment. We denote the type of assignments by Assignment.

We define the action of a singleton assignment \(\{i\}\) on a Cnf \(x\) to
be the Cnf formed by replacing every occurrence of \(i\) in \(x\) with
\(\top\) and every occurrence of \(\overline{i}\) in \(x\) with \(\bot\).
We define the action of a general assignment \(a\) on a Cnf \(x\) to be the
Cnf formed by the successive action of each element of \(a\) on \(x\).
We denote this new Cnf by \(x[a]\).

For example, \((x_1\vee\overline{x_2})\wedge (x_1\vee x_2\vee\overline{x_3})
\wedge(\overline{x_1}\vee x_4) \left[\{x_1,\overline{x_2}\}\right]
\;=\; (\top\vee\overline{x_2})\wedge (\top\vee x_2\vee\overline{x_3})
          \wedge(\bot\vee x_4)
          \left[\{\overline{x_2}\}\right]
\;=\; (\top\vee\top)\wedge (\top\vee \bot\vee\overline{x_3})
          \wedge(\bot\vee x_4)
\;=\; x_4\).

We simplify this notation by writing \(x[x_1, \overline{x_2}]\) instead of
\(x[\{x_1, \overline{x_2}\}]\).

\subsection{Satisfiability of Cnfs}
\label{sec:org1954548}
A Cnf \(x\) is \emph{satisfiable} if there exists a truth assignment \(a\) such that
\(x[a]=\top\). Otherwise, \(x\) is \emph{unsatisfiable}. Two Cnfs \(x\) and \(y\) are
\emph{equi-satisfiable} (denoted \(x\sim y\)) if either they are both
satisfiable or they are both unsatisfiable. We write \(x\sim\top\) or
\(x\sim\bot\) to denote that \(x\) is satisfiable or unsatisfiable
respectively.

We also introduce a map for Cnf-satisfiability denoted \(\sigma : \Cnf
\longrightarrow \text{Bool}\), which maps satisfiable Cnfs to \(\top\) and
unsatisfiable Cnfs to \(\bot\).

We note that assigning a literal \(i\) to \(\top\) does not change the
satisfiability status of a Cnf \(x\), i.e \(x[i] \sim x\wedge i\).

\subsection{Graphs and Satisfiability}
\label{sec:org2ef3df6}
In this section we define graphs inductively starting with vertices and
then building up through edges (looped, simple and hyper) and then
multi-edges. The standard term for these graph objects would be
``looped-multi-hypergraphs''. We then define a novel way of interpreting a
graph as a set of Cnfs that can ``live on that graph''. This translation of
graphs into sets of Cnfs is lies at the heart of our attempt to turn
boolean satisfiability instances into graph theory problems. Consequently,
we make this translation as explicit as possible and provide enough details
so that the interested reader may generate a translation ``algorithm'' from
our definitions to readily turn any graph into a set of Cnfs.

Let \(V\) be a countable set of \emph{vertices}. We will omit any mention of a
specific \(V\) and simply denote its elements using boldface symbols
\(\g{v_1}, \g{v_2}, \ldots : \Vertex\). A \emph{hyperedge} (or simply an edge)
on \(V\) is a nonempty set of vertices. We omit the surrounding braces while
denoting edges and simply write the vertices in a contiguous fashion. For
example, \(\g{v_1}\), \(\g{v_1v_2}\) and \(\g{v_2v_3v_4}\) denote edges of size
one (a loop), two (a simple edge) and three (a triangle) respectively.

A \emph{graph} on \(V\) is a nonempty multiset of edges on \(V\). In literature,
these are typically referred to as \emph{multi-hypergraphs} or
\emph{looped-multi-hypergraphs}, but we simply call them graphs. When listing
the edges of a graph, we omit the surrounding braces and wrap each edge in
parentheses, listing them one after the other. For example, \(\g{g} :
\Graph = \g{(v_1v_2)(v_1v_2v_3)(v_1v_4)}\). If an edge repeats in a graph,
we denote its multiplicity as a superscript. For example, \(\g{g'} : \Graph
= \g{(v_1v_2)(v_2v_3)^3(v_4)^2}\).

The \emph{degree} of a vertex in an edge is \(1\) if the vertex is an element of
the edge, and \(0\) otherwise. The \emph{degree} of a vertex in a graph is the
sum of the degrees of the vertex over all edges of the graph, while
counting multiplicities as distinct edges. We note that we draw a loop as
an edge starting an ending at the same vertex, but this is still considered
to contribute only \(1\) to the degree count of said vertex.

\subsubsection{Graphs as sets of Cnfs}
\label{sec:org2441f87}
We define the functions \(f_1, f_2, f_3, f_4\) as follows ---
\begin{itemize}
\item \(f_1 : \Vertex \rightarrow \Set\Literal\), such that \(\g{v} \mapsto \{v, \overline{v}\}\)
\item \(f_2 : \Edge \rightarrow \Set\Clause\), such that \(\g{v_1v_2\cdots v_k}
  \mapsto \left\{l_1\vee l_2\vee\cdots\vee l_k \mid (l_i : \Literal)\in
  f_1(\g{v_i})\right\}\)
\item \(f_3 : \mathbb{N}\times\Edge \rightarrow \Set\Cnf\), such that \((n,
  \g{e}) \mapsto \{c_1\wedge \cdots \wedge c_n \mid (c_i : \Clause) \in
  f_2(\g{e}) \text{ and } c_i\neq c_j \}\)
\item \(f_4 : \Graph \rightarrow \Set\Cnf\), such that
\(\g{(e_1)^{n_1}(e_2)^{n_2}\cdots (e_k)^{n_k}} \mapsto \left\{x_1 \wedge
  \cdots \wedge x_k \mid (x_i : \Cnf) \in f_3(n_i, \g{e_i})\right\}\)
\end{itemize}

The image of an edge under \(f_3\) gives us the set of all Cnfs that ``live
on that edge''. There are \(\displaystyle\binom{2^k}{n}\) Cnfs in the set corresponding
to an edge of size \(k\) with multiplicity \(n\). The image of a graph
under \(f_4\) gives us the set of all Cnfs that ``live on that graph''. We
use this set often enough that we will simply omit writing \(f_1, f_2,
f_3\), and \(f_4\) and we will conflate the graph with its set of Cnfs. For
example, the following graph is also a set.
\begin{align*}
\g{g} &= \g{(v_1v_2)(v_1v_3)} : \Graph\\
          &= \left\{(v_1\vee v_2)\wedge(v_1\vee v_3),\; (v_1\vee v_2)\wedge
             (v_1\vee\overline{v_3}),\; \ldots,\;
             (\overline{v_1}\vee\overline{v_2})\wedge(\overline{v_1}
              \overline{v_3})\right\} : \Set\Cnf
\end{align*}
Therefore, when we write set operations on graphs, say like
\(\g{g}\cup\g{h}\), we really mean \(f_4(\g{g})\cup f_4(\g{h})\).

We further define two special sets of Cnfs, which we denote by boldface
true and false symbols
\begin{itemize}
\item \(\btop : \Set\Cnf = \{(\top : \Cnf)\}\)
\item \(\bbot : \Set\Cnf = \{(\bot : \Cnf)\}\)
\end{itemize}
Note that both of these are sets of Cnfs that are not graphs (due to the
requirement that a graph be a nonempty multiset of edges). Similarly, the
empty set is also a valid term of type Set Cnf but is not a valid term of
type Graph.

The set of Cnfs living on a graph can sometimes be empty. For example, the
following graphs correspond to empty sets --- \(\g{(v_1)^3}, \;
\g{(v_1v_2)^5}, \; \g{(v_1v_2v_3)^9}\). These sets are empty because their
image under \(f_3\) is empty. In general, a graph is empty if and only if
it has an edge of size \(k\) with multiplicity more than \(2^k\).

\subsubsection{Graph equi-satisfiability}
\label{sec:orgcbfe531}
A set of Cnfs \(\g{g}\) is \emph{totally satisfiable} if it is nonempty and if
every Cnf in it is satisfiable. Otherwise, it is \emph{unsatisfiable}. We denote
by \(\bgamma : \Set\Cnf \rightarrow \text{Bool}\), the map that sends a
Graph to \(\top\) if it is totally satisfiable, and to \(\bot\) otherwise.
We note that for any set of Cnfs \(\g{g_1}\) and \(\g{g_2}\), we have
\(\bgamma(\g{g_1} \cup \g{g_2}) = \bgamma(\g{g_1}) \wedge
\bgamma(\g{g_2})\).

Two sets of Cnfs \(\g{g_1}\) and \(\g{g_2}\) are \emph{equi-satisfiable} if they
are both totally satisfiable or are both unsatisfiable. We write this as
\(\g{g_1}\bsim\g{g_2}\).

We say that \(\g{g_1}\) \emph{equi-implies} \(\g{g_2}\) by the
\(\bot\)-criterion (denoted \(\g{g_1}\xRightarrow{\;\bot\;}\g{g_2}\)) if,
\(\forall (v_1 : \Cnf) \in \g{g_1},\; \exists (v_2 : \Cnf) \in
\g{g_2},\;v_1\sim\bot \implies v_2\sim\bot\). If both \(\g{g_1}
\xRightarrow{\;\bot\;} \g{g_2}\) and \(\g{g_2} \xRightarrow{\;\bot\;}
\g{g_1}\), then we denote this more compactly as \(\g{g_1}
\xLeftrightarrow{\;\bot\;} \g{g_2}\).

We say that \(\g{g_1}\) \emph{equi-implies} \(\g{g_2}\) by the
\(A\bot\)-criterion (denoted \(\g{g_1}\xRightarrow{\;A\bot\;}\g{g_2}\)) if,
\(\forall (v_1 : \Cnf) \in \g{g_1},\; \exists (v_2 : \Cnf) \in \g{g_2},\;
\forall (a : \Assignment),\) \(v_1[a]=\bot \implies v_2[a]=\bot\). If both
\(\g{g_1} \xRightarrow{A\bot} \g{g_2}\) and \(\g{g_2} \xRightarrow{A\bot}
\g{g_1}\), then we denote this more compactly as \(\g{g_1}
\xLeftrightarrow{A\bot} \g{g_2}\).

\begin{remark}
The idea of graph equi-satisfiability is not necessarily useful for a
specific pair of graphs, but it is useful for discovering operations that
preserve satisfiability like when a part of the graph structure is fixed
and not changing.
\end{remark}
\subsubsection{The \GraphSAT decision problem}
\label{sec:org84170fa}
\GraphSAT is the graph decision problem that asks if a given
looped-multi-hypergraph is totally satisfiable.
\begin{itemize}
\item \textbf{Instance}: Given a specific looped-multi-hypergraph \(\g{g}\).
\item \textbf{Question}: Is every Cnf \(x\) such that \(x\in\g{g}\) satisfiable.
\end{itemize}

\(k\)\GraphSAT is a restriction of \GraphSAT that only allows
looped-multi-hypergraphs with hyperedge sizes at most \(k\).

In \cite{KaHi2020} we enumerated a complete list of minimal unsatisfiable
simple graphs and a proof that \(2\)\GraphSAT is in complexity class P. We
might hope for a similar complete list for \(3\)\GraphSAT, and perhaps a
different complexity class for \(3\)\GraphSAT. Unfortunately, this is not
the case. The \(3\)\GraphSAT problem is more complicated for three main
reasons ---
\begin{enumerate}
\item The minimality of the \(2\)\GraphSAT list hinged on homeomorphisms (i.e.
edge-subdivisions) preserving graph satisfiability. There is no single
analogue of homeomorphisms, edge-subdivisions, and topological minoring
in the case of hypergraphs.
\item Bruteforce checking of the satisfiability status of a graph is not a
sustainable option for \(3\)\GraphSAT owing to the large number of Cnfs
supported by a typical looped-multi-hypergraph.
\item In studying \(2\)\GraphSAT, we argued (in \cite{KaHi2020}) that we can
always reduce higher multiplicity edges down to \(1\). This was because
a multiplicity \(4\) edge is always unsatisfiable, a multiplicity \(3\)
edge forces an assignment on its vertices, and a multiplicity \(2\) edge
forces an equivalence on its vertices. Such a complete result does not
exist for hyperedges. The best we can do is to say that a multiplicity
\(8\) hyperedge is always unsatisfiable, and a multiplicity \(7\)
hyperedge forces assignments on its vertices.
\end{enumerate}

These factors conspire to make \(3\)\GraphSAT a harder problem to solve,
but also yield richer structures and relations between various
multi-hypergraphs and their satisfiability statuses.

\subsubsection{Disjunction and conjunction of graphs}
\label{sec:orgf3d1fdf}
We define \emph{disjunction} for sets of Cnfs to be a binary operation \(f_1\)
such that \[f_1 : \Set\Cnf \times \Set\Cnf \rightarrow \Set\Cnf,\; (s_1,
s_2) \mapsto \left\{f_2(x_1\vee x_2) \mid (x_1 : \Cnf) \in s_1,\; (x_2 :
\Cnf) \in s_2\right\},\] where, \(f_2 : \text{Boolean-Formula} \rightarrow
\Cnf\), is a map that converts boolean formulae into Cnfs by repeatedly
using distributivity of disjunction over conjunction. We note that this is
only one of several possible implementations of \(f_2\). We choose this
implementation because it has the advantage of not introducing any new
variables. In a computer implementation of \(f_2\), the use of
distributivity is inefficient. But, since we are concerned only with small
graphs, this is not an issue. We write \(\g{g_1\vee g_2}\) instead of
\(f_1(\g{g_1}, \g{g_2})\).

We define \emph{conjunction} for sets of Cnfs to be a binary operation \(f_3\)
such that \(f_3 : \Set\Cnf \times \Set\Cnf \rightarrow \Set\Cnf,\; (s_1,
s_2) \mapsto \left\{x_1\wedge x_2 \mid (x_1 : \Cnf) \in s_1,\; (x_2 : \Cnf)
\in s_2\right\}\). We write \(\g{g_1\wedge g_2}\) instead of \(f_3(\g{g_1},
\g{g_2})\).

We note that disjunction of graphs need not result in a graph, while
conjunction of graphs always does. We now list some properties of graph
disjunction, graph conjunction and equi-satisfiability.

\begin{prop}
Let \(\g{g_1}\) and \(\g{g_2}\) be sets of Cnfs.
\begin{enumerate}
\item \(\g{g_1}\bsim\g{g_2}\) if and only if \(\g{g_1}
   \xLeftrightarrow{\;\bot\;} \g{g_2}\). \\In other words, we can restrict our
   attention to only finding unsatisfiable Cnfs in both sets in order to
   prove their equi-satisfiability. \label{item:prop1}

\item \(\g{g_1} \xRightarrow{\;A\bot\;}
  \g{g_2}\) implies \(\g{g_1} \xRightarrow{\;\bot\;}
  \g{g_2}\) for every \(\g{g} : \Set \Cnf\). \\In other words,
  equi-implication by the \(A\bot\)-criterion is a stronger condition than
  the \(\bot\)-criterion.

\item \(\g{g_1} \xRightarrow{\;A\bot\;}
  \g{g_2}\) implies \(\g{g\wedge g_1} \xRightarrow{\;A\bot\;}
  \g{g\wedge g_2}\) for every \(\g{g} : \Set \Cnf\).  \label{item:prop3}

\item \(\g{g_1} \bsim \g{g_2}\) implies \(\g{g\wedge g_1 \sim
  g\wedge g_2}\), for every \(\g{g} : \Set\Cnf\).
  \\In other words, particularly when dealing with graphs,
  we can factor out the largest common subgraph and restrict our
  attention to finding Cnfs on the remaining edges in order to prove
  graph equi-satisfiability.

\item \(\g{g_1\vee g_2} \xRightarrow{A\bot} \g{g_1}\), and similarly,
  \(\g{g_1\vee g_2} \xRightarrow{A\bot} \g{g_2}\).

\item If both \(\g{g_1}\) and \(\g{g_2}\) are unsatisfiable, then we can
  conclude that \(\g{g_1\vee g_2}\) is also unsatisfiable.

\item If we take our arrows to be equi-implications under the
  \(A\bot\) criterion, then graph disjunction obeys the universal property
  of products, while union of Cnf sets possesses the universal property of
  sums.
\begin{enumerate}
  \item If \(\g{g}\xRightarrow{A\bot}\g{g_1}\) and
     \(\g{g}\xRightarrow{A\bot}\g{g_2}\), then
     \(\g{g}\xRightarrow{A\bot}\g{g_1\vee g_2}\).
  \item If \(\g{g_1}\xRightarrow{A\bot}\g{g}\) and
     \(\g{g_2}\xRightarrow{A\bot}\g{g}\), then
     \(\g{g_1}\cup\g{g_2}\xRightarrow{A\bot}\g{g}\). \label{item:prop7b}
\end{enumerate}

\item Let \(\g{g}\) be a set of Cnfs. If either one of \(\g{g\wedge g_1}\)
  or \(\g{g\wedge g_2}\) is totally satisfiable, then, so is \(\g{g\wedge
  }(\g{g_1 \vee g_2})\). \label{item:prop8}

\item \(\g{g_1}\xRightarrow{A\bot}\g{g_2}\) implies \((\g{g\wedge g_1})
  \cup (\g{g\wedge g_2) \sim g\wedge g_2}\), for every \(\g{g} :
  \Set\Cnf\). Proposition \ref{prop-conjunction-disjunction}.\label{item:prop9}

\end{enumerate}
\label{prop-conjunction-disjunction}
\end{prop}

\begin{proof}
~\newline
\begin{enumerate}

\item Suppose that \(\g{g_1\sim g_2}\). If both sets are totally
  satisfiable, then the statement is vacuously
  true. If both sets are unsatisfiable, then for every unsatisfiable Cnf in
  \(\g{g_1}\), we choose an unsatisfiable Cnf in \(\g{g_2}\). Similarly,
  for every unsatisfiable Cnf in \(\g{g_2}\), we can choose a corresponding
  unsatisfiable Cnf in \(\g{g_1}\), thus satisfying the requirements of the
  \(\bot\)-criterion of equi-implication. Conversely, if
  \(\g{g_1}\xLeftrightarrow{\;\bot\;}\g{g_2}\), then given an
  unsatisfiable Cnf in \(\g{g_1}\) we can obtain an unsatisfiable Cnf in
  \(\g{g_2}\). Thus, either both sets are totally satisfiable or both are
  unsatisfiable.

\item Proof follows from definition of the \(A\bot\)- and
  \(\bot\)-criterions of equi-implication.

\item If either \(\g{g}\), \(\g{g_1}\), or \(\g{g \wedge g_1}\) is empty,
  the conclusion is vacuously true. Let us suppose then that these sets are
  all nonempty. Let \((x \wedge x_1:\Cnf)\) be an element of
  \(\g{g \wedge g_1}\). Using the hypothesis, we can obtain a Cnf \(x_2\)
  such that \(x_2\) is falsified by every assignment that falsifies
  \(x_1\).

  Any assignment that falsifies \(x\wedge x_1\) also falsifies
  \(x\wedge x_2\). We have therefore proved that \(\g{g\wedge g_1}\)
  equi-implies \(\g{g\wedge g_2}\) by the \(A\bot\)-criterion.

\item The proof follows from items \ref{item:prop1} and \ref{item:prop3}.

\item Let \((x : \Cnf) \in \g{g_1\vee g_2}\). We can factor \(x\) as \(x =
  x_1\vee x_2\) for some \((x_1 : \Cnf)\) and \((x_2 : \Cnf)\), such that
  \(x_1\in\g{g_1}\) and \(x_2\in\g{g_2}\). We choose \(y=x_1\). Then, any
  assignment that falsifies \(x\) falsifies both \(x_1\) and \(x_2\). In
  particular, any such assignment also falsifies \(y\).

\item Let \((x_1 : \Cnf) \in \g{g_1}\) and \((x_2 : \Cnf) \in \g{g_2}\) be
  unsatisfiable Cnfs. Any assignment that satisfies \(x_1\vee x_2\) must also
  satisfy either \(x_1\) or \(x_2\) or both. Hence, we can conclude that no
  assignment satisfies \(x_1\vee x_2\). This proves that \(\g{g_1\vee g_2}\)
  is unsatisfiable.

\item For (a), we note that for every \((x : \Cnf)\in\g{g}\), there exist
  Cnfs \(x_1\in\g{g_1}\) and \(x_2\in\g{g_2}\) such that any assignment
  that falsifies \(x\) also falsifies both \(x_1\) and \(x_2\). Thus, any
  assignment that falsifies \(x\) also falsifies \(x_1\vee x_2\). For (b),
  we note that for every \((x_1 : Cnf)\in\g{g_1}\), there exists a
  Cnf \(x\in\g{g}\) such that any assignment that falsifies \(x_1\) also
  falsifies \(x\). This covers every Cnf in \(\g{g_1}\cup\g{g_2}\) coming
  from \(\g{g_1}\). A similar argument hold for every Cnf coming from
  \(\g{g_2}\). This proves that \(\g{g_1}\cup\g{g_2}
  \xRightarrow{A\bot}\g{g}\).

\item Suppose \(\g{g\wedge }(\g{g_1 \vee g_2})\) is unsatisfiable. Then,
  there exist Cnfs \(x\in\g{g}\), \(x_1\in\g{g_1}\), and \(x_2\in\g{g_2}\)
  such that \(x\wedge(x_1\vee x_2)\) is unsatisfiable. Then, both \(x\wedge
  x_1\sim\bot\) and \(x\wedge x_2\sim\bot\). In other words, both
  \(\g{g\wedge g_1}\) and \(\g{g\wedge g_2}\) are unsatisfiable.

\item It suffices to show that \(\g{g_1}\cup\g{g_2} \xLeftrightarrow{A\bot}
  \g{g_2}\). \(\g{g_1}\cup\g{g_2} \xRightarrow{A\bot} \g{g_2}\) follows
  from \ref{item:prop7b}. \(\g{g_2} \xRightarrow{A\bot} \g{g_1}\cup\g{g_2}\)
  is trivially true.

\end{enumerate}
\end{proof}

\begin{remark}
We note that graph conjunction and graph disjunction have several familiar
properties while also missing several properties that we are used to.
\begin{enumerate}
\item Both operations are commutative and associative.
\item The operations do not obey distributivity laws, i.e. for general sets or
Cnfs \(\g{g_1}, \g{g_2},\) and \(\g{g_3}\),
\begin{enumerate}
\item \(\g{g_1\wedge}(\g{g_2 \vee g_3}) \neq (\g{g_1\wedge g_2}) \g{\vee}
      (\g{g_1 \wedge g_3})\). In fact, we have \(\g{g_1\wedge}(\g{g_2 \vee
      g_3}) \subset (\g{g_1\wedge g_2}) \g{\vee} (\g{g_1 \wedge g_3})\). To
see that this is a proper subset, we can merely count the number of
Cnfs on both sides. The left side has \(n_1n_2n_3\) Cnfs, where
\(n_i\) is the cardinality of the set \(\g{g_i}\), while the right
side has \(n_1^2n_2n_3\) Cnfs.
\item \(\g{g_1\vee}(\g{g_2 \wedge g_3}) \neq (\g{g_1\vee g_2}) \g{\wedge}
      (\g{g_1 \vee g_3})\). In fact, we have \(\g{g_1\vee}(\g{g_2 \wedge
      g_3}) \subset (\g{g_1\vee g_2}) \g{\wedge} (\g{g_1 \vee g_3})\). To
see that this is a proper subset, we count \(n_1n_2n_3\) Cnfs on the
left side, while the right side has \(n_1^2n_2n_3\) Cnfs.
\end{enumerate}
\end{enumerate}
\label{remark-distributivity}
\end{remark}

\subsubsection{Graph disjunction does not always result in graphs}
\label{sec:org1981d72}
Essential to carrying out local graph rewriting (detailed in \S \ref{sec:org7839d0f}) is the calculation of graph disjunctions of the form
\(\g{g_i \vee h_i}\). A graph disjunction, even if we start with
\(\g{g_i}\) and \(\g{h_i}\) being graphs, frequently yields a set of Cnfs
that might not be a graph. For example,
\begin{align*}
  \g{a\vee a} &= \{x \vee y \mid (x, y : \Cnf)\in \g{a}\}\\
              &= \{a \vee a,\; a\vee \overline{a},\; \overline{a} \vee a,\; \overline{a} \vee\overline{a}\}\\
              &= \{a,\; \overline{a},\; \top\}
\end{align*}
We note that these three Cnfs do not belong to the set corresponding to any
single graph because the Cnfs do not have the same vertex set. The Cnfs
\(a\) and \(\overline{a}\) have a vertex set of \(\g{a}\); the true Cnf has
an empty vertex set. Nevertheless, we can write it as a union of graphs. We
can write \(\g{a\vee a} = \btop\cup\g{a}\).

Nevertheless, we cannot always write a disjunction even as a union of
graphs. For example,
\begin{align*}
  \g{a \vee} (\g{a\wedge b})
  &= \{a\vee(a\wedge b),\quad a\vee(a\wedge \overline{b}),\quad
       a\vee(\overline{a}\wedge b),\quad a\vee(\overline{a}\wedge \overline{b}),\\
  &\qquad \overline{a}\vee(a\wedge b),\quad \overline{a}\vee(a\wedge \overline{b}),\quad
       \overline{a}\vee(\overline{a}\wedge b),\quad \overline{a}\vee(\overline{a}\wedge \overline{b})
      \}\\
  &= \{a\wedge(a\vee b),\quad   a\wedge(a\vee \overline{b}),\quad
       a\vee b,\quad            a\vee\overline{b},\\
  &\qquad \overline{a}\vee b,\quad \overline{a}\vee\overline{b},\quad
       \overline{a}\wedge(\overline{a}\vee b),\quad \overline{a}\wedge(\overline{a}\vee \overline{b})
      \}\\
  &= \g{ab} \cup \{a\wedge(a\vee b),\quad   a\wedge(a\vee \overline{b}),\quad
       \overline{a}\wedge(\overline{a}\vee b),\quad \overline{a}\wedge(\overline{a}\vee \overline{b})
      \}
\end{align*}
Although the remaining Cnfs all belong to the same graph, namely
\(\g{a\wedge ab}\), we note that we do not have the complete set. For
example, we are missing the Cnf \(a\wedge(\overline{a}\vee b)\). Hence we
cannot write \(\g{a\vee}(\g{a\wedge b}) = \g{ab} \;\cup\; \g{a\wedge ab}\).
The best we can do is --- \[\g{ab} \subset \g{a\vee}(\g{a\wedge b}) \subset
\g{ab} \;\cup\; \g{a\wedge ab}.\]

These subset-superset pairs give us a lower and upper bound for the set in
the middle. The utility of these subset-superset pairs becomes apparent
when we view them in the context of satisfiability statuses. In the
following equation, let \(\g{s}\) an arbitrary graph. Then, we have ---
\[\bgamma(\g{s\wedge ab}) \leftarrow \bgamma(\g{s\wedge }
(\g{a\vee}(\g{a\wedge b}))) \leftarrow \bgamma(\g{s\wedge ab}) \wedge
\bgamma(\g{s\wedge a\wedge ab}).\]

This means that if \(\g{s\wedge ab}\) is unsatisfiable, then is
\(\g{s\wedge }(\g{a\vee}(\g{a\wedge b}))\) too. On the other hand, if
\(\g{s\wedge ab}\) is totally satisfiable, then by checking if \(\g{s\wedge
a\wedge ab}\) is also totally satisfiable, we can conclude that
\(\g{s\wedge }(\g{a\vee}(\g{a\wedge b}))\) is totally satisfiable as well.
This is a technique that we often use since the graph disjunction does not
always yield a union of graphs.

We include tables of such standard graph disjunctions in Appendix \ref{sec:org7678103}. These tables come in handy performing local rewrites,
as seen in \S \ref{sec:org60b361b}. The first two tables list graph
disjunctions that can be written exactly as a union of graphs; the third
table lists graph disjunctions that can only be listed as a subset-superset
pair.

\section{Local rewriting in graphs}
\label{sec:org7839d0f}
Graph rewriting concerns the technique of creating a new graph out of an
original graph algorithmically. Formally, a graph rewriting system consists
of a set of graph rewrite rules of the form \(\g{g_L}
\boldsymbol\rightsquigarrow \g{g_R}\). The idea is that to apply the rule
to a graph \(\g{g}\), we search for the presence of a subgraph \(\g{g_L}\)
of \(\g{g}\) and replace the subgraph with \(\g{g_R}\) while leaving the
rest unchanged. Such rewrite rules come in two forms ---

\begin{enumerate}
\item \textbf{Local rules} --- when a graph is rewritten at a particular vertex. In
this case, all edges not adjacent to the vertex remain unaffected by the
rewrite.
\item \textbf{Global rules} --- when a graph is rewritten by searching for specific
subgraphs that are isomorphic to \(\g{g_L}\).
\end{enumerate}

We will concern ourselves with local rules in this section, while global
rules will be handled by \S \ref{sec:org60b361b}.

We are looking for hypergraph analogue(s) of
edge-subdivisions/edge-smoothing in order to define the corresponding
notion of minimality in \(3\)\GraphSAT. This leads us naturally into local
rewriting, i.e. operations where change the graph at a single vertex and
its neighborhood. The idea is to find changes or rewrites that leave the
satisfiability of a graph unchanged. For example, edge-subdivision can be
though of as the rewrite \(\g{g\wedge ab} \rightsquigarrow \g{g\wedge
ac\wedge bc}\), while the inverse operation of edge smoothing can be
thought of as \(\g{g\wedge ac\wedge bc} \rightsquigarrow \g{ab}\). Using
local rewriting, we try to generalize this rule as well as the proof that
it leave the satisfiability of a graph unchanged.

For local rules, we focus on the extended notion of ``making assignments''
detailed in \S \ref{sec:org4e3aba8}. Just as for a Cnf \(x\), \(x[l]\)
denotes its an assignment at a literal \(l\), we write \(\g{g[v]}\) to
denote assignment at a vertex \(\g{v}\) of a graph \(\g{g}\).

Literal assignments on Cnfs have two key properties that make them useful
--
\begin{enumerate}
\item If \(x\) is a Cnf and \(l\) is a literal that is in the set of literals
of \(x\), then \(x[l]\) is guaranteed to be smaller than \(x\) -- it
either has fewer clauses, or it has the same number of clauses but with
those clauses having fewer literals in them.
\item \(x\) is satisfiable if and only if either one of \(x[v]\) or
\(x[\overline{v}]\) are satisfiable.
\end{enumerate}

Vertex assignments on graphs also have similar properties ---
\begin{enumerate}
\item If \(\g{g}\) is a graph and \(\g{v}\) is a vertex in the vertex set of
\(\g{g}\), then \(\g{g[v]}\) will always have Cnfs that are smaller than
the Cnfs in \(\g{g}\).
\item \(\g{g[v]}\) is totally satisfiable if and only if \(\g{g}\) is totally
satisfiable.
\end{enumerate}

This operation of vertex assignment is defined rigorously in \S \ref{sec:org4e3aba8}. The local graph rewriting theorem (Theorem \ref{thm-local_rewriting})
presents an alternate expression for computing \(\g{g[v]}\) that is easier
to write when performing calculations, easier to code when programming it
into a computer, and is a form that is used for proving several global
graph rewrite rules. The essence of this theorem is that even though
satisfiability (both boolean and graph) is a global problem i.e. it is
affected by the full structure of the Cnf, it can also be broken down into
a series of local assignments in the case of Cnfs and a series of local
rewrites in the case of graphs. Given a graph, we can decompose it at one
of its vertices of by computing this alternate expression in terms of the
link and rest of the graph without needing to step down to the level of
Cnfs. We close this section with a discussion on some consequences of this
theorem and an implementation of local rewriting in code (as a part of
\texttt{graphsat}).

\begin{remark}
We note that replacing a part of a CNF with an equisatisfiable part
breaks the equi-satisfiability of the whole CNF. For example, even though
\(a\vee b\) is equisatisfiable to \(c\vee d\), we cannot replace \((a\vee
  b) \wedge \overline{a} \wedge \overline{b}\) with \((c\vee d) \wedge
  \overline{a} \wedge \overline{b}\), since the former is unsatisfiable
while the latter is satisfiable.

However, in graphs, replacing some special subgraphs with equisatisfiable
pieces preserves equi-satisfiability. This is the basis for finding
rewrite rules using the local rewriting theorem (Theorem
\ref{thm-local_rewriting}.
\end{remark}

\subsection{Assignments on graphs}
\label{sec:org4e3aba8}
For a clause or Cnf, we defined assignments at a literal in \S \ref{sec:org5ecd0fd}.
We now define assignments for a graph at a vertex. We note that this is a
completely new notion that does not exists in graph theory and can be
defined here only because of the connection we have established between
graphs and Cnfs in the previous sections.

Let \(\g{v}\) be a vertex, and \(\g{g}\) be a graph. We define, \[\g{g[v]} :
\Set\Cnf = \left\{x[v]\vee x[\overline{v}]\;\mid (x:\Cnf)\in\g{g},\;(v :
\Literal) \in \g{v}\right\}\]

We note here that despite what the notation might suggest, \(x[v]\) is in
general not an element of \(\g{g[v]}\). If \(x\in\g{g}\), then \(x[v]\) is
in fact an element of the graph \(\link(\g{g}, \g{v})\g{\wedge}\rest(\g{g},
\g{v})\). The definitions of link and rest can be found in \S \ref{sec:orge9a5fbd}.

Next, we prove that assignments on graphs do not alter their satisfiability
status. This is useful because post-assignment the graphs always result in
sets with Cnfs having one fewer variable, while not altering their
satisfiability status.
\begin{lemma}
Let \(\g{g}\) be a graph and let \(\g{v}\) be a vertex. Then, \(\g{g[v]
\sim g}\).
\label{lemma-equisat-vertex-assign}
\end{lemma}
\begin{proof}
From the definition of \(\g{g[v]}\), we know that \(\g{g[v]} =
\left\{x[v]\vee x[\overline{v}]\;\mid (x:\Cnf)\in\g{g},\;(v : \Literal) \in
\g{v}\right\}\).

First, we note that each \(x[v]\) is equisatisfiable to \(x\wedge v\) and
each \(x[\overline{v}]\) is equisatisfiable to \(x\wedge\overline{v}\).
Thus, we can write \(x[v] \vee x[\overline{v}] \sim (x\wedge v) \vee
(x\wedge\overline{v}) = x\), using the fact that disjunction of
equisatisfiable Cnfs is equisatisfiable.

Next, we note that we can replace each Cnf of the set \(\g{g[v]}\) with an
equisatisfiable Cnf without affecting the satisfiability status of the set.
Thus, we write \(\g{g[v]} \sim \left\{x\;\mid (x:\Cnf)\in\g{g}\right\} =
\g{g}\).
\end{proof}

\subsection{Parts of a graph}
\label{sec:orge9a5fbd}
We now define some parts of graphs that will be useful for stating the
local rewriting theorem (Theorem \ref{thm-local_rewriting}). Let \(\g{g}\) be a
graph and let \(\g{v}\) be a vertex.

\begin{itemize}
\item The \emph{star} of \(\g{g}\) at a vertex \(\g{v}\) of \(\g{g}\) is the graph of
all edges containing \(\g{v}\).\\
\(\star : \Graph\times\Vertex \rightarrow \Graph, \text{ such that }
  (\g{(e_1)^{n_1}\cdots (e_k)^{n_k}},\;\g{v}) \mapsto \g{(e_1)^{m_1}\cdots
  (e_k)^{m_k}}\), where \(m_i = n_i\) if \(\g{v}\in\g{e_i}\) and \(m_i=0\)
otherwise, i.e. we omit the edge \(\g{e_i}\) otherwise.

\item The \emph{link} of \(\g{g}\) at \(\g{v}\) is the graph formed by removing
\(\g{v}\) from each edge of the star of \(\g{g}\) at \(\g{v}\). In the
following equation, \((-)\) denotes the usual set difference.\\
\(\link : \Graph\times\Vertex \rightarrow \Graph, \text{ such that }
  \left(\g{(e_1)^{n_1}\cdots (e_k)^{n_k}},\;\g{v}\right) \mapsto\)\\
\(\g{(e_1 - \{v\})^{m_1}\cdots (e_k - \{v\})^{m_k}}\), where \(m_i =
  n_i\) if \(\g{v}\in\g{e_i}\) and \(m_i=0\) otherwise. Any edges with
\(0\) multiplicity or with size \(0\) after deleting the vertex \(\g{v}\)
will simply be omitted from the resulting graph.

\item The \emph{rest} of \(\g{g}\) at \(\g{v}\) is the graph formed by all edges not
containing \(\g{v}\). In other words, these are the edges of \(\g{g}\)
not contained in \(\star(\g{g}, \g{v})\)\\
\(\rest : \Graph\times\Vertex \rightarrow \Graph, \text{ such that }
  (\g{(e_1)^{n_1}\cdots (e_k)^{n_k}},\;\g{v}) \mapsto \g{(e_1)^{m_1}\cdots
  (e_k)^{m_k}}\), where \(m_i = n_i\) if \(\g{v}\notin\g{e_i}\) and
\(m_i=0\) otherwise, i.e. we omit the edge \(\g{e_i}\) otherwise.
\end{itemize}

We note that when computing the \(2\)-partitions of a multiset (for example
a graph), we can split an edge with multiplicity across the partitions.
This observation will come in handy when using the local rewriting theorem
(Theorem \ref{thm-local_rewriting}). For example, we will consider
\(\lbrace\g{(e_1)(e_2)},\; \g{(e_2)^2}\rbrace\) to be a valid
\(2\)-partition of the graph \(\g{(e_1)(e_2)^3}\).

A graph \(\g{g}\) is a \emph{subgraph} of a graph \(\g{h}\) if every edge of
\(\g{g}\) (counting duplicates as distinct) is also an edge of \(\g{h}\).
We denote this partial order on graphs by \(\g{g \le h}\). An edge
\(\g{e}\) is a \emph{face} of an edge \(\g{f}\) if every vertex in \(\g{e}\) is
also in \(\g{f}\). A graph \(\g{g}\) is a \emph{shaved version} of a graph
\(\g{h}\) if \(\g{g}\) can be constructed by replacing each edge of
\(\g{h}\) (counting duplicates as distinct) by a nonempty face of itself.
We note that each graph is a shaved version of itself. We denote this
partial order on graphs by \(\g{g \ll h}\).

We now prove some lemmas outlining the relation between subgraphs, shaved
versions and satisfiability.

\begin{lemma}
Let \(\g{g_1}\) and \(\g{g_2}\) be graphs such that \(\g{g_1 \le
  g_2}\). Then, \(\g{g_1}\xRightarrow{A\bot}\g{g_2}\).
\label{lemma-abot_of_subgraph}
\end{lemma}
\begin{proof}
We can write \(\g{g_2} = \g{g_1\wedge g}\) for some \((\g{g} : \Graph)\). Let
\((x_1 : \Cnf)\in\g{g_1}\). Let \((x : \Cnf) \in \g{g}\) be an arbitrary
Cnf in \(\g{g}\). If \((a:\Assignment)\) is such that \(x_1[a]=\bot\), then
we have \((x_1\wedge x)[a] = x_1[a] \wedge x = \bot\).
\end{proof}

\begin{lemma}
Let \(\g{e}\) and \(\g{f}\) be edges such that \(\g{e}\) is a face of
\(\g{f}\). Then, \(\g{f}\xRightarrow{A\bot}\g{e}\).
\end{lemma}
\begin{proof}
Let \((c_f:\Clause)\in \g{f}\) be arbitrary. We can write \(c_f = c_e \vee
c\) for some \((c_e:\Clause)\in\g{e}\) and some
\((c:\Clause)\in\g{f\setminus e}\). Then, any assignment \(a\) that
falsifies \(c_f\) necessarily falsifies \(c_e\), hence proving the result.
\end{proof}

\begin{lemma}
Let \(\g{g_1}\) and \(\g{g_2}\) be graphs such that \(\g{g_1 \ll g_2}\).
Then, \(\g{g_2} \xRightarrow{A\bot}\g{g_1}\).
\label{lemma-shaved-version}
\end{lemma}
\begin{proof}
We argue that the result follows for each face from the previous lemma. We
can string these faces together to form the graphs while preserving
equi-implication under the \(A\bot\)-criterion using Proposition
\ref{item:prop3}.
\end{proof}

\subsection{The local rewriting theorem}
\label{sec:org3b2f34d}
\begin{thm}
For every graph \(\g{g}\) and every vertex \(\g{v}\) with degree at least
\(2\), we have the following equi-satisfiability relation.
\begin{equation}\label{eqn:graph_rewriting}
  \g{g} \;\bsim\hspace*{-1em}
        \bigcup_{\substack{\g{h_1\;h_2}\;:\;\Graph\\
                 \{\g{h_1},\; \g{h_2}\}\;
                 \in \text{ 2-partitions of } \link(\g{g},\g{v})}}
	    \hspace*{-3em}
        \left(\g{h_1\vee h_2}\right) \g{\wedge} \rest(\g{g}, \g{v}).
\end{equation}
For every vertex \(\g{v}\) of \(\g{g}\) with degree \(1\), we have
\(\g{g\sim}\rest(\g{g},\g{v})\).
\label{thm-local_rewriting}
\end{thm}
\begin{proof}
For vertices \(\g{v}\) of degree \(1\), the theorem states that we can
delete the single edge \(\g{e}\) incident on \(\g{v}\) without affecting
the satisfiability status of \(\g{g}\). This is true because any \((x :
\Cnf)\in \g{g}\) can be written as \(x = y \wedge (c \vee v)\) or \(x = y
\wedge (c \vee \overline{v})\), where \((y : \Cnf) \in \rest(\g{g},
\g{v})\), and \((c : \Clause)\) such that \(c\vee v\) is a clause supported
on the edge \(e\). Assigning \(v\) to true and false respectively leaves us
with the result \(x \sim y\), and thus \(\g{g\sim}\rest(\g{g}, \g{v})\).

We now consider the case where degree of \(\g{v}\) in \(\g{g}\) is at least
\(2\). Since the star and the rest together form a partition of the edge
set of a graph, we can write \(\g{g} = \star(\g{g}, \g{v})\g{\wedge
}\rest(\g{g}, \g{v})\). Since we know from Lemma
\ref{lemma-equisat-vertex-assign} that \(\g{g} \bsim \g{g[v]}\), we can
also write \(\g{g} \bsim \star(\g{g}, \g{v})\g{[v]\;\wedge }\rest(\g{g},
\g{v})\).

Since graph conjunction is defined as the pairwise conjunction of the
Cartesian product of underlying Cnfs, we can infer that graph disjunction
commutes with set union. Thus, it suffices to show that
\[\star(\g{g},\g{v})\g{[v]} \;=\hspace{-1em}
        \bigcup_{\substack{\g{h_1\;h_2}\;:\;\Graph\\
                 \{\g{h_1},\; \g{h_2}\}\;
                 \in \text{ 2-partitions of } \link(\g{g},\g{v})}}
	    \hspace*{-3em}
        \left(\g{h_1\vee h_2}\right).\]

Suppose \((x : \Cnf) \in \star(\g{g}, \g{v})\g{[v]}\). Then, there exists a
\((y : \Cnf) \in \star(\g{g}, \g{v})\) such that \(x = y[v] \vee
y[\overline{v}]\). Potentially, there are four types of clauses in \(y\)
--- those that contain the literal \(v\), those that contain
\(\overline{v}\), contain both, or contain neither. The last two cases are
not possible. A clause cannot contain both \(v\) and \(\overline{v}\)
because otherwise the edge corresponding to the clause will be incident on
the vertex \(\g{v}\) twice, which is something that the definition of an
edge does not allow. A case where a clause contains neither \(v\) nor
\(\overline{v}\) is impossible because the Cnf \(y\) belongs to the star of
\(\g{g}\) at \(\g{v}\).

We can therefore partition \(y\) into \(y_1\) containing clauses that
contain \(v\), and \(y_2\) containing clauses that contain
\(\overline{v}\), i.e.~we can write \(y=y_1\wedge y_2\). Thus, \(x =
y[v]\vee y[\overline{v}] = (\top \wedge y_2[v]) \vee
(y_1[\overline{v}]\wedge \top) = y_2[v] \vee y_1[\overline{v}]\). Let
\(\g{h_1}\) and \(\g{h_2}\) be the graphs that support
\(y_1[\overline{v}]\) and \(y_2[v]\) respectively. We have shown that \(x
\in \g{h_1\vee h_2}\). Furthermore, \(\g{h_1}\) and \(\g{h_2}\) form a
\(2\)-partition of \(\link(\g{g})\) since \(y_1\) and \(y_2\) form a
\(2\)-partition of the Cnf \(y\) in \(\star(\g{g}, \g{v})\). Generalizing
this to all the different Cnfs \(x\) in \(\star(\g{g}, \g{v})\g{[v]}\), we
have shown that \[\star(\g{g},\g{v})\g{[v]} \;\;\subseteq\hspace{-1em}
\bigcup_{\substack{\g{h_1\;h_2}\;:\;\Graph\\ \{\g{h_1},\; \g{h_2}\}\; \in
\text{ 2-partitions of } \link(\g{g},\g{v})}} \hspace*{-3em}
\left(\g{h_1\vee h_2}\right).\]

Conversely, suppose \((x : \Cnf) \in \g{h_1\vee h_2}\), such that
\(\{\g{h_1},\; \g{h_2}\}\) is some \(2\)-partition of \(\link(\g{g},
\g{v})\). Then, we can factor \(x\) as \(x=x_1\vee x_2\), for some Cnfs
\(x_1\) and \(x_2\) such that \(x_1\in\g{h_1}\) and \(x_2\in\g{h_2}\).
Consider then the Cnf \(y\) given by \(y=(x_1 \vee v)\wedge
(x_2\vee\overline{v})\). Firstly, we observe that \(x = y[v] \vee
y[\overline{v}]\). Secondly, we note that \(y \in \star(\g{g}, \g{v})\)
since the effect of disjuncting with \(v\) is to extend each clause in
\(x_1\) and \(x_2\) by a literal in \(\g{v}\). Thus, we can write \(x \in
\star(\g{g}, \g{v})\g{[v]}\). We have now shown that
\[\star(\g{g},\g{v})\g{[v]} \;\;\supseteq\hspace{-1em}
\bigcup_{\substack{\g{h_1\;h_2}\;:\;\Graph\\ \{\g{h_1},\; \g{h_2}\}\; \in
\text{ 2-partitions of } \link(\g{g},\g{v})}} \hspace*{-3em}
\left(\g{h_1\vee h_2}\right).\]
\end{proof}

An implementation of the right side of \eqref{eqn:graph_rewriting} in
Python using our \texttt{graphsat} package is in Appendix \ref{sec:org322849c}. Details of the \texttt{graphsat} package will be provided in a
separate publication. The source code of \texttt{graphsat} is available at
\cite{KaHi2021}.

Next, we will state two corollaries that can be turned into a test for
satisfiable graphs. These corollaries are turned into procedures for
testing the satisfiability status of a graph and are presented in detail
later in this section.

\begin{cor}
Let \(\g{g}\) be a graph and let \(\g{v}\) be a vertex of \(\g{g}\). If
\(\g{g}\) is unsatisfiable, then there exists a \(2\)-partition
\(\{\g{h_1}, \g{h_2}\}\) of \(\link(\g{g}, \g{v})\) such that both
\((\g{h_1\;\wedge\;}\rest(\g{g}, \g{v}))\) and
\((\g{h_2\;\wedge\;}\rest(\g{g}, \g{v}))\) are unsatisfiable.
\label{cor-exists_unsat_partition_of_unsat}
\end{cor}
\begin{proof}
If \(\g{g}\) is unsatisfiable, then so is the set \((\g{h_1\vee h_2})
\g{\wedge} \rest(\g{g}, \g{v})\) for some \(2\)-partition \(\{\g{h_1},\;
\g{h_2}\}\) of \(\link(\g{g}, \g{v})\). The result follows from Proposition
\ref{item:prop8}.
\end{proof}

\begin{cor}
Let \(\g{g}\) be a graph and let \(\g{v}\) be a vertex of \(\g{g}\). If
either one of \((\g{h_1\;\wedge\;}\rest(\g{g}, \g{v}))\) or
\((\g{h_2\;\wedge\;}\rest(\g{g}, \g{v}))\) is totally satisfiable for
every \(2\)-partition \(\{\g{h_1}, \g{h_2}\}\) of \(\link(\g{g}, \g{v})\),
then \(\g{g}\) itself is totally satisfiable.
\label{cor-satisfiable-of-all-sat-partitions}
\end{cor}
\begin{proof}
This is the contrapositive of Corollary
\ref{cor-exists_unsat_partition_of_unsat}.
\end{proof}

The following procedures for checking graph satisfiability status use
Corollary \ref{cor-satisfiable-of-all-sat-partitions}.

\begin{procedure}
A non-recursive procedure for checking that a graph \(\g{g}\) is totally
satisfiable is as follows ---
\begin{enumerate}
\item Pick a vertex \(\g{v}\) of \(\g{g}\) (the vertex of lowest degree will
make the procedure easier).
\item Pick a \(2\)-partition \(\{\g{h_1}, \g{h_2}\}\) of \(\link(\g{g},
   \g{v})\), compute the satisfiability status of
\((\g{h_1}\;\g{\wedge}\;\rest(\g{g}, \g{v}))\) and
\((\g{h_2}\;\g{\wedge}\;\rest(\g{g}, \g{v}))\). Satisfiability status of
each graph can be checked using a graph satchecker (see \S
\ref{sec:orga404456}).
\item If both are unsatisfiable then the result is \textsc{inconclusive} and
we can exit the procedure (this follows from Remark
\ref{remark-distributivity}). If either of these is totally satisfiable, then
move on the next \(2\)-partition and repeat Step 2.
\item If all \(2\)-partitions have been exhausted, then \(\g{g}\) is
\textsc{totally satisfiabile}.
\end{enumerate}
\label{procedure-non-rec}
\end{procedure}

Note that in this procedure, once a vertex has been picked, the graphs used
in Step 2 are both smaller than \(\g{g}\) and do not use the vertex
\(\g{v}\). Although the graphs in Step 2 are smaller, there are many more
graphs to check since we have to check all possible \(2\)-partitions.
Hence, a naive application of Procedure \ref{procedure-non-rec} does not
necessarily improve the efficiency of checking the satisfiability of a
graph. The real application of this procedure is to serve as a step in the
recursion outlined in the procedure below.

\begin{procedure}
This procedure is a recursive version of Procedure \ref{procedure-non-rec}. It
checks the satisfiability status of a graph \(\g{g}\).
\begin{enumerate}
\item If \(\g{g}\) is small, i.e. \(3\) vertices or fewer, then check it using
a graph satchecker (see \S \ref{sec:orga404456}). If
\(\g{g}\) is not small, we pick a vertex \(\g{v}\) of \(\g{g}\) (the
vertex of lowest degree will make the procedure easier).
\item Pick a \(2\)-partition \(\{\g{h_1}, \g{h_2}\}\) of \(\link(\g{g},
   \g{v})\). Let \(\g{g_1} = (\g{h_1}\;\g{\wedge}\;\rest(\g{g}, \g{v}))\) and
\(\g{g_2} = (\g{h_2}\;\g{\wedge}\;\rest(\g{g}, \g{v}))\). We can think of
\(\g{g}\) as being the \emph{parent graph} of its \emph{child graphs} \(\g{g_1}\)
and \(\g{g_2}\).
\item The parent graph is totally satisfiable if either one of its child
graphs is totally satisfiable (this follows from Corollary
\ref{cor-satisfiable-of-all-sat-partitions}). To check satisfiability-status
of each child graph \(\g{g_i}\), go back to Step 1. but with \(\g{g_i}\)
in place of \(\g{g}\) and a vertex \(\g{v_i}\) of \(\g{g_i}\) in place
of \(\g{v}\).
\item If both child graphs are unsatisfiable, then check the parent graph
using a graph satchecker (see \S \ref{sec:orga404456}).
\item If either child graph is totally satisfiable, then conclude that their
parent graph is satisfiable.
\item Once we have backtracked all the way back to the original graph
\(\g{g}\), we can exit the procedure with a result of
\textsc{totally satisfiable} or \textsc{unsatisfiable}.
\end{enumerate}
\label{procedure-rec}
\end{procedure}

This procedure has several advantages over Procedure \ref{procedure-non-rec}.
Firstly, the recursion creates smaller graphs, each of which can be checked
more quickly. Secondly, unlike Procedure \ref{procedure-non-rec}, Procedure
\ref{procedure-rec} never returns an \textsc{inconclusive} result. However, this
comes at the potential cost of having to backtrack all the way to the
original graph \(\g{g}\) and then having to sat-check the entire graph.
Thirdly, the algorithm can be parallelized and memoized (storing and using
the satisfiability statuses of graphs already seen when checking the
satisfiability of new graphs). This helps in mitigating the potential
slowdowns caused by having to now many graphs, since each recursion of
Procedure \ref{procedure-rec} adds exponentially more graphs that need to be
checked.

\begin{procedure}
Procedure for checking satisfiability status of a graph \(\g{g}\) using the
process of ``graph completion''.
\begin{enumerate}
\item Choose a vertex \(\g{v}\) of \(\g{g}\) (the vertex of lowest degree will
make the procedure easier).
\item Compute the set of Cnfs \[A \;=
   \bigcup_{\substack{\g{h_{1}\;h_{2}}\;:\;\Graph\\ \{\g{h_1},\;
   \g{h_2}\}\; \in \text{ 2-partitions of } \link(\g{g},\g{v})}}
   \hspace*{-3em} \left(\g{h_1\vee h_2}\right) \g{\wedge} \rest(\g{g},
   \g{v}).\]
\item Construct a set of graphs \(\{\g{g_i} \mid i\in I\}\) for some index set
\(I\) such that \(A \subseteq \bigcup_{i\in I}\g{g_i}\). One way to
construct this is to look at the image set of \(A\) under the map that
sends a Cnf to the graph that supports it. We call this process ``graph
completion''.
\item If every \(\g{g_i}\) is totally satisfiable, then \(\g{g}\) is
\textsc{totally satisfiable}. If not, then \(\g{g}\) is
\textsc{inconclusive}.
\item To check the satisfiability status of \(\g{g_i}\), we can recursively
call this procedure on each \(\g{g_i}\) in place of \(\g{g}\).
\end{enumerate}
\label{procedure-graph-completion}
\end{procedure}

We note that the run-time complexity of these procedures has not been
analyzed.

The trouble with Procedures \ref{procedure-non-rec} and
\ref{procedure-graph-completion} is that it can only prove the total
satisfiability of a graph. If the graph \(\g{g}\) is unsatisfiable, then
the procedure is inconclusive. It is possible to prove that a graph is
totally satisfiable using this procedure, but not that it is unsatisfiable.
This shortcoming can however be mitigated somewhat by our choice of vertex
\(\g{v}\). We will prefer using a vertex of low degree to as to ensure a
smaller size of \(\link(\g{g}, \g{v})\). Picking a vertex of degree \(d\)
results in a link of size \(d\). The number of nonempty partitions of the
link are \(2^{d-1}-1\). (We only care about nonempty partitions because the
empty partition terms do not affect satisfiability.) Since \(d\) appears in
the exponent of this count, we choose a vertex of lowest possible degree
\(d\).

\section{Graph reduction rules}
\label{sec:org60b361b}
This section states some global graph rewriting rules that leave the
satisfiability status of a graph unchanged. We call these \emph{graph reduction
rules}. Using the graph local rewriting theorem (Theorem
\ref{thm-local_rewriting}), we prove the invariance of satisfiability status of
graphs under these reduction rules. The equi-satisfiability results in all
calculations in this section follow directly from the theorem.

These reduction rules yield a set of simple search-and-replace rules that
can be used to simplify a graph, make it smaller, and then subject it to a
graph-satchecker. In the following subsections, we will always label the
rest of the graph \(\g{g}\) at vertex \(\g{1}\) as \(\g{s}\). Since \(\g{s}\)
has no edges incident on \(\g{1}\), when decomposing locally at that
vertex, we can always write \(\g{s[1]} = \g{s}\).

We start by proving a lemma that we use often in this section.
\begin{lemma}
If \(\g{g_1}\) and \(\g{g_2}\) are graphs such that \(\g{g_1}\) is a
subgraph of \(\g{g_2}\), then \((\g{s\wedge g_1}) \cup (\g{s\wedge g_2})
\bsim \g{s\wedge g_2}\).
\label{lemma-subgraph-drop}
\end{lemma}
\begin{proof}
From Lemma \ref{lemma-abot_of_subgraph}, we know that \(\g{g_1}
\xRightarrow{A\bot} \g{g_2}\). From Proposition
\ref{prop-conjunction-disjunction}.\ref{item:prop9}, the result follows.
\end{proof}

\subsection{Deleting leaf vertices}
\label{sec:org34ae191}
We show the effect of local rewriting at vertices with a single edge
incident on them (not counting edge multiplicities), i.e. at \emph{leaf
vertices}. Applying local graph rewriting (Theorem \ref{thm-local_rewriting}) to
leaf vertices incident on triangles of varying multiplicities we get ---
\begin{enumerate}
\item \(\g{s\wedge 123^1 \sim s}\).
\item \(\g{s\wedge 123^2 \sim s\wedge }(\g{23\vee23}) = \g{s\wedge }(\btop \cup
   \g{23}) = \g{s\wedge 23}\).
\item \(\g{s\wedge 123^3 \sim s\wedge }(\g{23\vee23^2}) = \g{s\wedge }(\g{23}
   \cup \btop) = \g{s\wedge 23}\).
\item \(\g{s\wedge 123^4 \sim s\wedge }(\g{23\vee23^3}) \cup \g{s\wedge
   }(\g{23^2\vee23^2}) = \g{s\wedge }(\g{23}\cup\btop) \cup \g{s\wedge
   }(\g{23^2} \cup \g{23}\cup \btop) = (\g{s\wedge 23^2}) \cup (\g{s\wedge
   23}) \cup \g{s} \bsim \g{s\wedge 23^2}\). This last equi-satisfiability
follows from Lemma \ref{lemma-subgraph-drop}.
\end{enumerate}

All the graph disjunctions used in 1. 2. and 3. above can be found in
Tables \ref{table-graph_disjunctions_size_2} and \ref{table-graph_disjunctions_size_3}.
The disjunctions in 4. are computed using the \texttt{operations.graph\_or}
function from our \texttt{graphsat} package. Code-snippets and their outputs are
provided below for reference but these can also be checked by hand
following the theory explained in \S \ref{sec:org1981d72}.

\vspace*{2ex}

\begin{minted}[frame=lines,label= (python3.9) (scratch) <<calculation1>>]{python}
import cnf, mhgraph
from operations import graph_or

g1 = mhgraph.mhgraph([[2, 3]])
g3 = mhgraph.mhgraph([[2, 3]]*3)

for x in graph_or(g1, g3):
    print(x)
\end{minted}
\footnotesize
\color{darkgray}
\uline{Output \(\g{23\vee23^3}\)}: \texttt{(<Bool: TRUE>)  (2,3)  (2,-3)  (-2,3)  (-2,-3)}
\normalsize
\color{black}

\begin{minted}[frame=lines,label= (python3.9) (scratch) <<calculation2>>]{python}
import cnf, mhgraph
from operations import graph_or

g2 = mhgraph.mhgraph([[2, 3]]*2)

for x in graph_or(g2, g2):
    print(x)
\end{minted}
\footnotesize
\color{darkgray}
\uline{Output \(\g{23^2\vee23^2}\)}: \texttt{(<Bool: TRUE>)  (2,3)  (2,-3)  (-2,3)  (-2,-3)  (2,3)(2,-3)  (2,3)(-2,3)  (2,3)(-2,-3)}\\
\texttt{(2,-3)(-2,3)  (2,-3)(-2,-3)  (-2,3)(-2,-3)}
\normalsize
\color{black}

We generalize all these calculations in Proposition \ref{prop-generalized-leaf}.
We first state the following observation without proof.
\begin{remark}
Let \(k \in \mathbb{N}\), and \(0 < k\). For every \(m, n \in
  \mathbb{N}\) such that \(0 < m < 2^k\) and \(0 < n < 2^k\), we have
\[\g{(1\cdots k)^m \vee (1\cdots k)^n} \;=\; \bigcup_{i=m+n-2^k}^{\min(m,
  n)} \g{(1\cdots k)^i},\] where it is understood that we will omit writing
terms \(\g{(1\cdots k)^i}\) for \(i < 0\) and we will write \(\btop\) in
place of \(\g{(1\cdots k)^0}\).
\label{remark-disjunction-same-edge}
\end{remark}

\begin{prop}
Let \(\g{s}\) be a graph not incident on the vertex \(\g{1}\). Let \(k
\in\mathbb{N}\), such that \(k \ge 1\). Then, for \(n\in\mathbb{N}\), we have ---
\begin{itemize}
\item \(\g{s\wedge (12\cdots k) \sim s}\)
\item \(\g{s\wedge (12\cdots k)^n \sim s\wedge (2\cdots k)^{\lfloor
  n/2\rfloor}}\), if \(1 < n < 2^k\).
\item \(\g{s\wedge (12\cdots k)^n \sim\bot}\), if \(n \ge 2^k\).
\end{itemize}
\label{prop-generalized-leaf}
\end{prop}
\begin{proof}
If \(n=1\), then by the result follows from the local rewriting theorem
(Theorem \ref{thm-local_rewriting}). If \(n > 1\), then we can compute
\(\link(\g{s\wedge (12\cdots k)^n},\; \g{1}) = \g{(2\cdots k)^n}\). Then,
every nonempty \(2\)-partition of the link is of the form \(\{\g{h_1},
  \g{h_2}\}\), where \(\g{h_1} = \g{(2\cdots k)^m}\) and \(\g{h_2} =
  \g{(2\cdots k)^{n-m}}\), for some \(1\le m \le n-1\). Using Remark
\ref{remark-disjunction-same-edge}, we infer that
\[\g{h_1\vee h_2} \;=\; \g{(2\cdots k)^m \vee (2\cdots k)^{n-m}} \;=\;
  \bigcup_{i=n-2^{k-1}}^{\min\{m,\; n-m\}} \g{(2\cdots k)^i}.\]

Thus, from Theorem \ref{thm-local_rewriting} we have
\[\g{s\wedge (12\cdots k)^n} \bsim \bigcup_{m=1}^{n-1}
  \bigcup_{i=n-2^{k-1}}^{\min\{m,\; n-m\}} \g{s\wedge (2\cdots k)^i} =
  \bigcup_{i=0}^{\lfloor n/2\rfloor} \g{s\wedge (2\cdots k)^i}.\]

Using Lemma \ref{lemma-subgraph-drop}, we can conclude that \(\g{s\wedge
  (12\cdots k)^n \sim s\wedge(2\cdots k)^{\lfloor n/2 \rfloor}}\).
\end{proof}

\subsection{Smoothing edges}
\label{sec:orgf7d5a42}
Having dealt with leaf vertices (i.e. vertices of degree \(1\)), we now
consider vertices of degree \(2\). Edges and hyperedges incident at such
vertices can be ``smoothed'' without affecting the satisfiability status of a
graph. We call these operations smoothing because each operation results in
graphs with fewer vertices.

We can smooth out the intersection of two simple edges as \(\g{s\wedge
12\wedge 13} \bsim \g{s\wedge }(\g{2 \vee 3}) = \g{s\wedge 23}\).
Similarly, we can smooth out the intersection of two triangles sharing a
common edge as \(\g{s\wedge 123\wedge 124} \bsim \g{s\wedge }(\g{23 \vee
24}) = \g{s\wedge }(\btop \cup \g{234}) \bsim (\g{s\wedge 234})\). This
last equi-satisfiability is obtained from Lemma \ref{lemma-subgraph-drop}.

Smoothing out the intersection of two triangles sharing a common vertex
results in a size \(4\) hyperedge --- \(\g{s\wedge 123\wedge 145} \bsim
\g{s\wedge }(\g{23 \vee 45}) = \g{s\wedge 2345}\). Smoothing at an
edge-triangle pair with a common vertex yields \(\g{s\wedge 12\wedge 134}
\bsim \g{s\wedge }(\g{2\vee 34}) = \g{s\wedge 234}\).

We state without proof a general pattern for smoothing of hyperedges
incident at a common vertex. For \(k_1, k_2 \in \mathbb{N}\), we have ---
\[\g{s\wedge (12\cdots k_1) \wedge(k_1+1 \cdots k_2)} \;\bsim\; \g{s\wedge (12\cdots k_2)}\]

\subsection{Tucking edges}
\label{sec:org4bd3a19}
We now prove a series of reduction rules that allow deletion of degree
\(2\) or higher vertices, resulting in graphs with fewer edges. Visually,
these operations look like tucking-in of an extended fin of the graph.

Tucking-in at an edge-hyperedge intersection incident at a common edge
yields \(\g{s\wedge 12\wedge 123} \bsim \g{s\wedge }(\g{2 \vee 23}) =
\g{s\wedge}(\g{23} \cup \btop) = (\g{s\wedge 23}) \cup \g{s} \bsim
\g{s\wedge 23}\). This last equi-satisfiability follows from Lemma
\ref{lemma-subgraph-drop}.

Degree \(3\) intersections have reduction rules for the following cases ---
\begin{enumerate}
\item A hyperedge with an edge incident on two of its three sides, i.e.
\(\g{s\wedge 12\wedge 13\wedge 123}\). \label{item:deg3_case_1}
\item A hyperedge of multiplicity \(2\), with an edge incident on one of its
sides, i.e. \(\g{s\wedge 12\wedge 123^2}\). \label{item:deg3_case_2}
\item A hyperedge with an edge of multiplicity two incident on one of its
sides, i.e. \(\g{s\wedge 12^2\wedge 123}\). \label{item:deg3_case_3}
\end{enumerate}

For instance \ref{item:deg3_case_1}, we get ---
\vspace*{-2ex}
\begin{flalign*}
\g{s\wedge 12\wedge 13\wedge 123} &\bsim \g{s\wedge }((\g{2\wedge 23}) \g{\vee 3})
        \;\cup\; \g{s\wedge }((\g{3\wedge 23}) \g{\vee 2})
        \;\cup\; \g{s\wedge }((\g{2\wedge 3}) \g{\vee 23}) &\\
      &= \g{s\wedge }(\g{3} \;\cup\; \g{23}) \;\cup\; \g{s\wedge }(\g{2}
        \;\cup\; \g{23}) \;\cup\; \g{s\wedge }(\g{23} \;\cup\; \btop) &\\
      &= \g{s} \;\cup\; (\g{s\wedge 2}) \;\cup\; (\g{s\wedge 3}) \;\cup\; (\g{s\wedge 23}) &\\
      &\bsim (\g{s\wedge 2}) \;\cup\; (\g{s\wedge 3}) \;\cup\; (\g{s\wedge 23}) &\\
      &\bsim (\g{s\wedge 2}) \;\cup\; (\g{s\wedge 3})&
\end{flalign*}

For instance \ref{item:deg3_case_2}, we get ---
\vspace*{-2ex}
\begin{flalign*}
\g{s\wedge 12\wedge 123^2}
   &\bsim \g{s\wedge }(\g{2 \vee 23^2}) \;\cup\; \g{s\wedge }(\g{23 \vee} (\g{2\wedge 23})) &\\
   &= \g{s\wedge }(\btop \;\cup\; \g{2} \;\cup\; \g{23}) \;\cup\; \g{s\wedge }(\btop \;\cup\; \g{23}) &\\
   &= \g{s} \;\cup\; (\g{s\wedge 2}) \;\cup\; (\g{s\wedge 23}) &\\
   &\bsim (\g{s\wedge 2}) \;\cup\; (\g{s\wedge 23}) &\\
   &\bsim \g{s\wedge 2}
\end{flalign*}

For instance \ref{item:deg3_case_3}, we get \(\g{s\wedge 12^2\wedge 123 \bsim
s\wedge }(\g{2 \vee} (\g{2\wedge 23})) \cup \g{s\wedge }(\g{23 \vee
2^2})\). From Table \ref{table-graph-disjunctions-incomplete}, we have
\[\btop\cup \g{2} \cup \g{23} \;\subset\; \g{2 \vee} (\g{2\wedge 23})
\;\subset\; \btop\cup \g{2} \cup \g{23} \cup (\g{2\wedge 23}).\] Thus, we
can write ---
\providecommand{\centralgraph}{\g{s\wedge 12^2\wedge 123}}
\providecommand{\s}{\g{s\wedge }}
\providecommand{\sep}{\;\subset\;}
\providecommand{\larr}{\;\leftarrow\;}
\providecommand{\bg}{\bgamma}
\[\s(\btop \cup \g{2} \cup \g{23}) \;\cup\; (\s\g{23}) \sep \centralgraph
\sep \s(\btop\cup \g{2} \cup \g{23} \cup (\g{2\wedge 23})) \cup
(\s\g{23})\]

Applying the \(\bg\) graph-satisfiability map, yields --- \[\bg(\s\btop)
\wedge \bg(\s\g{2}) \wedge \bg(\s\g{23}) \larr \bg(\centralgraph) \larr
\bg(\s\btop) \wedge \bg(\s\g{2}) \wedge \bg(\s\g{23}) \wedge
\bg(\s\g{2\wedge 23})\]

This can in turn be simplified (using the properties of \(\bg\) detailed in
\S \ref{sec:orgcbfe531}, the Lemmas \ref{lemma-abot_of_subgraph} and
\ref{lemma-shaved-version}, and the techniques outlined in Sec \ref{sec:org1981d72}) --- \[\bg(\s\g{2}) \larr
\bg(\centralgraph) \larr \bg(\s\g{2\wedge 23})\] This last result can be
seen as a partial rewrite rule owing to the presence of the subset-superset
pair.

\subsection{Opening a triple-intersection vertex}
\label{sec:org6227f90}
We can replace a three-hyperedge intersection incident on a common vertex
with three simple edges on the boundary. The proof of this reduction rule
is as follows ---
\begin{align*}
\g{s\wedge 123\wedge 124\wedge 134}\;
  &\bsim\; \g{s\wedge }(\g{23 \vee} (\g{24\wedge 34}))
     \;\cup\; \g{s\wedge }(\g{24 \vee} (\g{23\wedge 34}))
     \;\cup\; \g{s\wedge }(\g{34 \vee} (\g{23\wedge 24})) \\
  &=\; \g{s\wedge }(\btop \;\cup\; \g{23} \;\cup\; \g{234})
     \;\cup\; \g{s\wedge }(\btop \;\cup\; \g{24} \;\cup\; \g{234})
     \;\cup\; \g{s\wedge }(\btop \;\cup\; \g{34} \;\cup\; \g{234}) \\
  &=\; \g{s} \;\cup\; (\g{s\wedge 23}) \;\cup\; (\g{s\wedge 24}) \;\cup\; (\g{s\wedge 34})
     \;\cup\; (\g{s\wedge 234}) \\
  &\bsim\; (\g{s\wedge 23}) \;\cup\; (\g{s\wedge 24}) \;\cup\; (\g{s\wedge 34})
\end{align*}

\section{Satisfiability of mixed hypergraphs}
\label{sec:orga404456}
In this section we provide a list of known totally satisfiable and
unsatisfiable mixed-hypergraphs i.e. hypergraphs which have edges of size
\(1\), \(2\) or \(3\). List of candidate hypergraphs are generated
programmatically using SageMath's \texttt{nauty} package \cite{McKayPiPerno2014},
which has various tools for generating canonically-labeled, non-isomorphic
graphs with certain properties.

We list below several methods by which the satisfiability of a graph may be
checked using our \texttt{graphsat} package ---
\begin{enumerate}
\item If a graph is finite and small (fewer than 6 hyperedges), we can
sat-check it by passing it to a graph satchecker. Such a satchecker can
be found in the \texttt{mhgraph\_pysat\_satcheck} function in the \texttt{sat.py}
module.
\item If a graph is finite but not too small (7 to 20 hyperedges), we can
decompose it using local rewriting at the min-degree vertex. For this,
we use the \texttt{decompose} function found in the \texttt{graph\_rewrite.py} module.
\item If a graph is finite and big (20+ hyperedges), then we can reduce it to
a smaller graph by passing it to the \texttt{make\_tree} function in the
\texttt{operations.py} module.
\item Lastly, if a graph is infinite, we have to work out its satisfiability
status manually using reduction rules. Reduction rules themselves can be
generated by the \texttt{local\_rewrite} function in the \texttt{graph\_rewrite.py}
module.
\end{enumerate}

Figure \ref{fig-unsat-graphs} shows a selection of unsatisfiable hypergraphs up to
vertex size \(5\). A larger list can be found in the Appendix \ref{sec:org9d7b11d}.

\begin{figure}[h!]
\centering

\begin{subfigure}{0.17\textwidth}
\begin{center}
\includegraphics[trim={17cm 21cm 23cm 19cm}, clip,width=2.2cm]{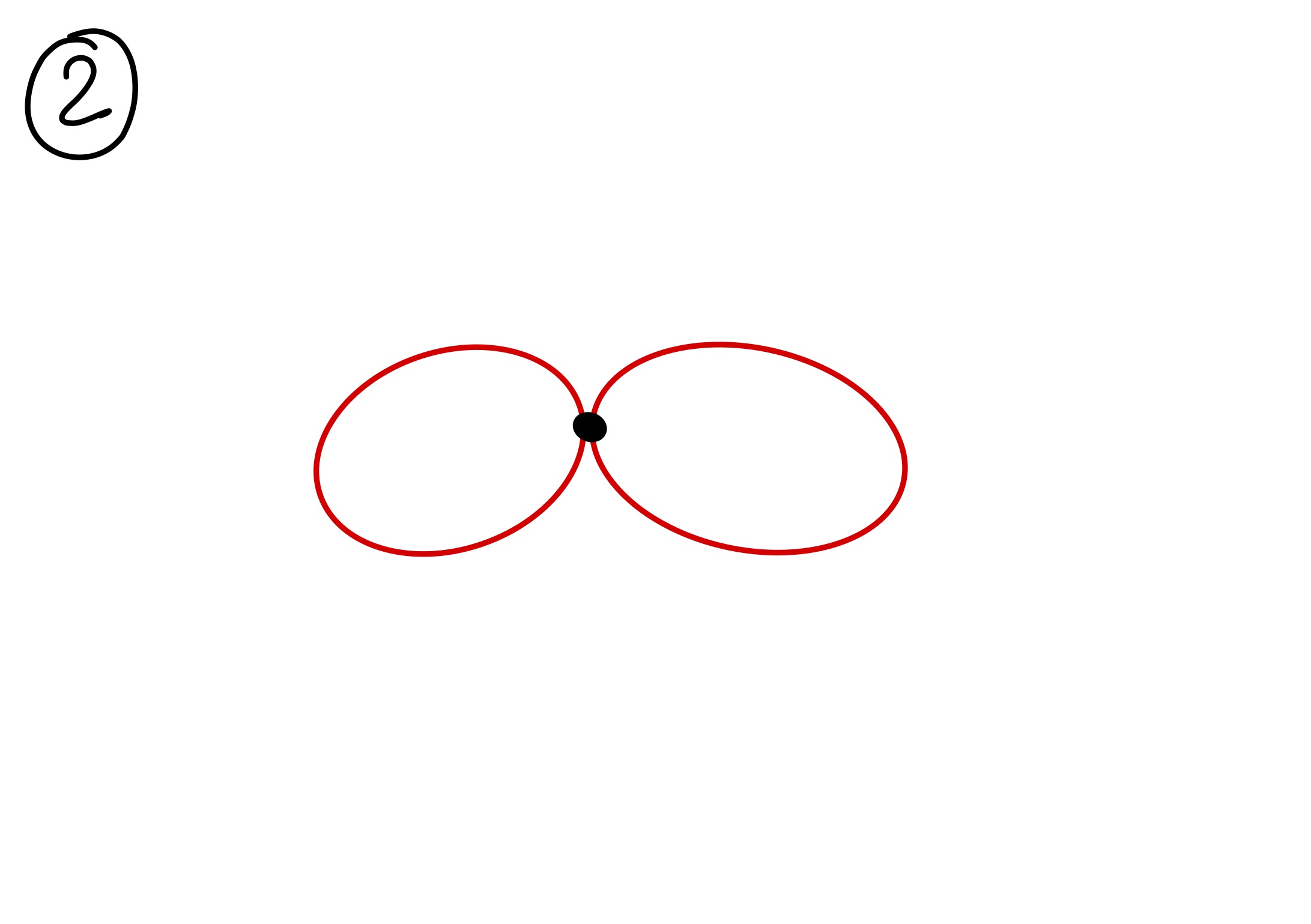}
\end{center}
\end{subfigure}
\begin{subfigure}{0.17\textwidth}
\begin{center}
\includegraphics[trim={17cm 28cm 23cm 19cm}, clip,width=2.2cm]{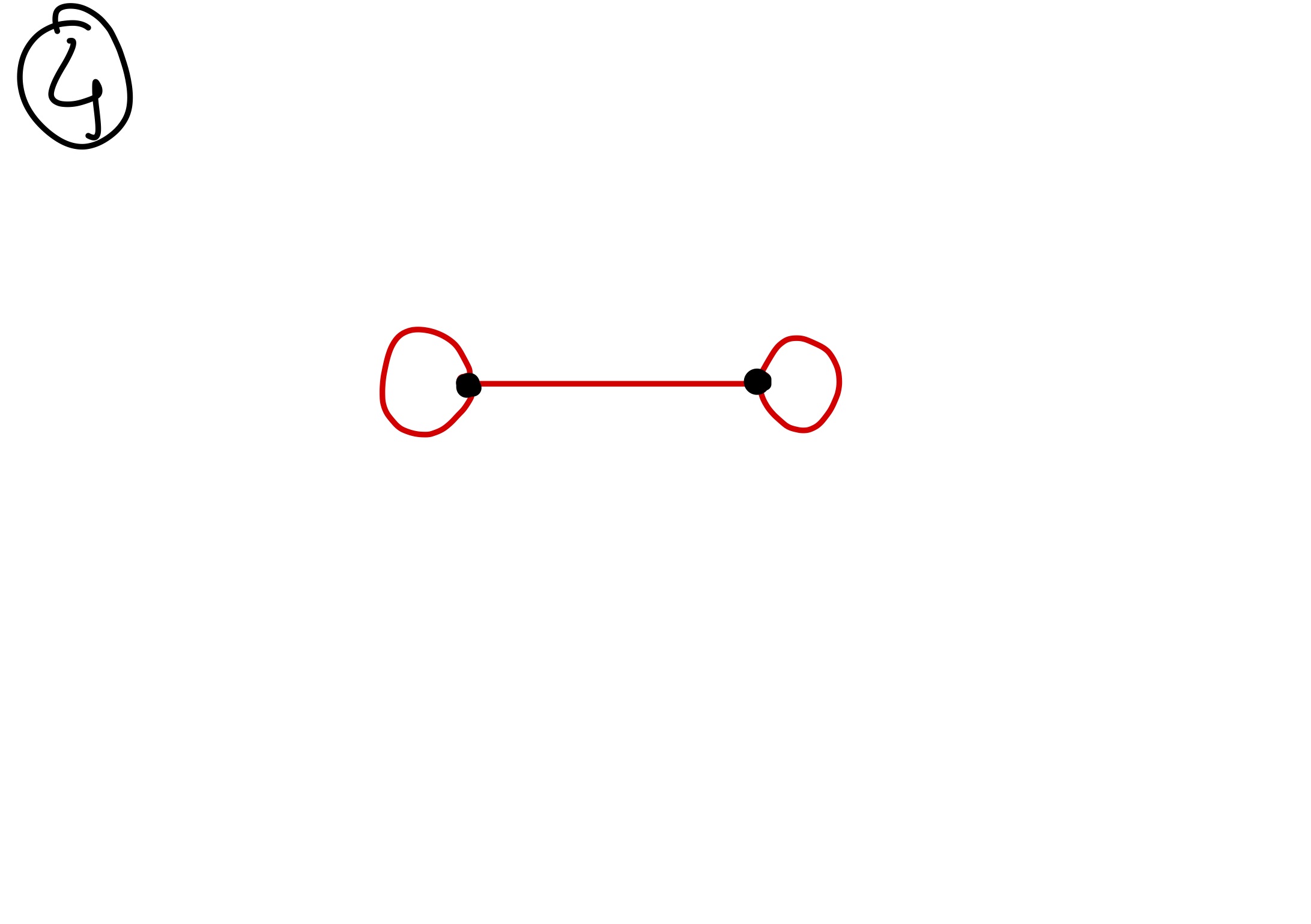}
\end{center}
\end{subfigure}
\begin{subfigure}{0.17\textwidth}
\begin{center}
\includegraphics[trim={17cm 23cm 28cm 9cm}, clip,width=2.2cm]{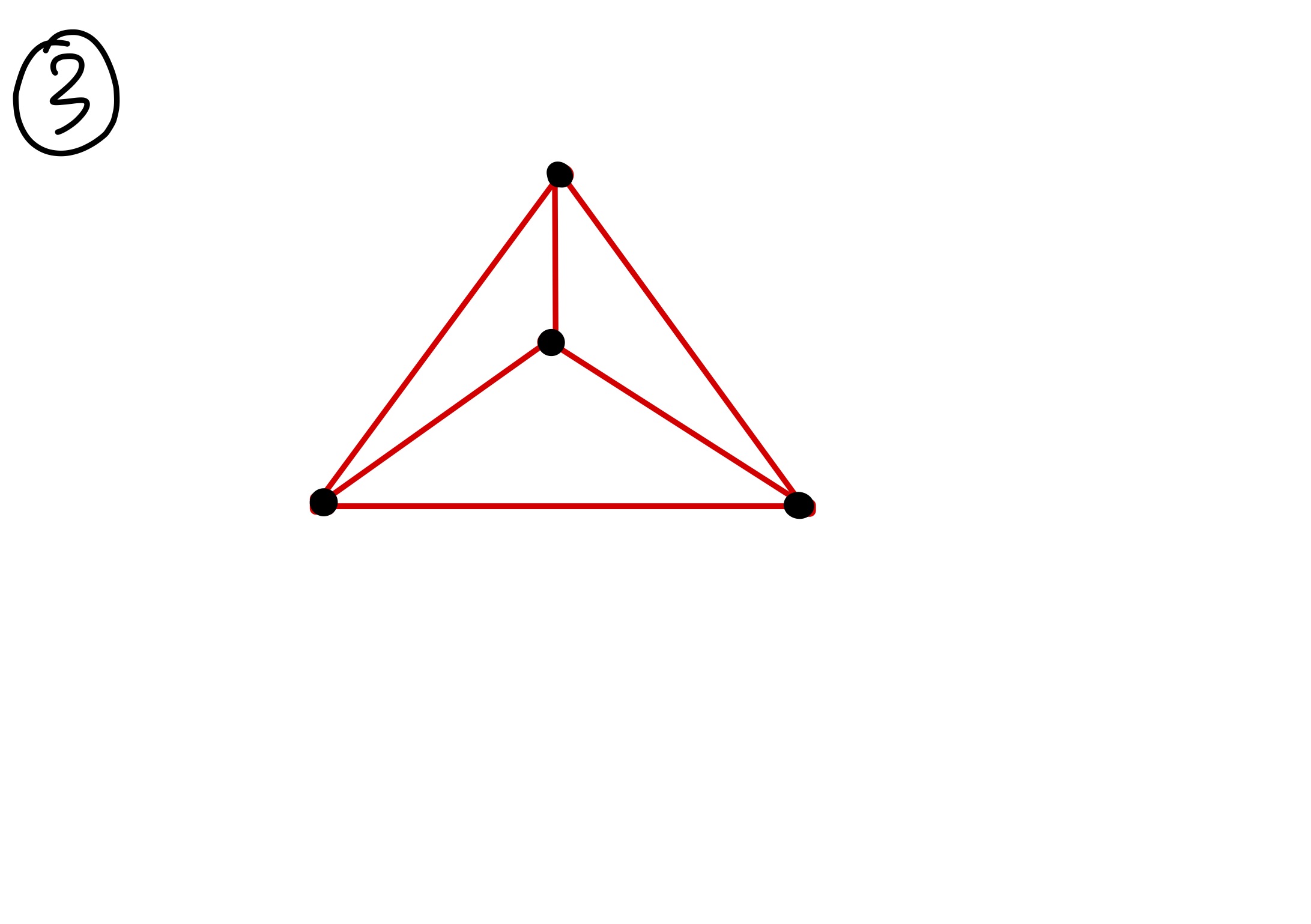}
\end{center}
\end{subfigure}
\begin{subfigure}{0.17\textwidth}
\begin{center}
\includegraphics[trim={5cm 17cm 8cm 2cm}, clip,width=2.2cm]{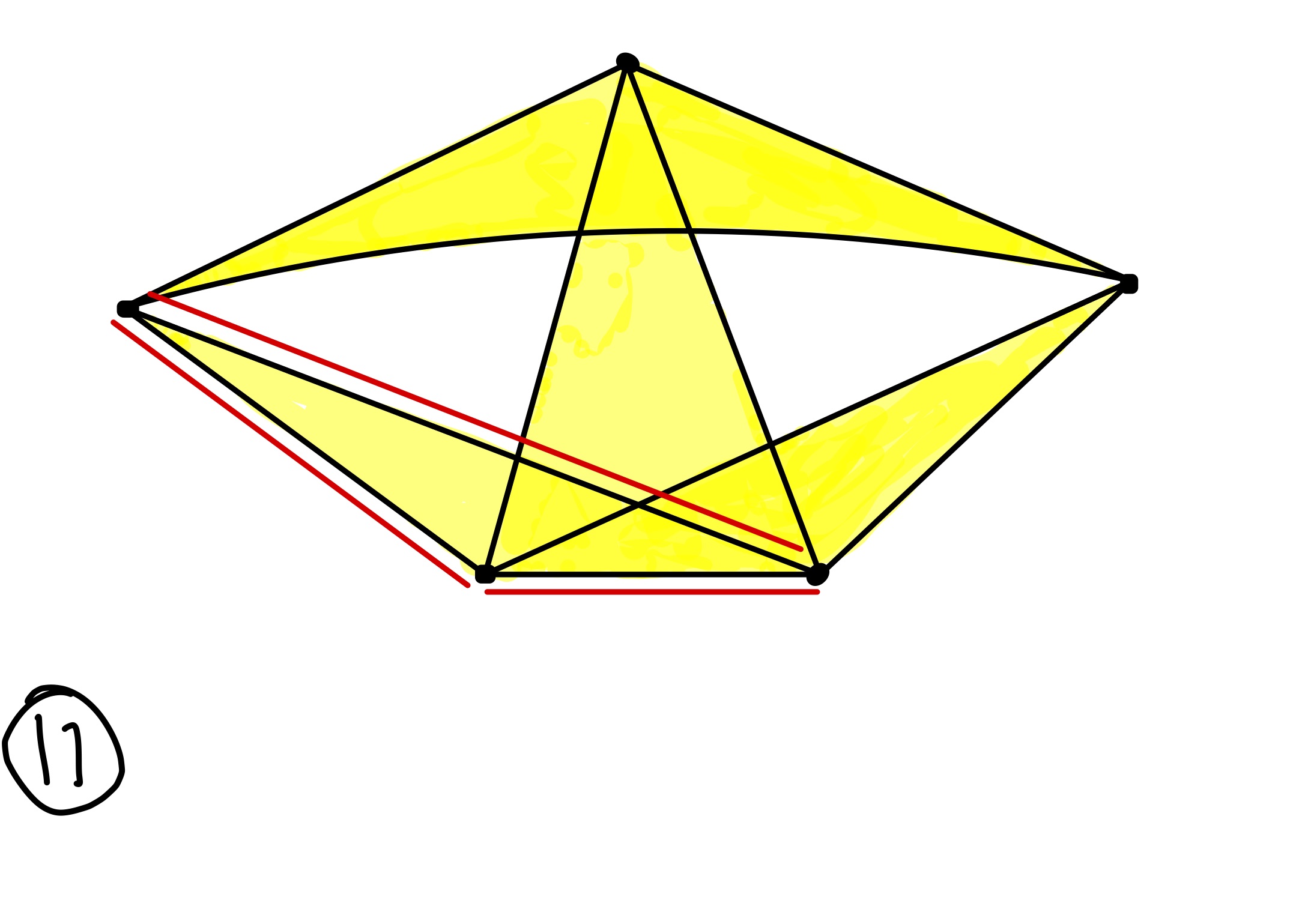}
\end{center}
\end{subfigure}
\begin{subfigure}{0.17\textwidth}
\begin{center}
\includegraphics[trim={7cm 9cm 24cm 13cm}, clip,width=2.2cm]{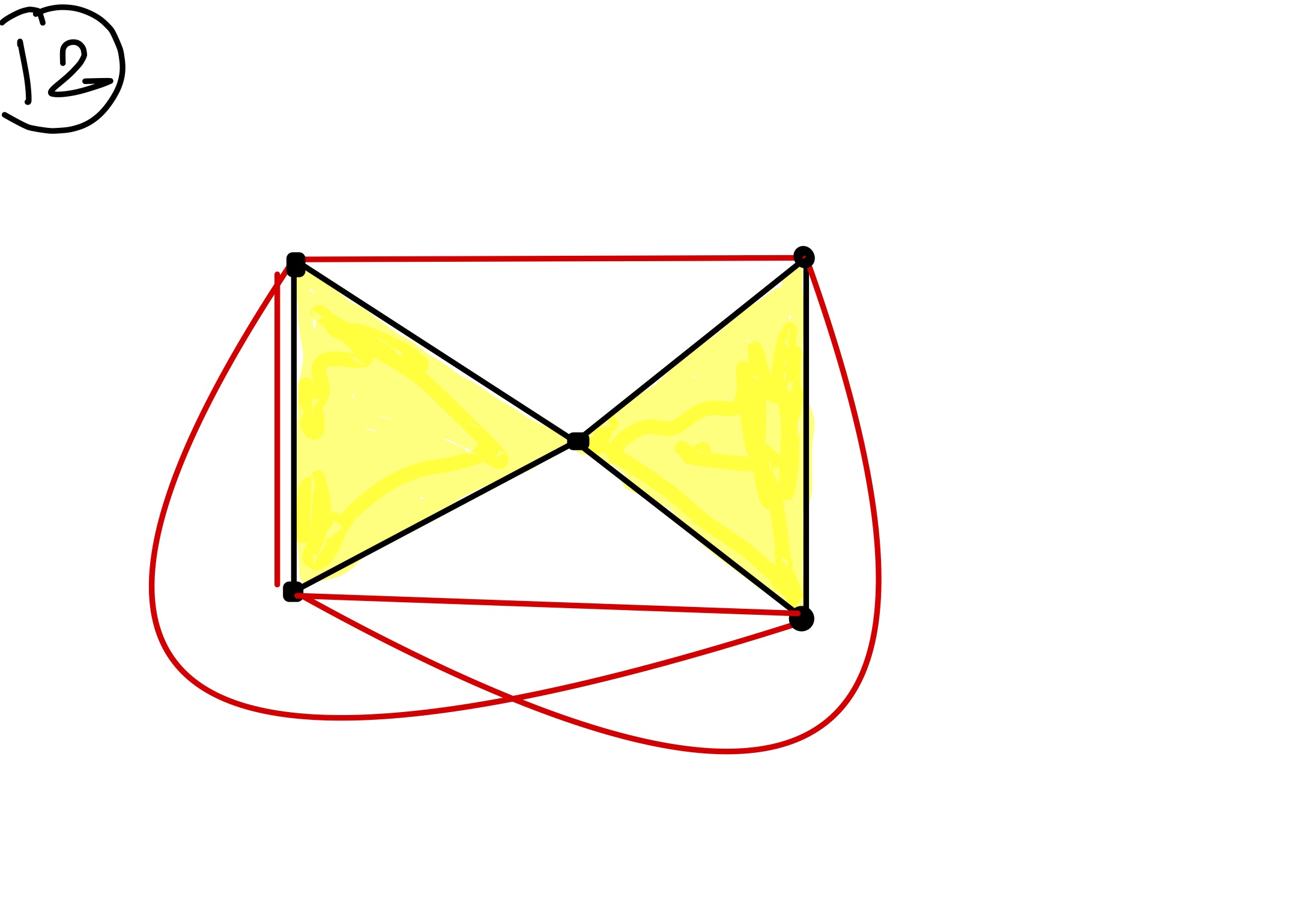}
\end{center}
\end{subfigure}

\begin{subfigure}{0.17\textwidth}
\begin{center}
\includegraphics[trim={23cm 28cm 23cm 18cm}, clip,width=2.2cm]{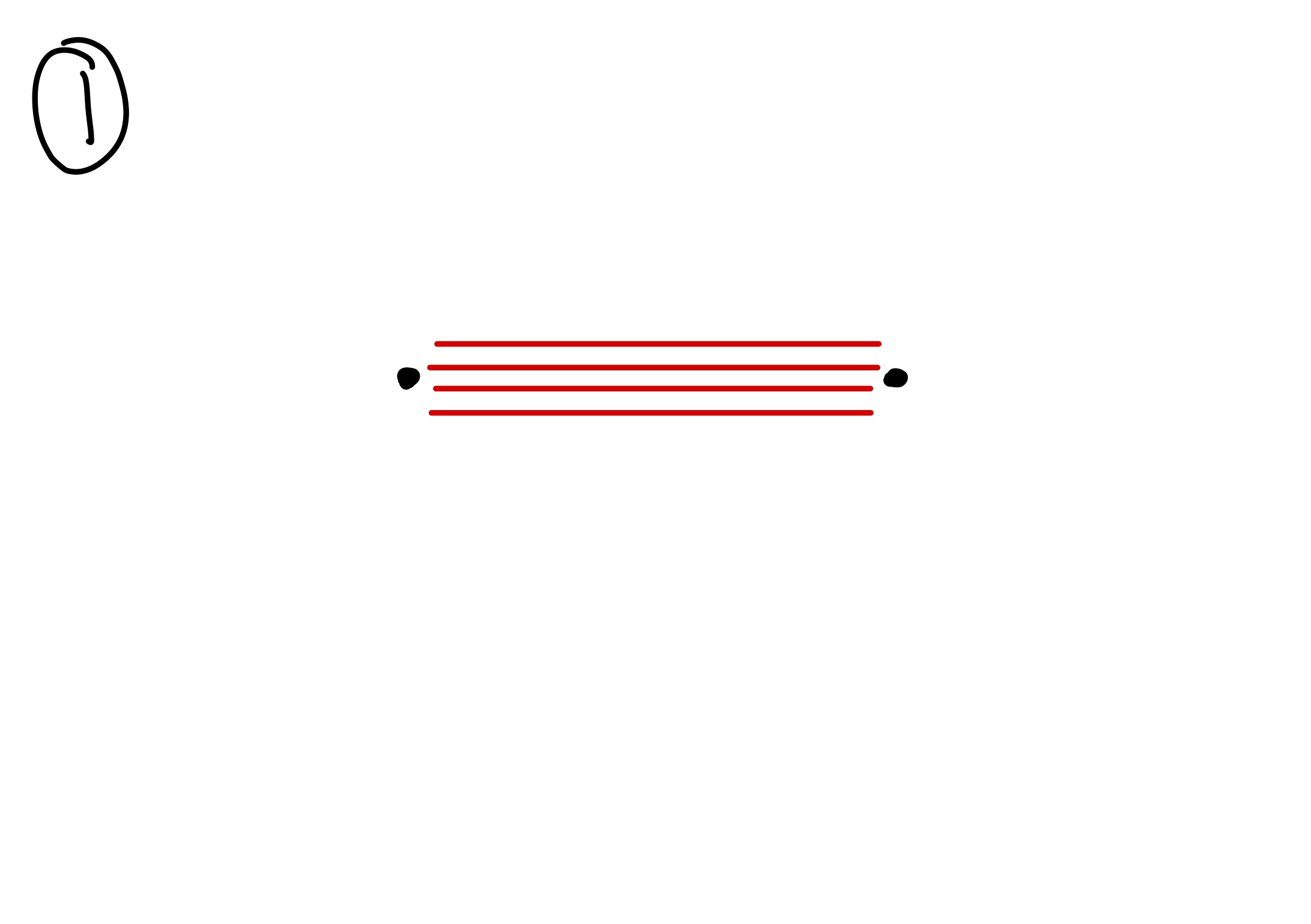}
\end{center}
\end{subfigure}
\begin{subfigure}{0.17\textwidth}
\begin{center}
\includegraphics[trim={13cm 9cm 16cm 9cm}, clip,width=2.2cm]{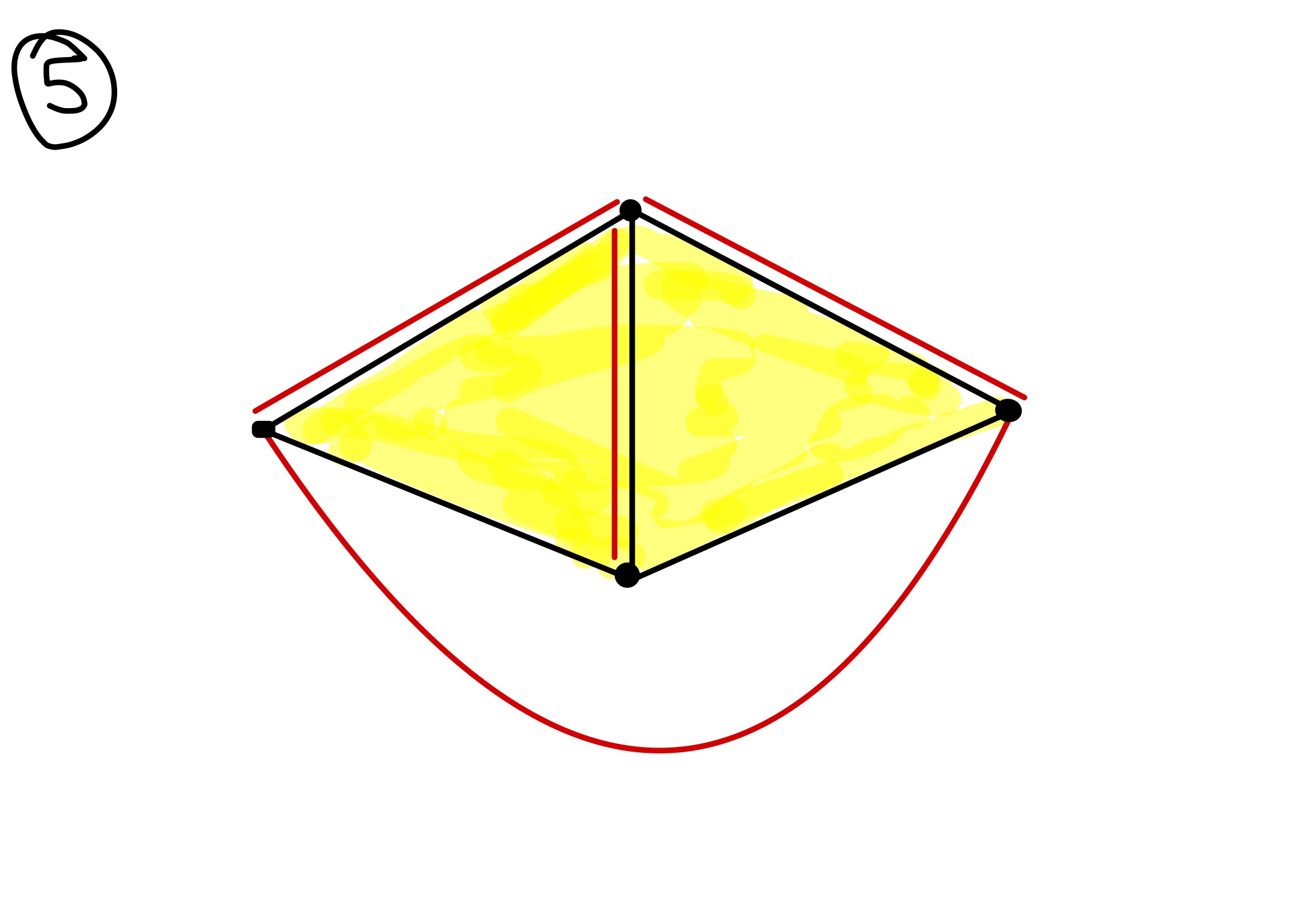}
\end{center}
\end{subfigure}
\begin{subfigure}{0.17\textwidth}
\begin{center}
\includegraphics[trim={15cm 9cm 15cm 9cm}, clip,width=2.2cm]{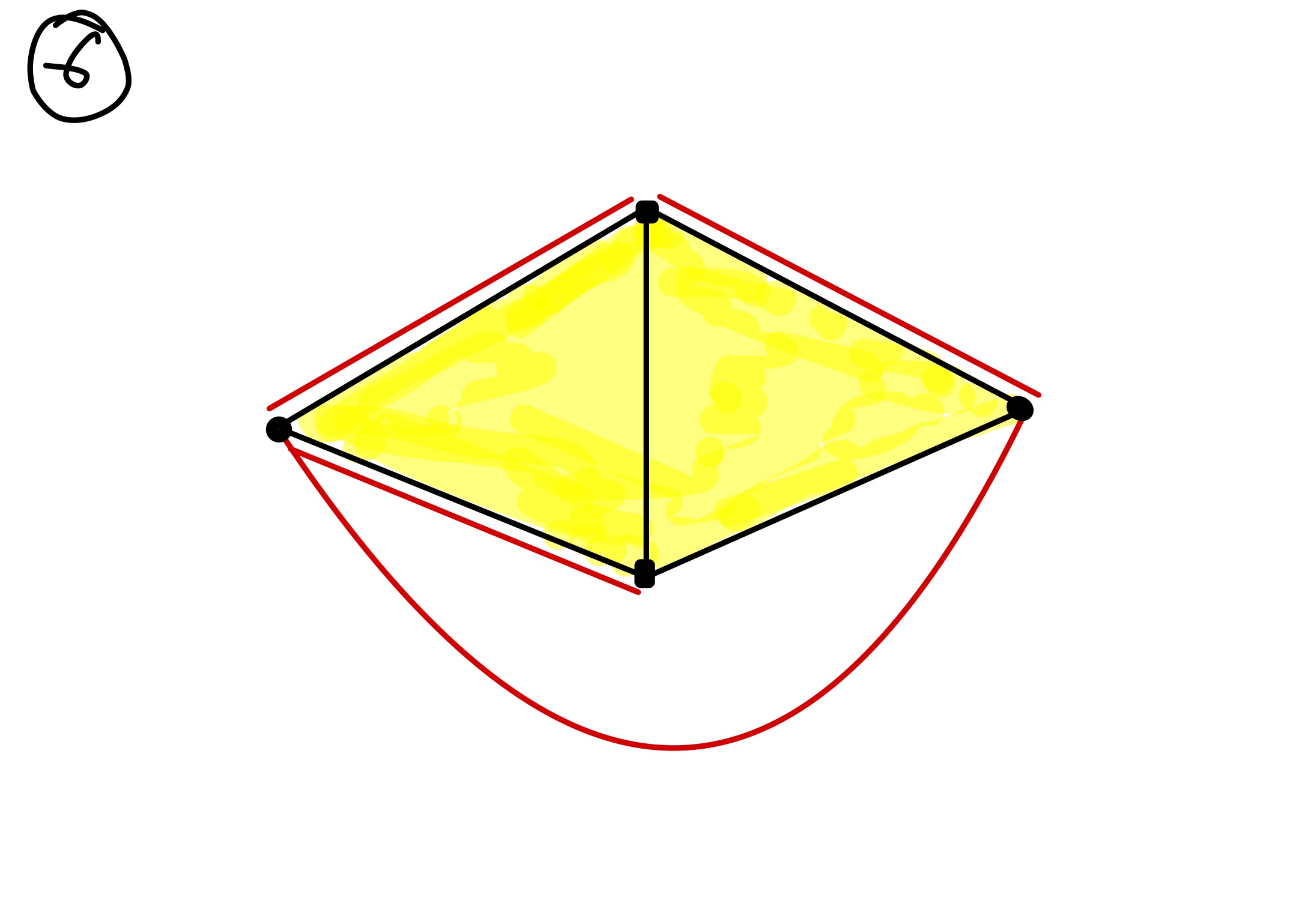}
\end{center}
\end{subfigure}
\begin{subfigure}{0.17\textwidth}
\begin{center}
\includegraphics[trim={15cm 7cm 13cm 12cm}, clip,width=2.2cm]{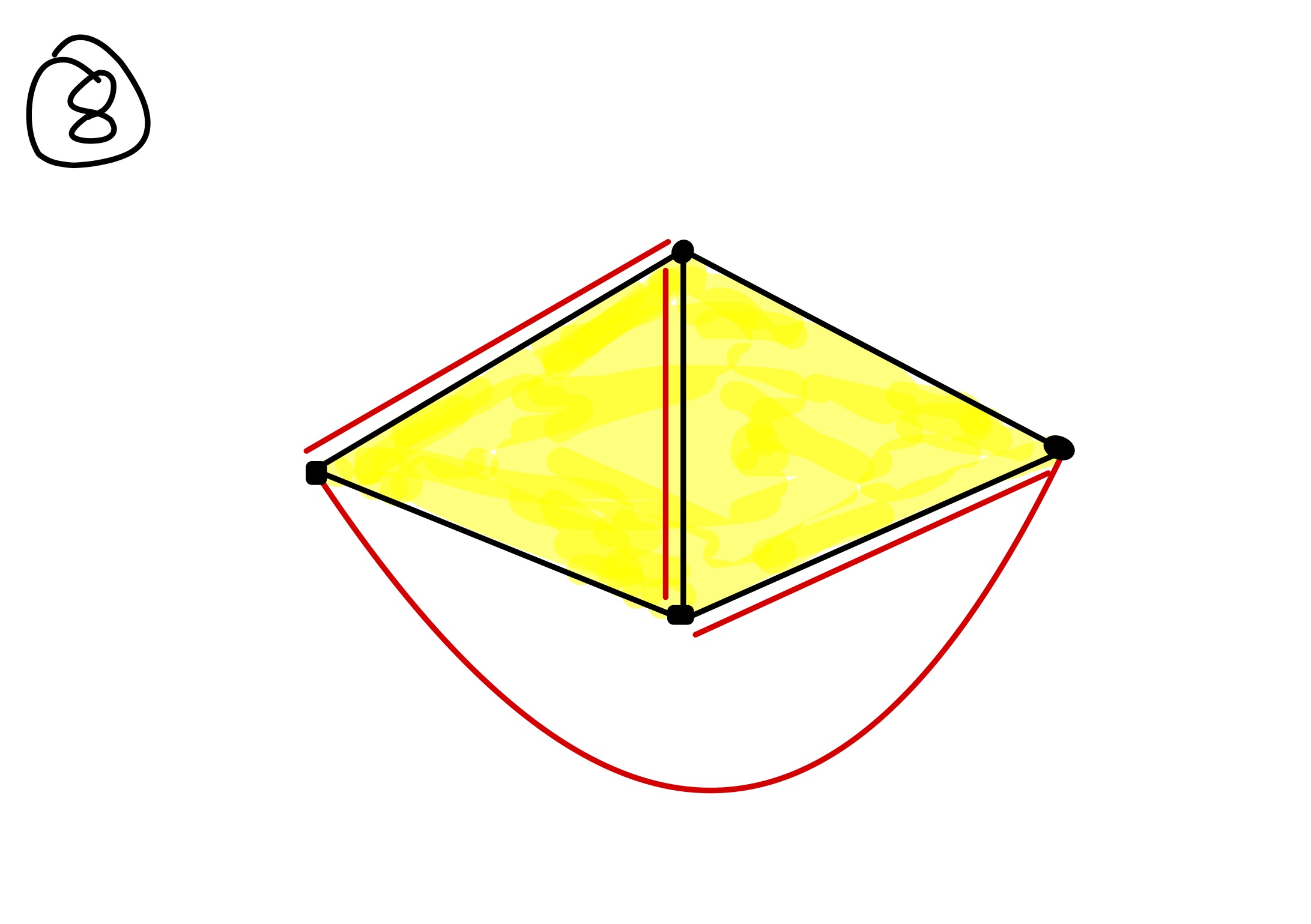}
\end{center}
\end{subfigure}
\begin{subfigure}{0.17\textwidth}
\begin{center}
\includegraphics[trim={21cm 14cm 20cm 12cm}, clip,width=2.2cm]{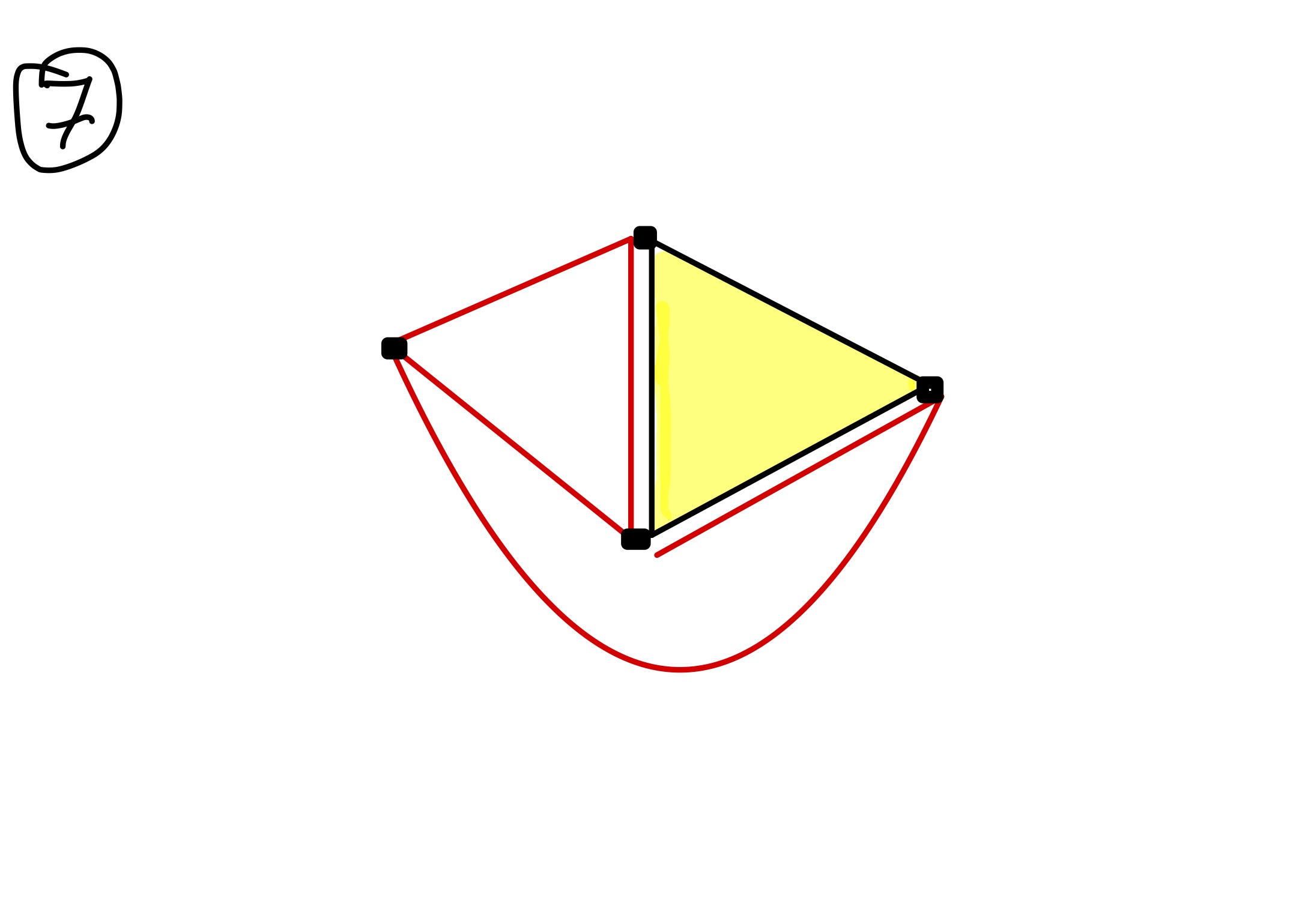}
\end{center}
\end{subfigure}

\begin{subfigure}{0.17\textwidth}
\begin{center}
\includegraphics[trim={23cm 18cm 22cm 12cm}, clip,width=2.2cm]{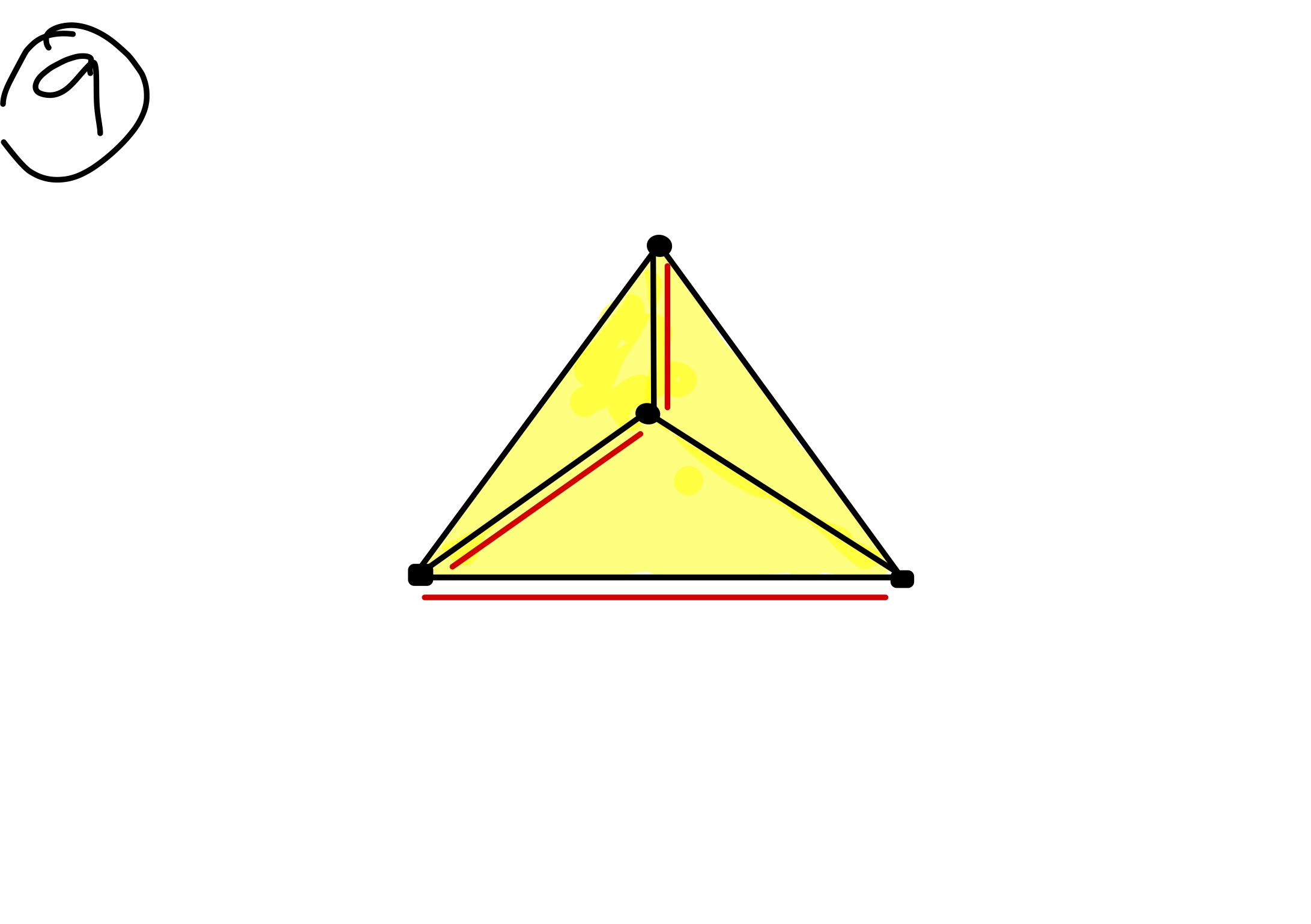}
\end{center}
\end{subfigure}
\begin{subfigure}{0.17\textwidth}
\begin{center}
\includegraphics[trim={15cm 11cm 26cm 9cm}, clip,width=2.2cm]{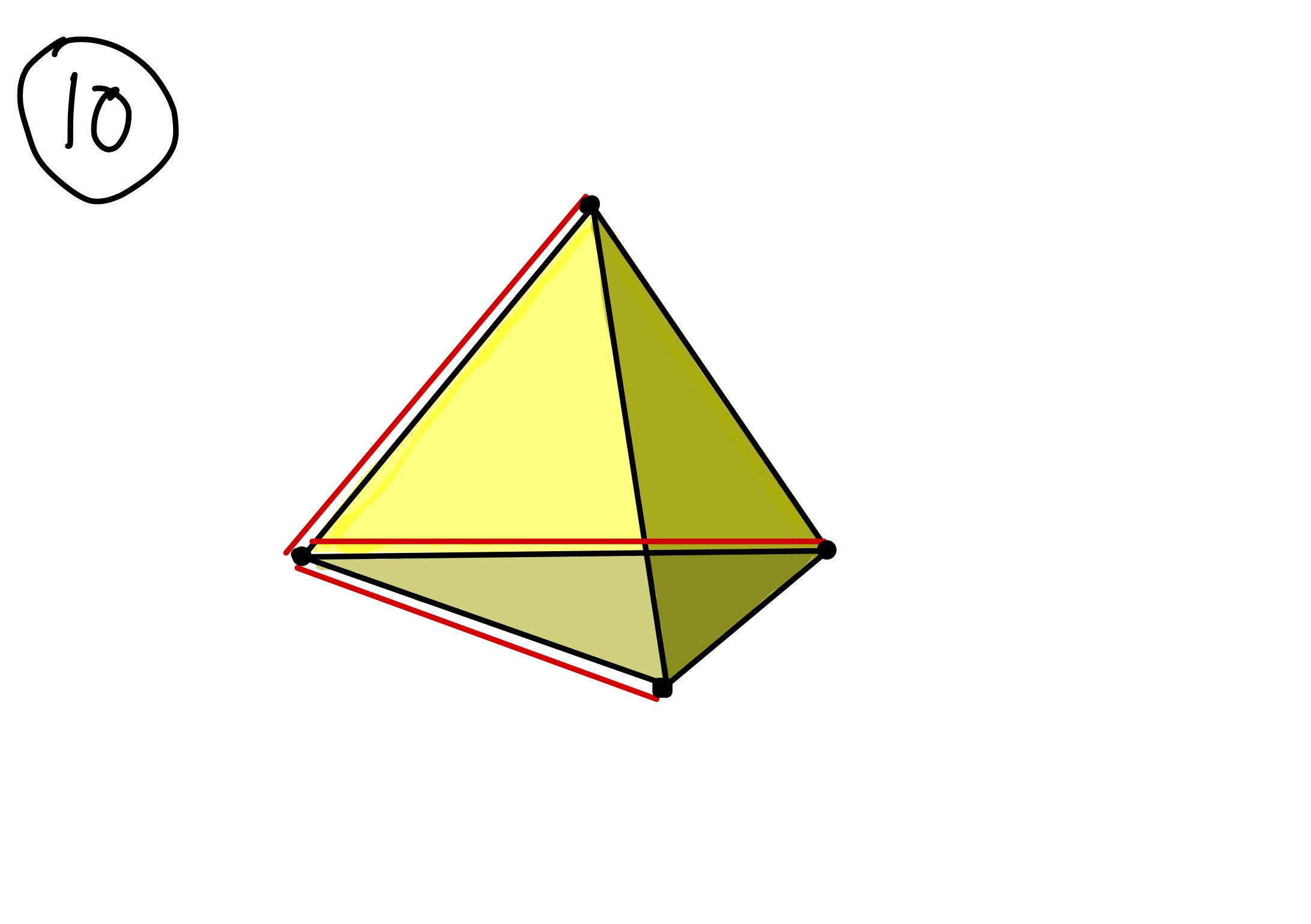}
\end{center}
\end{subfigure}
\begin{subfigure}{0.17\textwidth}
\begin{center}
\includegraphics[trim={15cm 14cm 11cm 9cm}, clip,width=2.2cm]{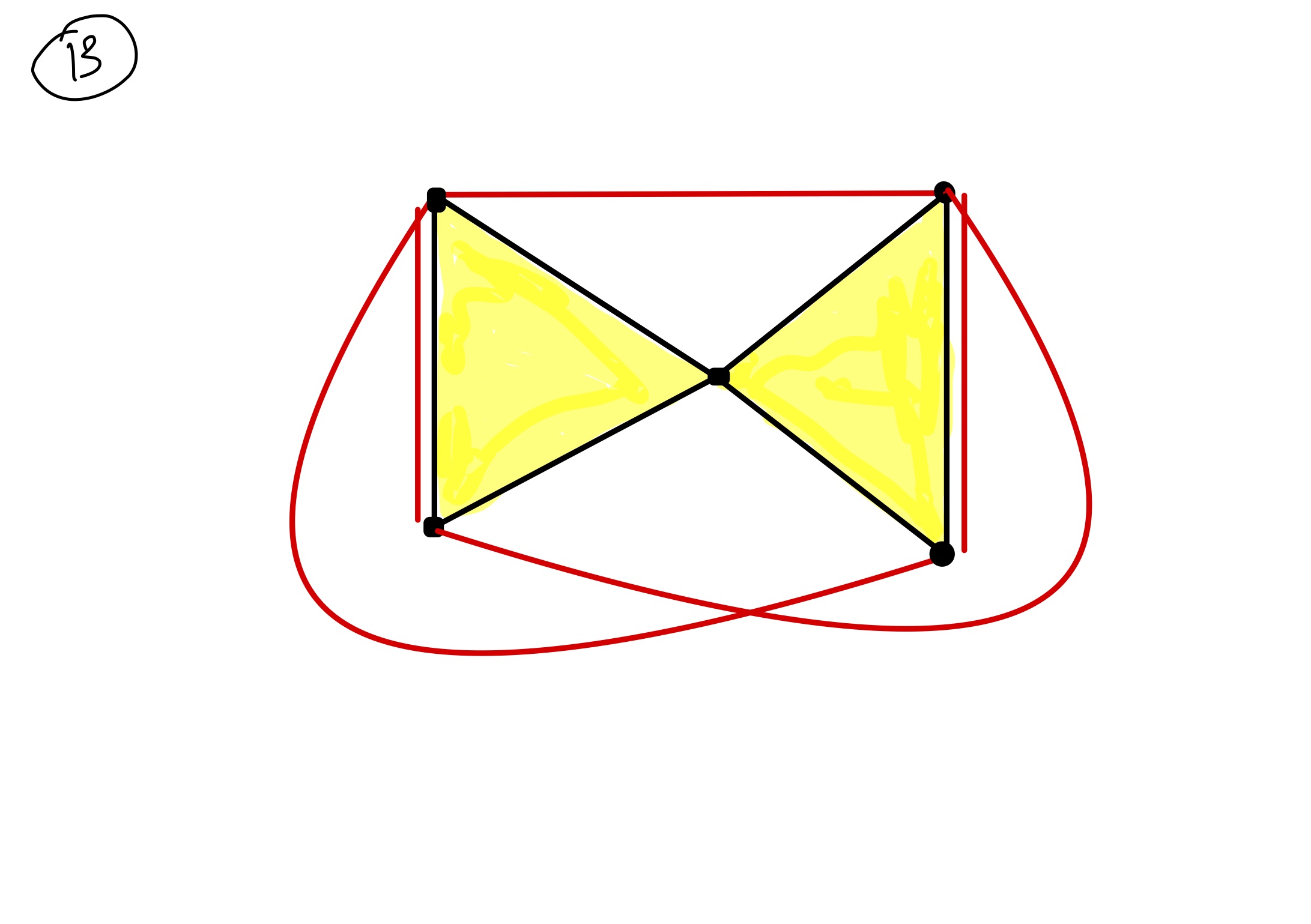}
\end{center}
\end{subfigure}
\begin{subfigure}{0.17\textwidth}
\begin{center}
\includegraphics[trim={17cm 19cm 29cm 13cm}, clip,width=2.2cm]{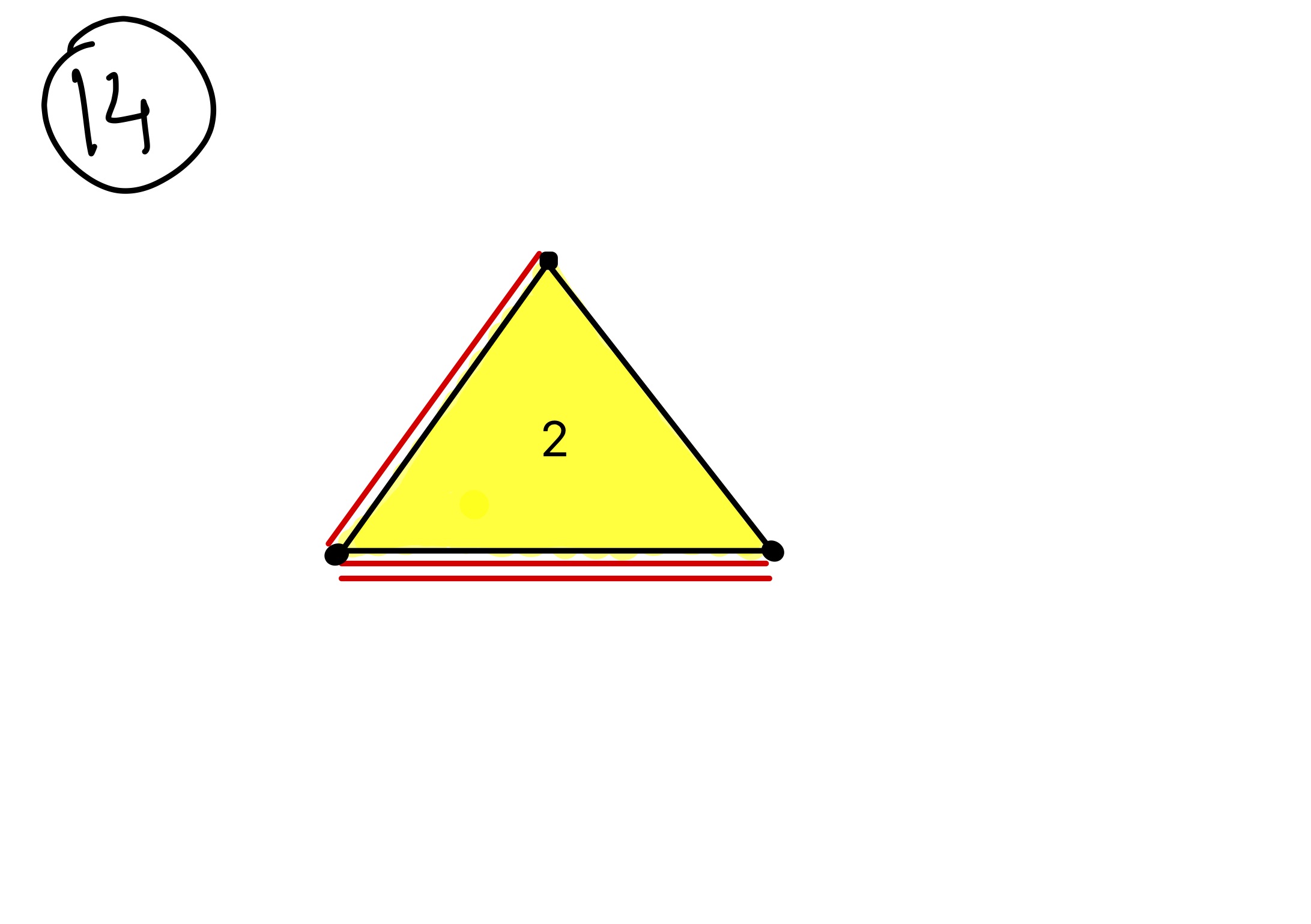}
\end{center}
\end{subfigure}
\begin{subfigure}{0.17\textwidth}
\begin{center}
\includegraphics[trim={22cm 12cm 22cm 10cm}, clip,width=2.2cm]{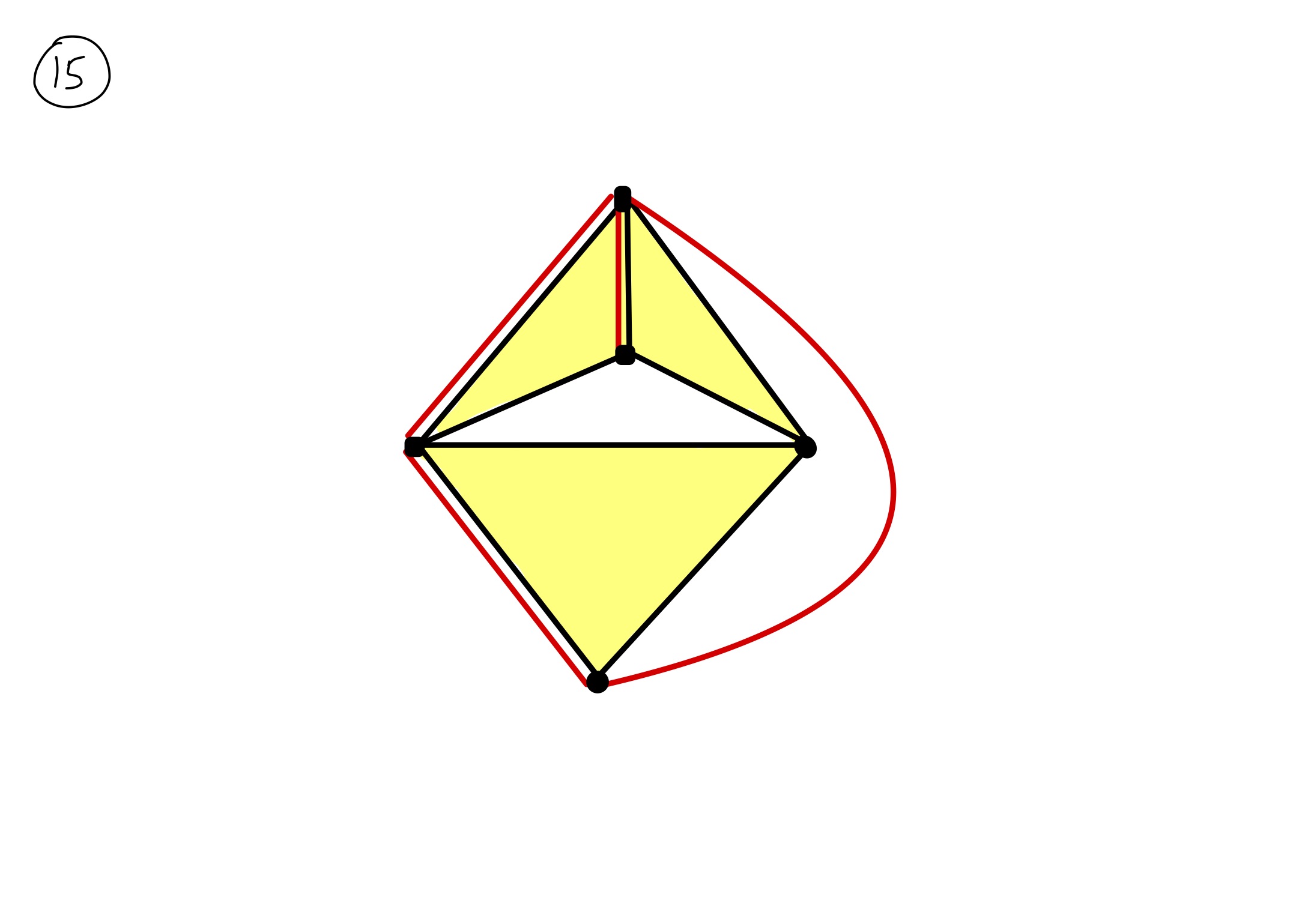}
\end{center}
\end{subfigure}

\begin{subfigure}{0.17\textwidth}
\begin{center}
\includegraphics[trim={24cm 12cm 21cm 10cm}, clip,width=2.2cm]{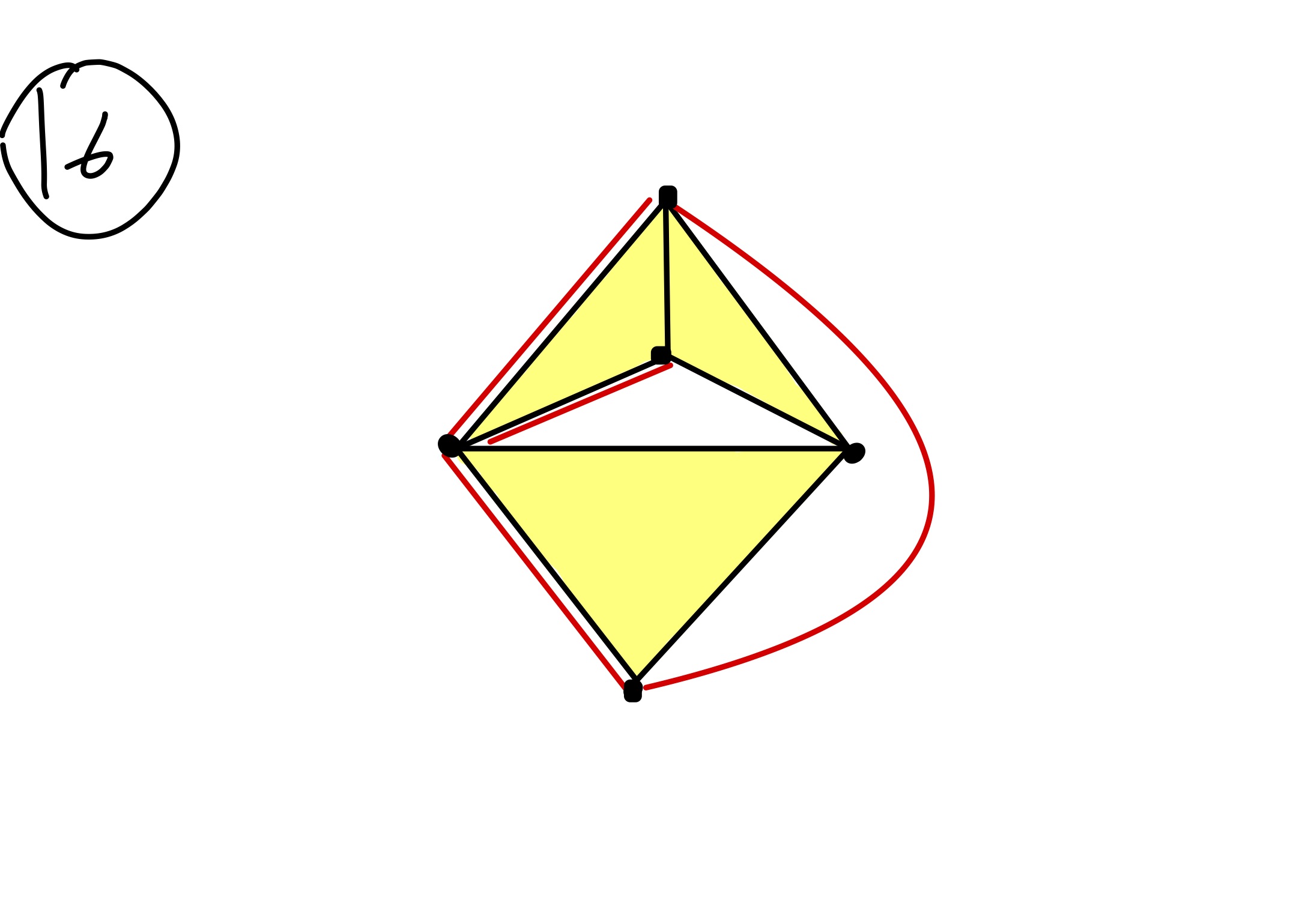}
\end{center}
\end{subfigure}
\begin{subfigure}{0.17\textwidth}
\begin{center}
\includegraphics[trim={17cm 13cm 22cm 9cm}, clip,width=2.2cm]{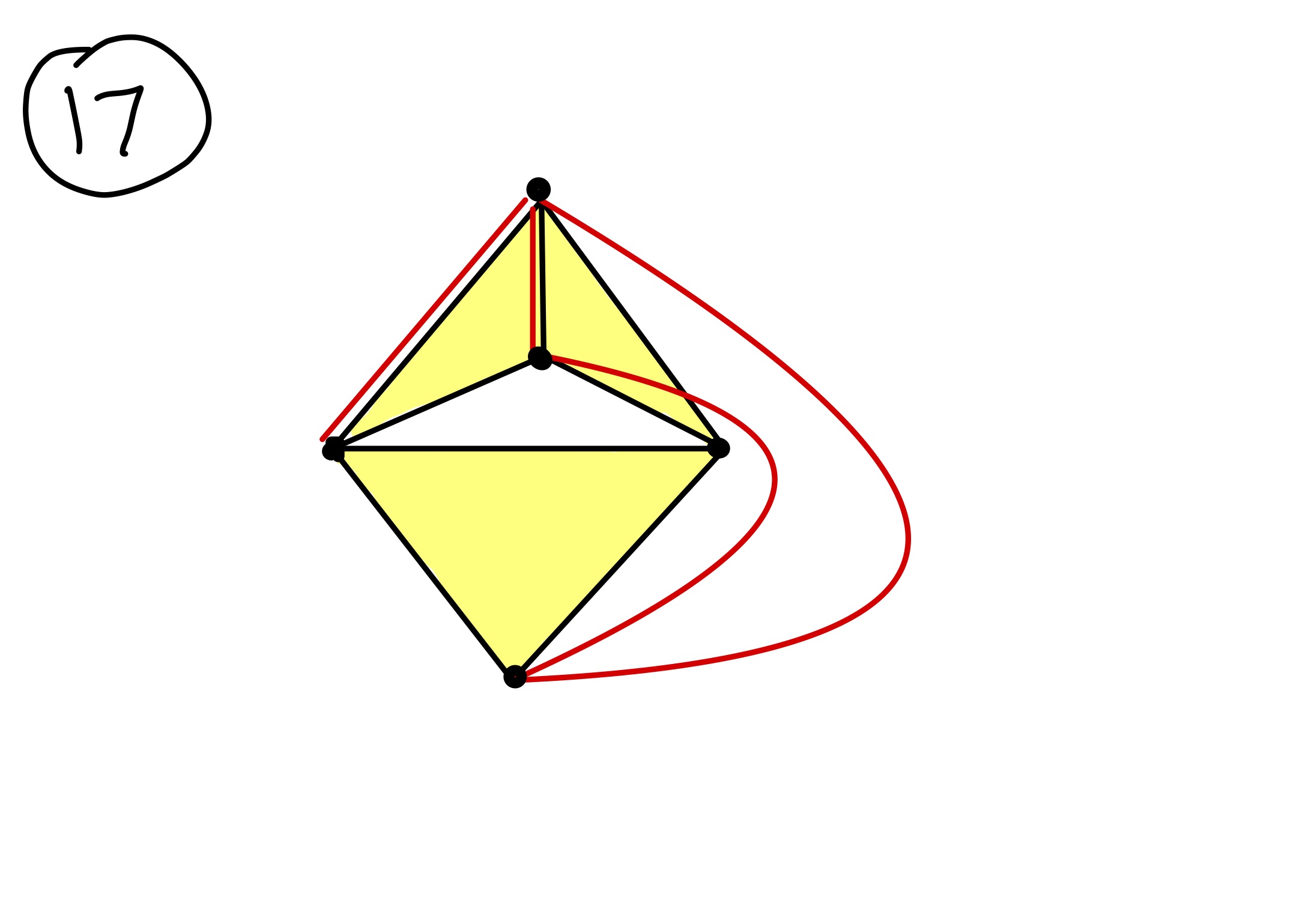}
\end{center}
\end{subfigure}
\begin{subfigure}{0.17\textwidth}
\begin{center}
\includegraphics[trim={09cm 2cm 9cm 10cm}, clip,width=2.2cm]{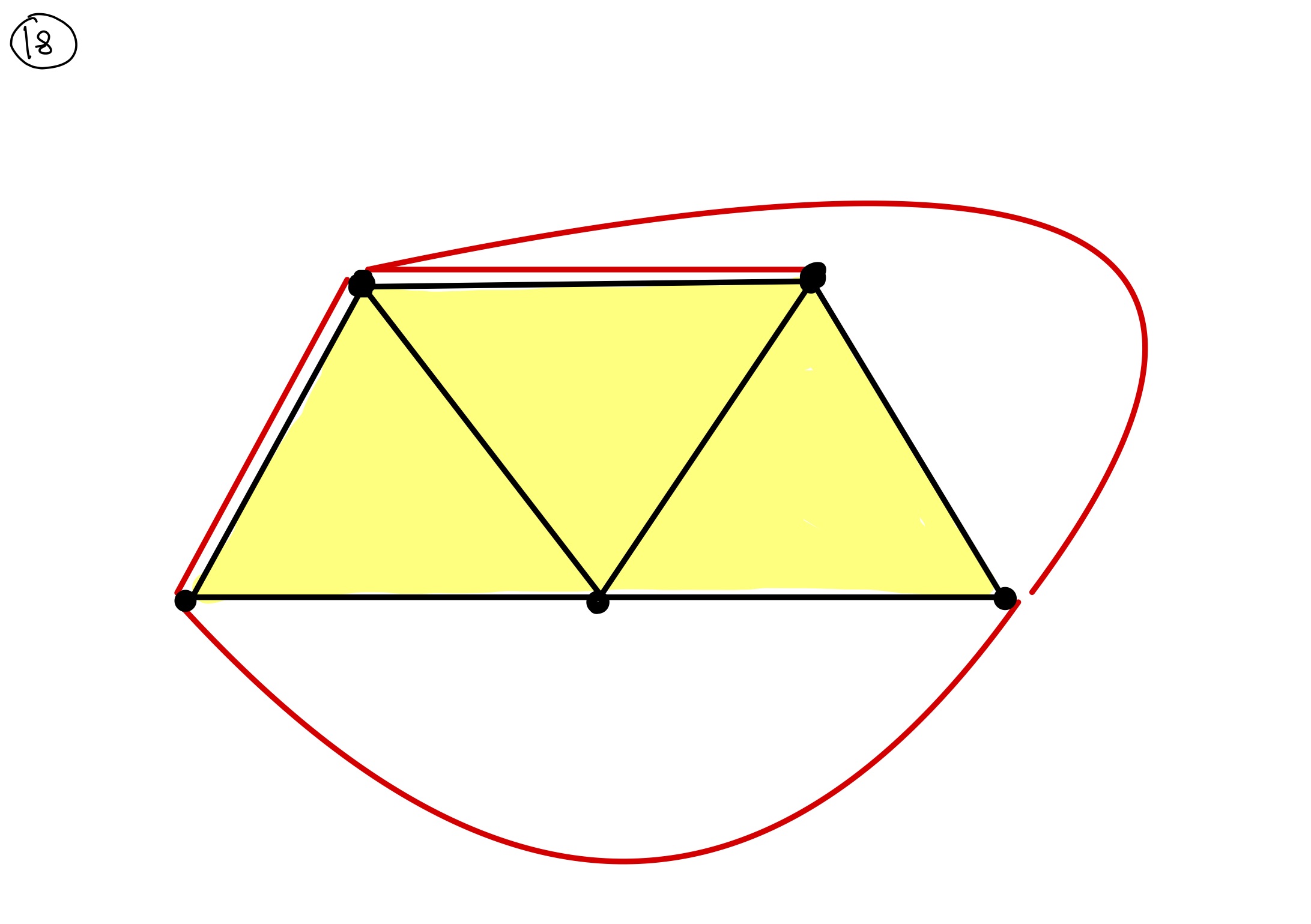}
\end{center}
\end{subfigure}
\begin{subfigure}{0.17\textwidth}
\begin{center}
\includegraphics[trim={10cm 6cm 13cm 4cm}, clip,width=2.2cm]{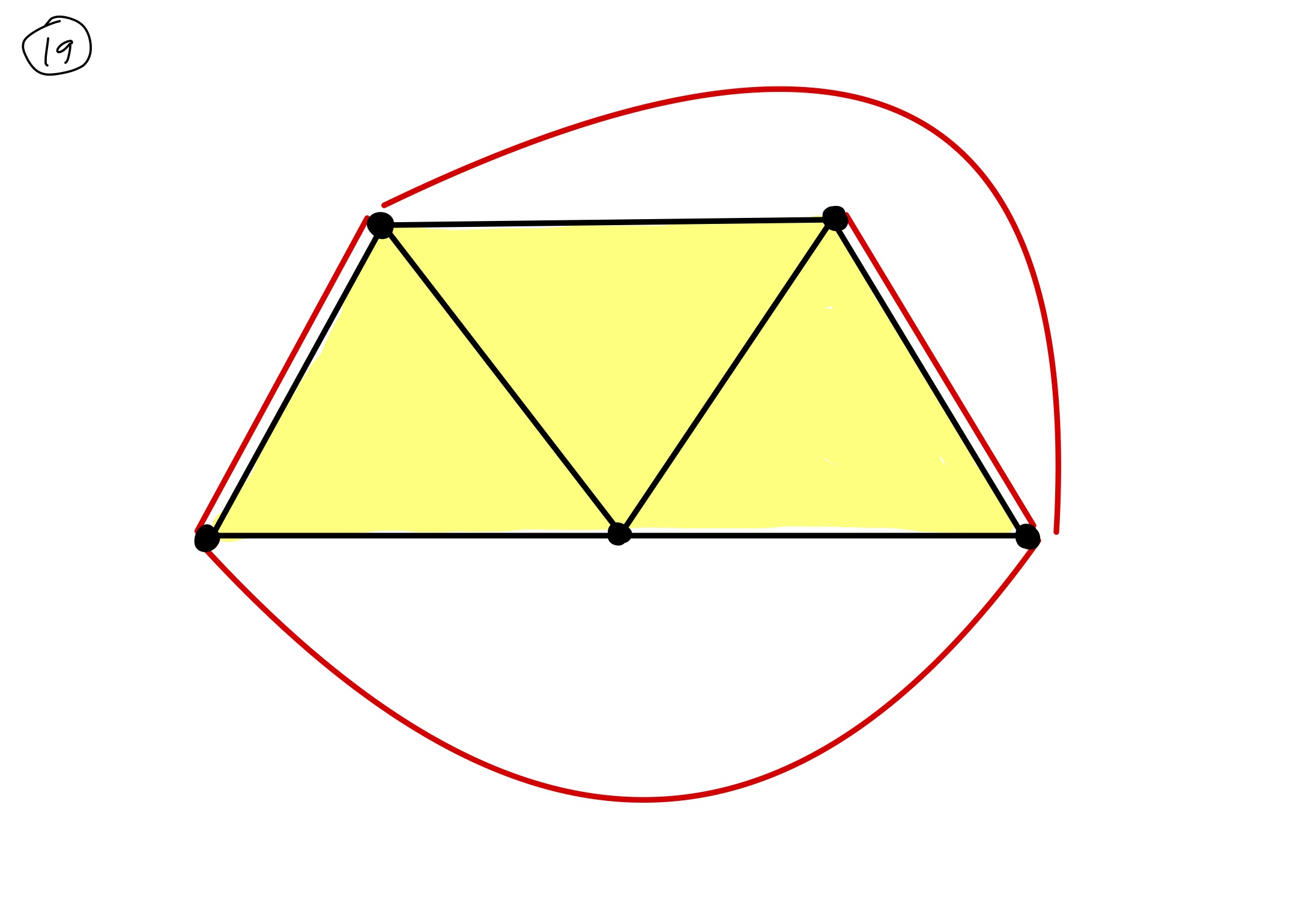}
\end{center}
\end{subfigure}
\begin{subfigure}{0.17\textwidth}
\begin{center}
\includegraphics[trim={20cm 3cm 10cm 9cm}, clip,width=2.2cm]{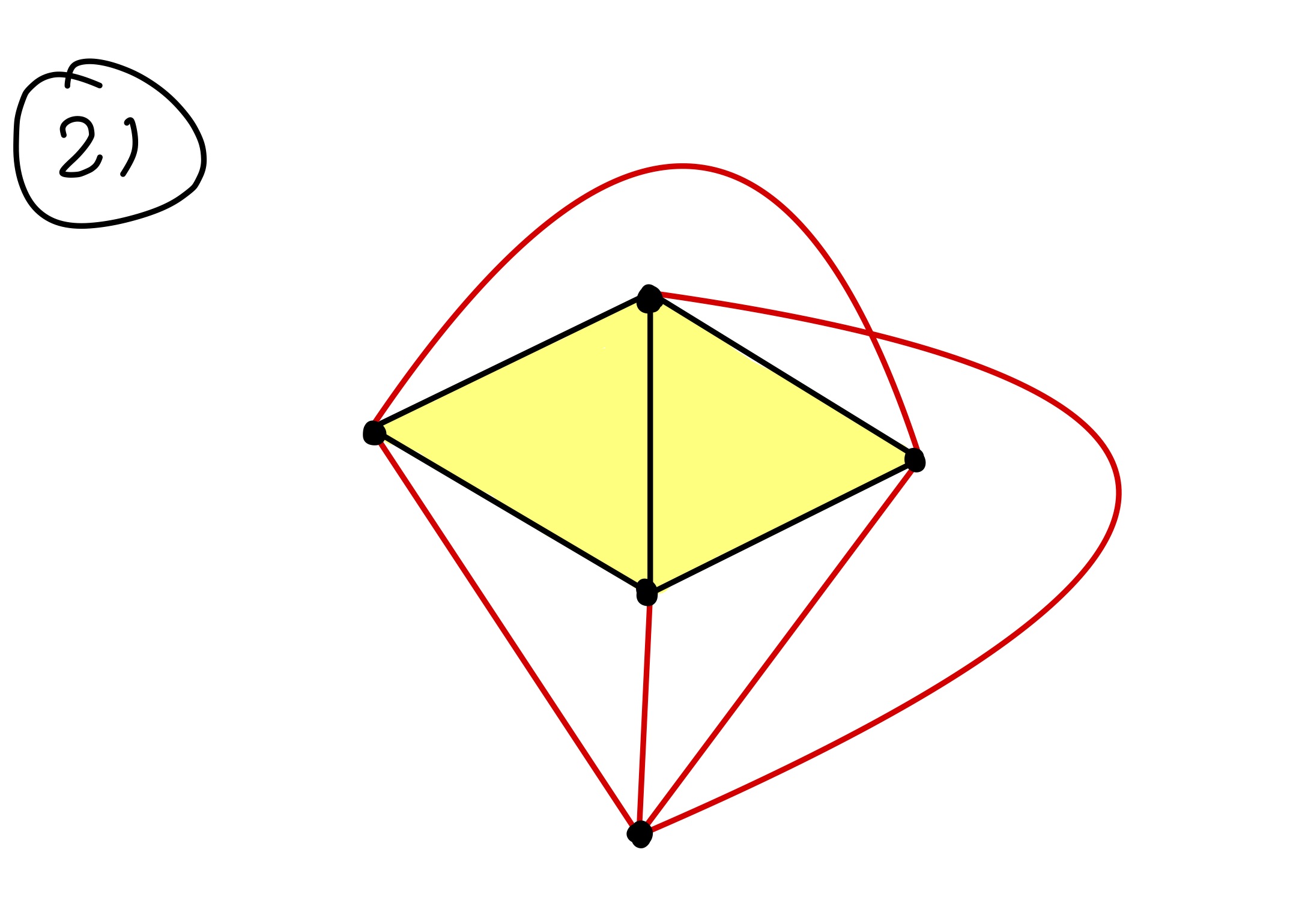}
\end{center}
\end{subfigure}

\begin{subfigure}{0.17\textwidth}
\begin{center}
\includegraphics[trim={24cm 25cm 29cm 9cm}, clip,width=2.2cm]{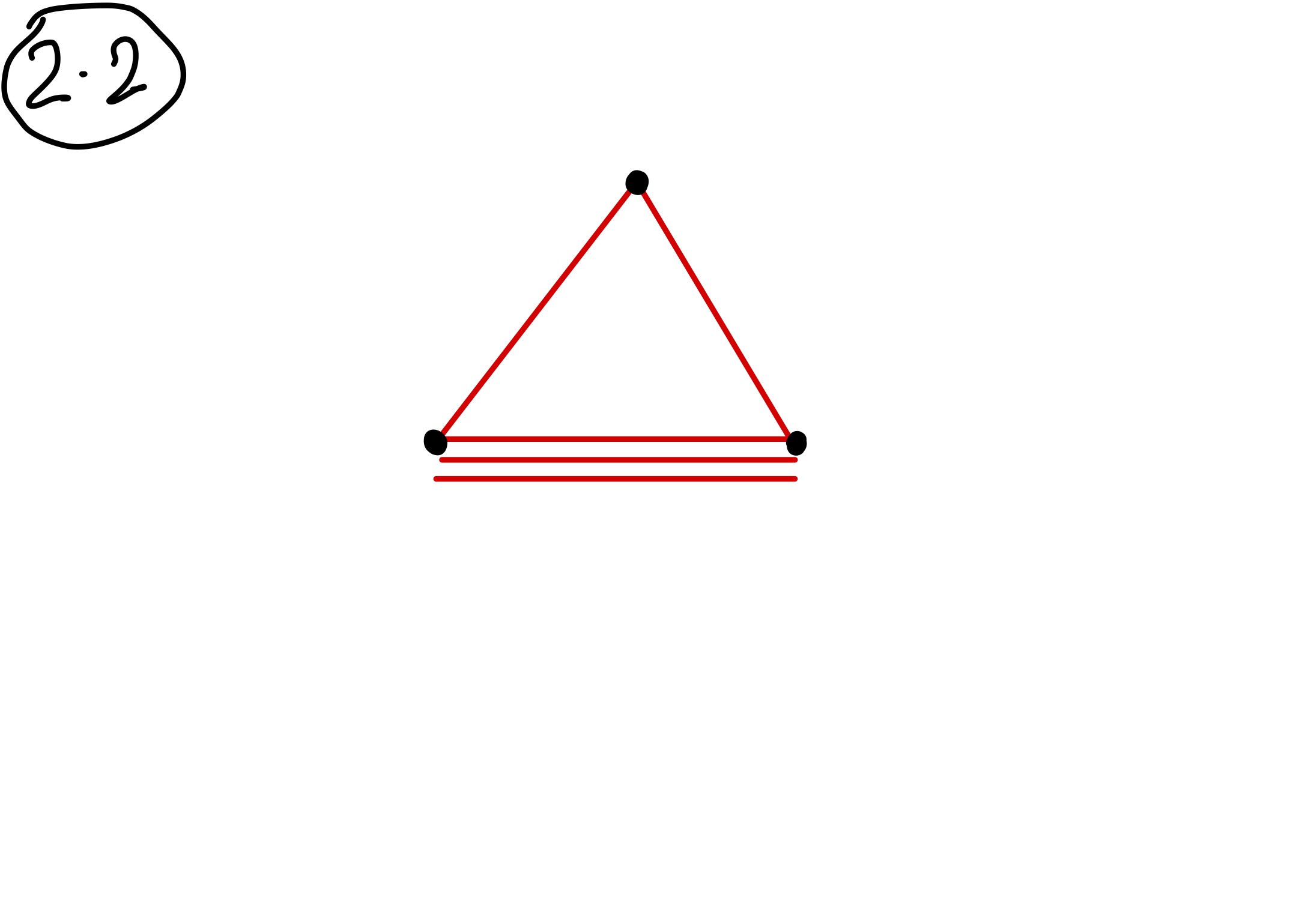}
\end{center}
\end{subfigure}
\begin{subfigure}{0.17\textwidth}
\begin{center}
\includegraphics[trim={16cm 30cm 30cm 17cm}, clip,width=2.2cm]{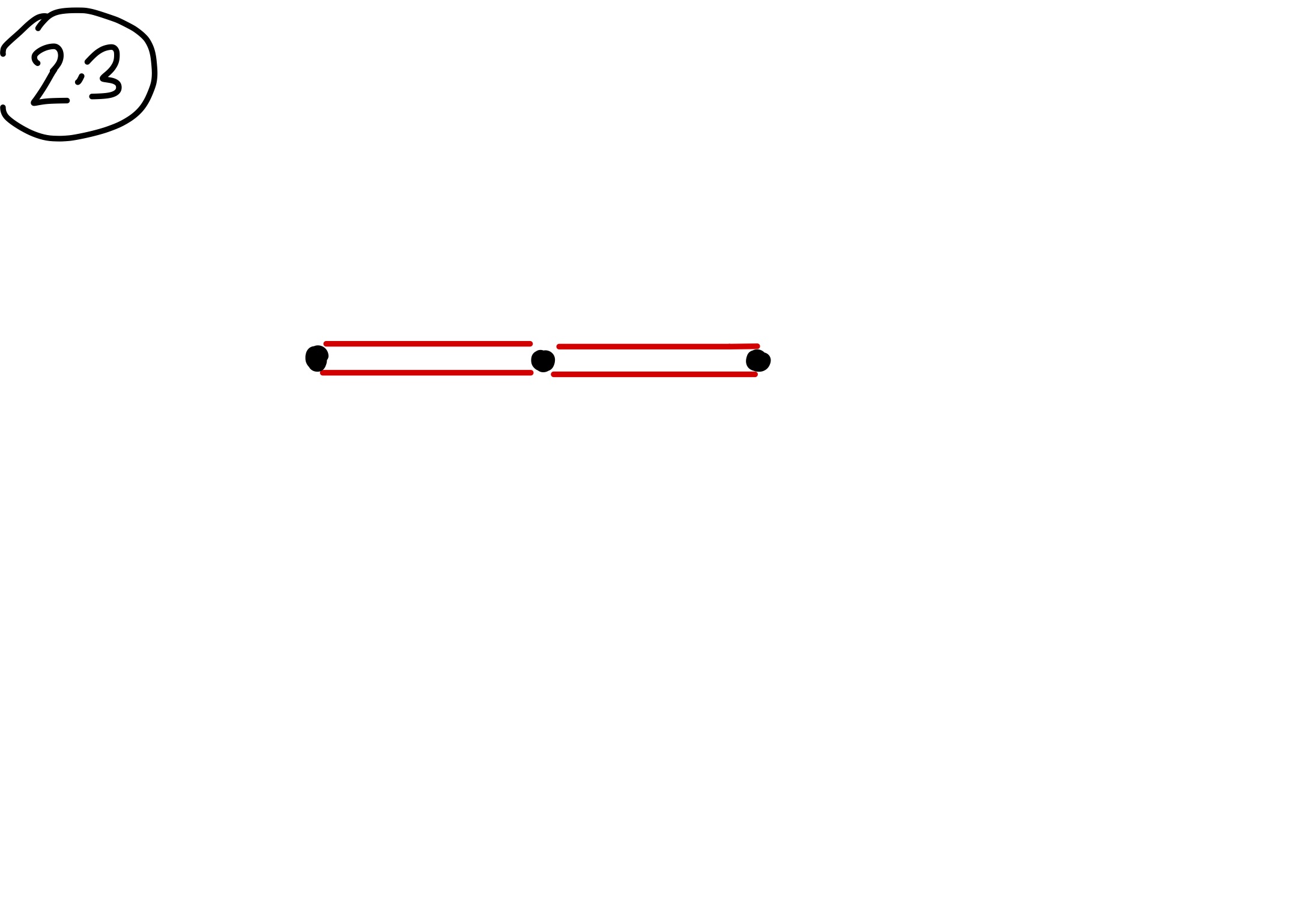}
\end{center}
\end{subfigure}
\begin{subfigure}{0.17\textwidth}
\begin{center}
\includegraphics[trim={14cm 23cm 30cm 15cm}, clip,width=2.2cm]{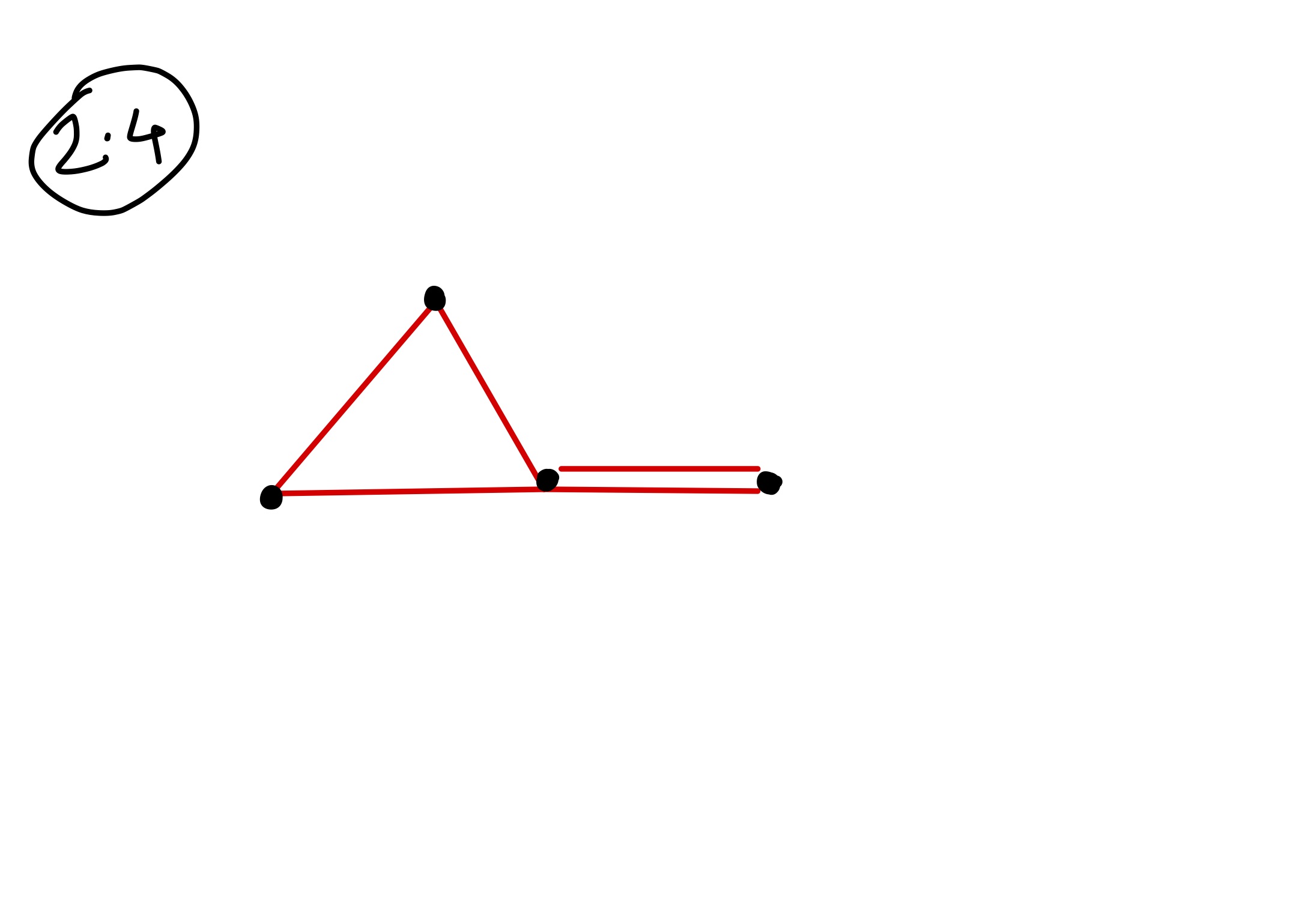}
\end{center}
\end{subfigure}
\begin{subfigure}{0.17\textwidth}
\begin{center}
\includegraphics[trim={23cm 15cm 28cm 12cm}, clip,width=2.2cm]{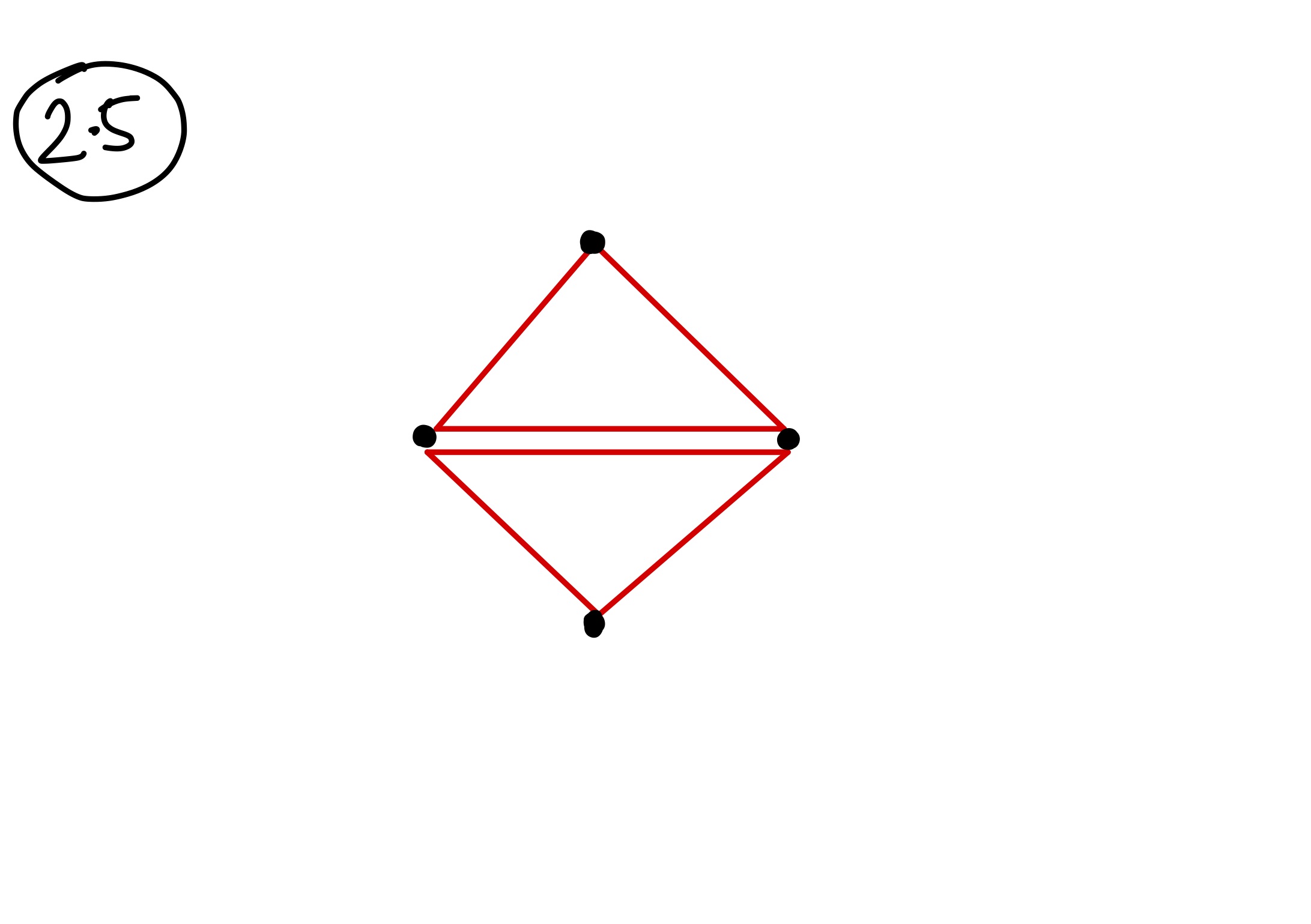}
\end{center}
\end{subfigure}
\begin{subfigure}{0.17\textwidth}
\begin{center}
\includegraphics[trim={19cm 19cm 26cm 16cm}, clip,width=2.2cm]{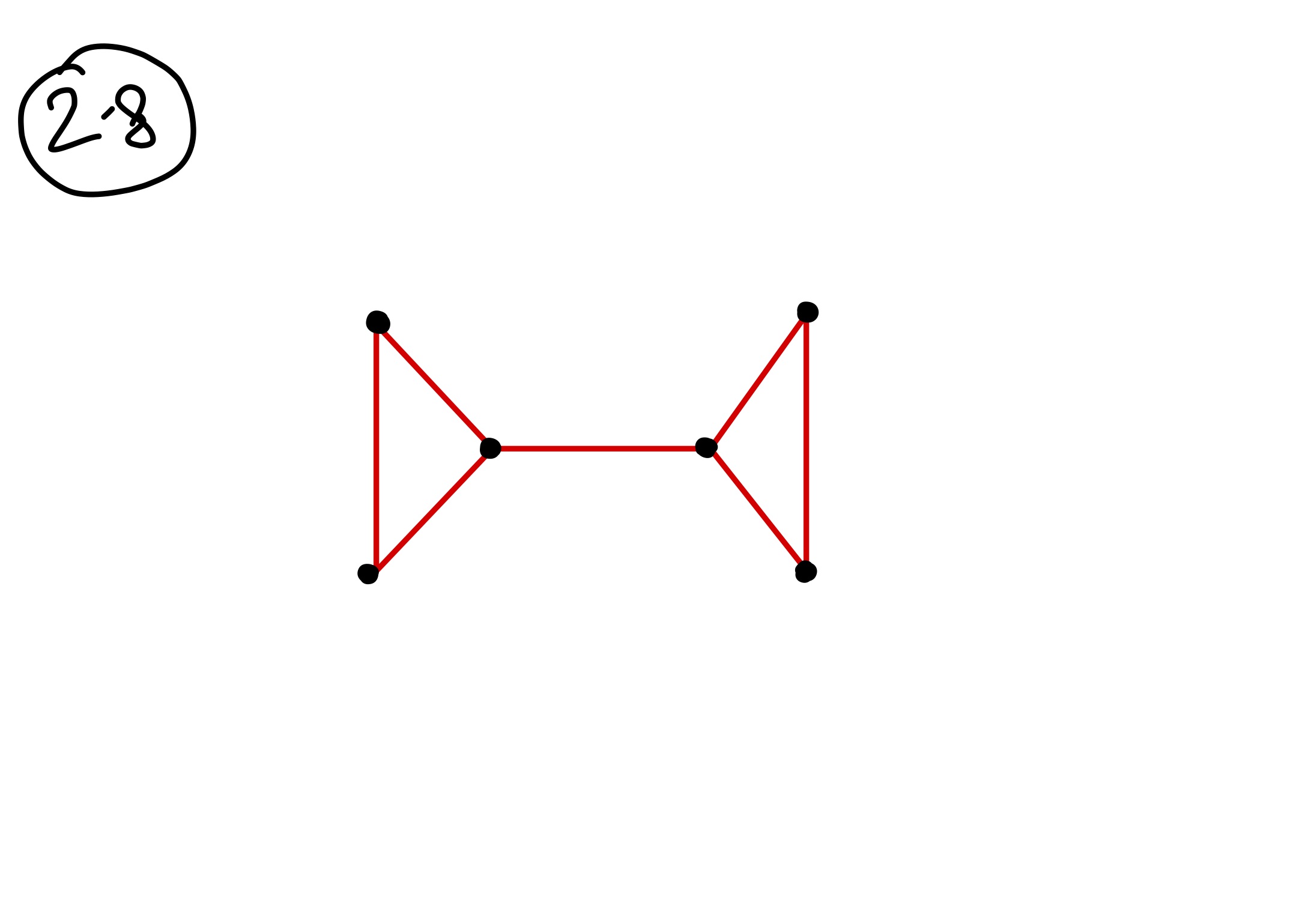}
\end{center}
\end{subfigure}

\begin{subfigure}{0.17\textwidth}
\begin{center}
\includegraphics[trim={20cm 23cm 26cm 8cm}, clip,width=2.2cm]{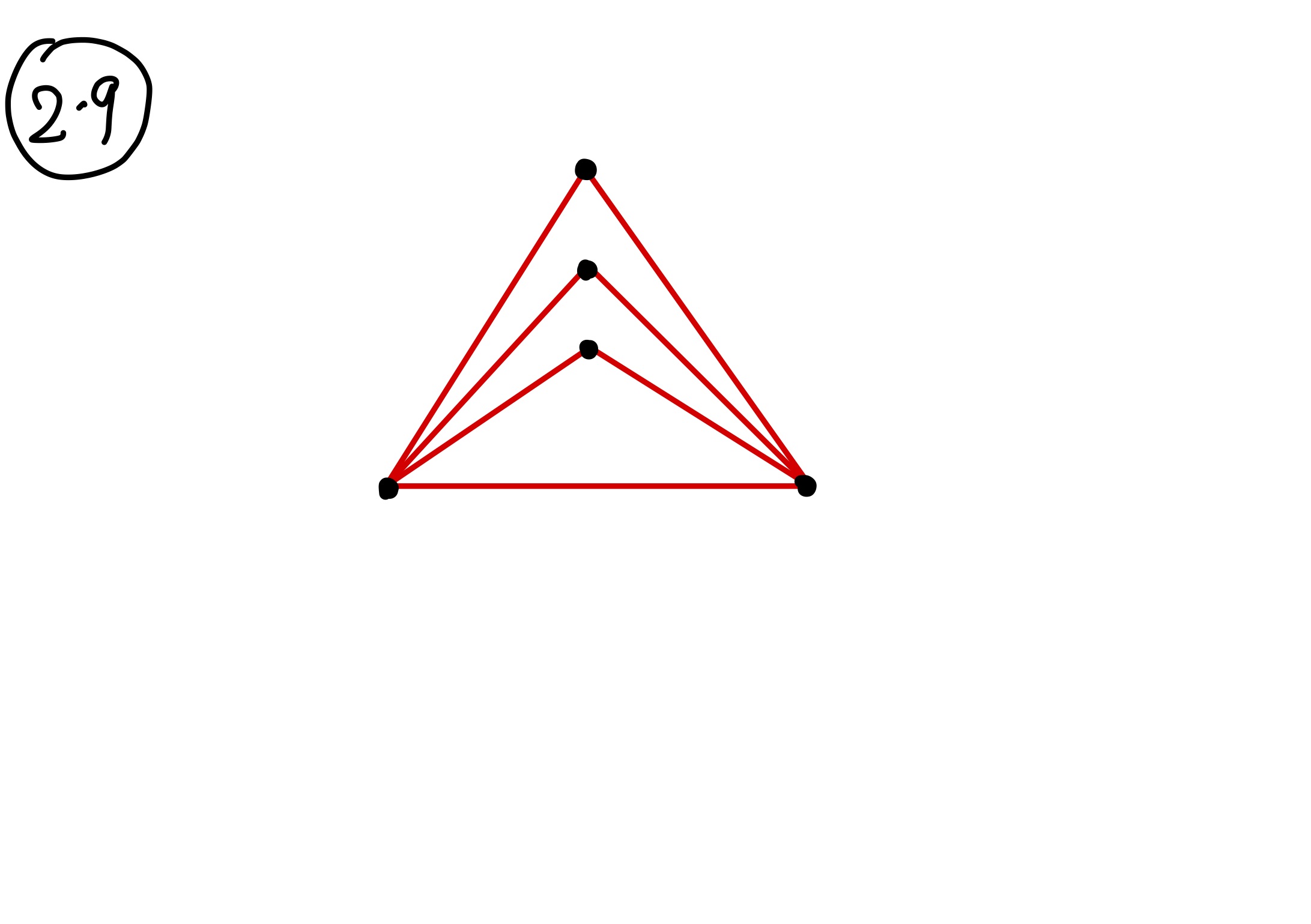}
\end{center}
\end{subfigure}
\begin{subfigure}{0.17\textwidth}
\begin{center}
\includegraphics[trim={21cm 19cm 27cm 16cm}, clip,width=2.2cm]{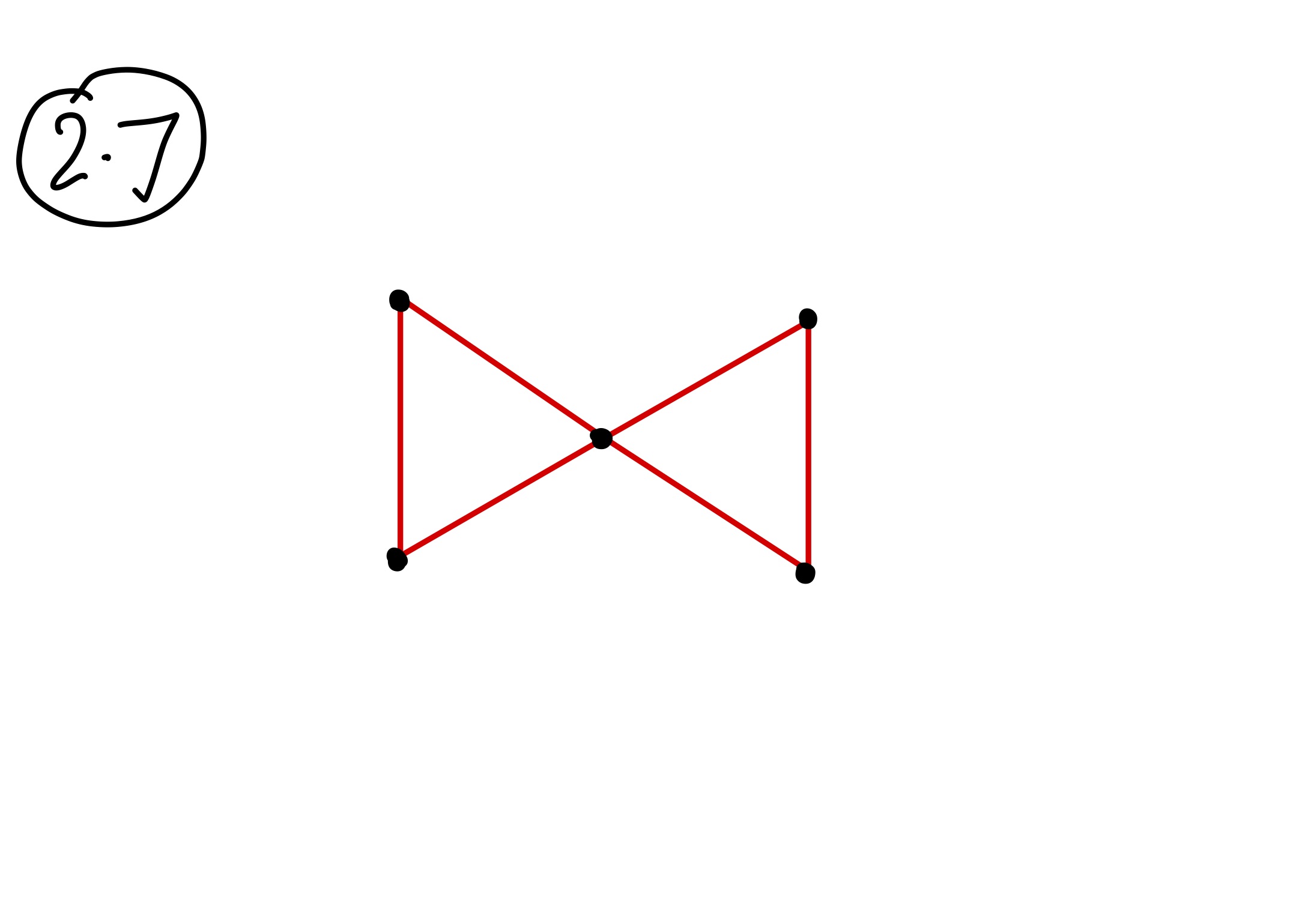}
\end{center}
\end{subfigure}
\begin{subfigure}{0.17\textwidth}
\begin{center}
\includegraphics[trim={21cm 17cm 27cm 12cm}, clip,width=2.2cm]{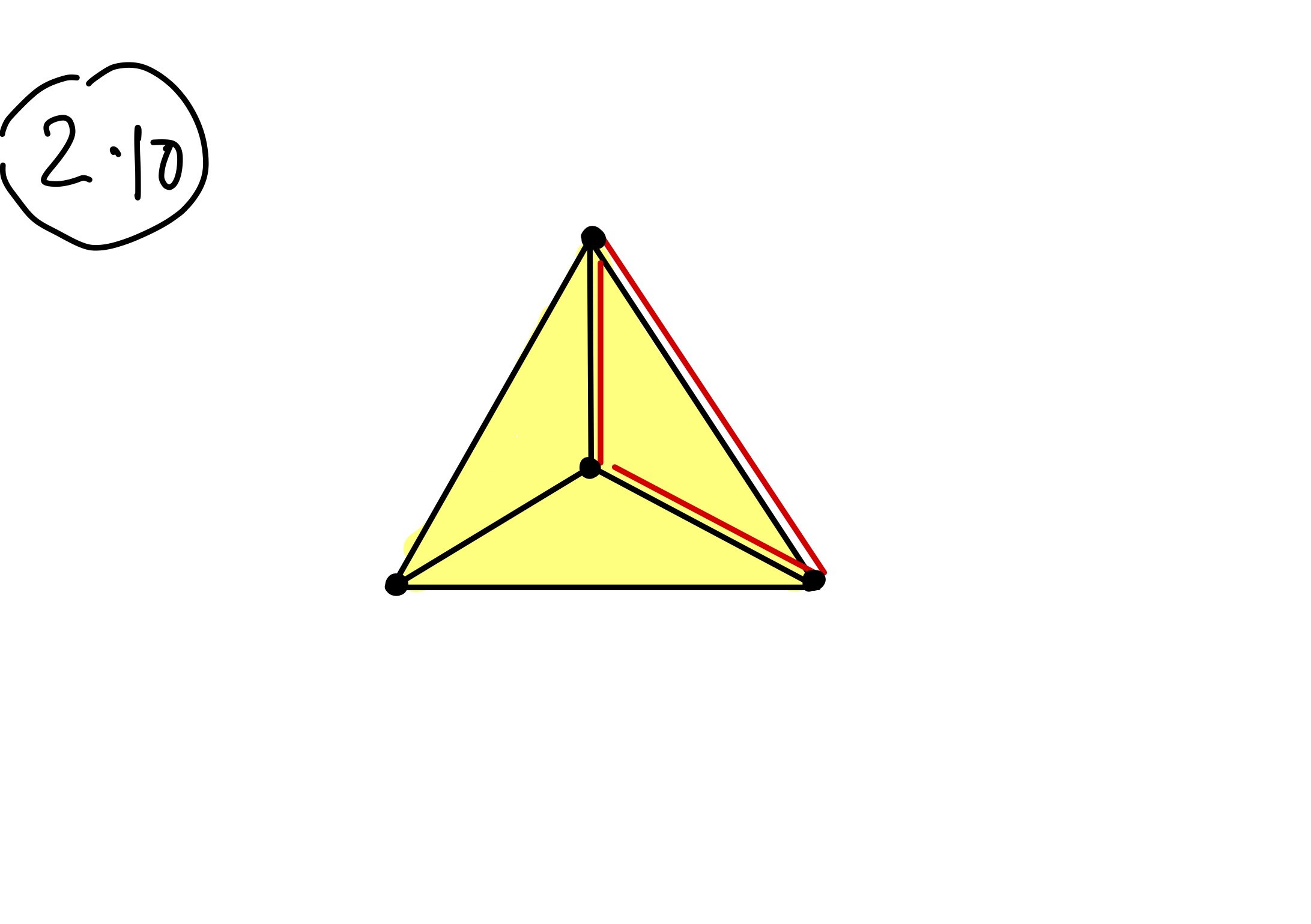}
\end{center}
\end{subfigure}
\begin{subfigure}{0.17\textwidth}
\begin{center}
\includegraphics[trim={18cm 16cm 20cm 15cm}, clip,width=2.2cm]{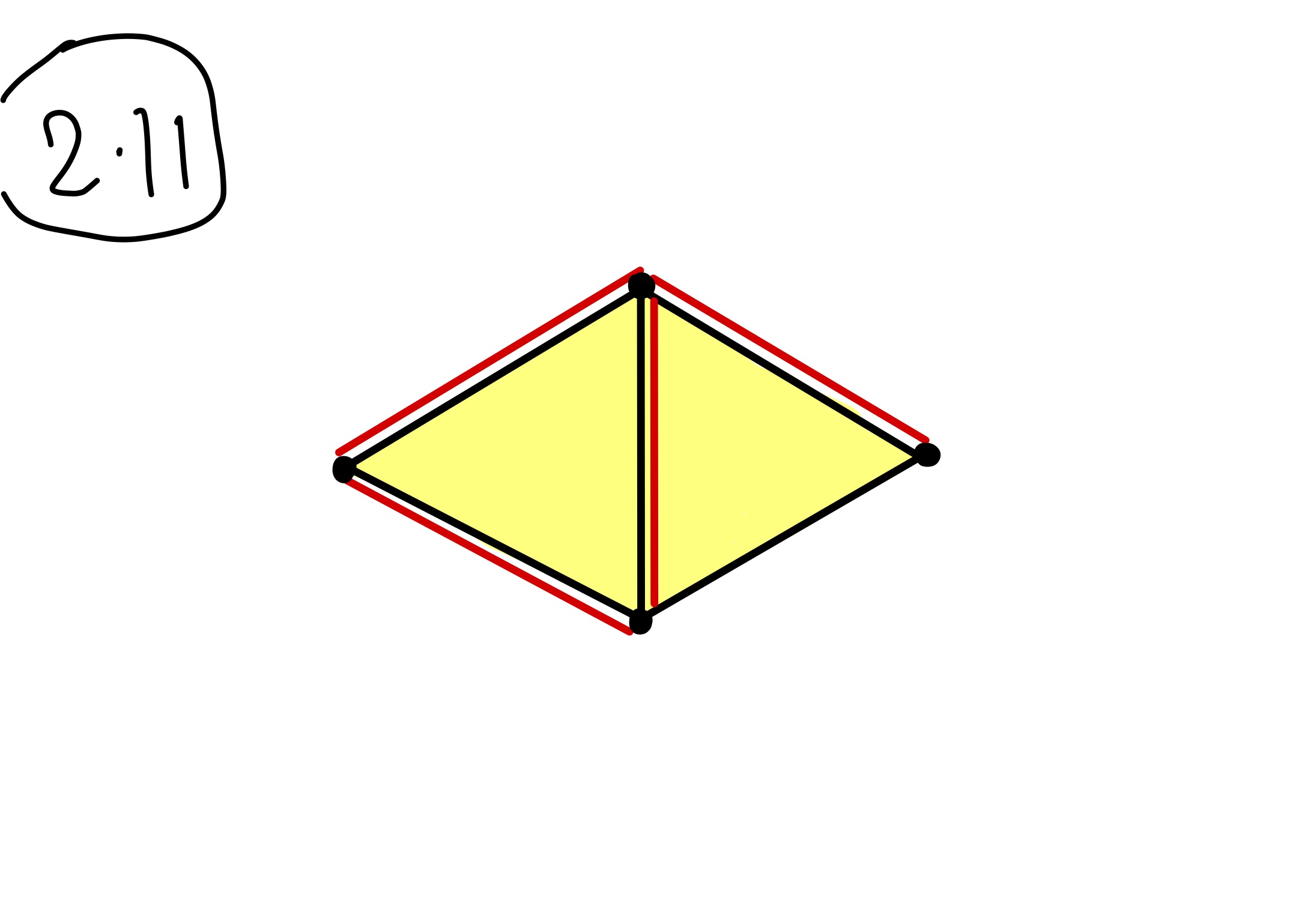}
\end{center}
\end{subfigure}
\begin{subfigure}{0.17\textwidth}
\begin{center}
\includegraphics[trim={19cm 12cm 19cm 12cm}, clip,width=2.2cm]{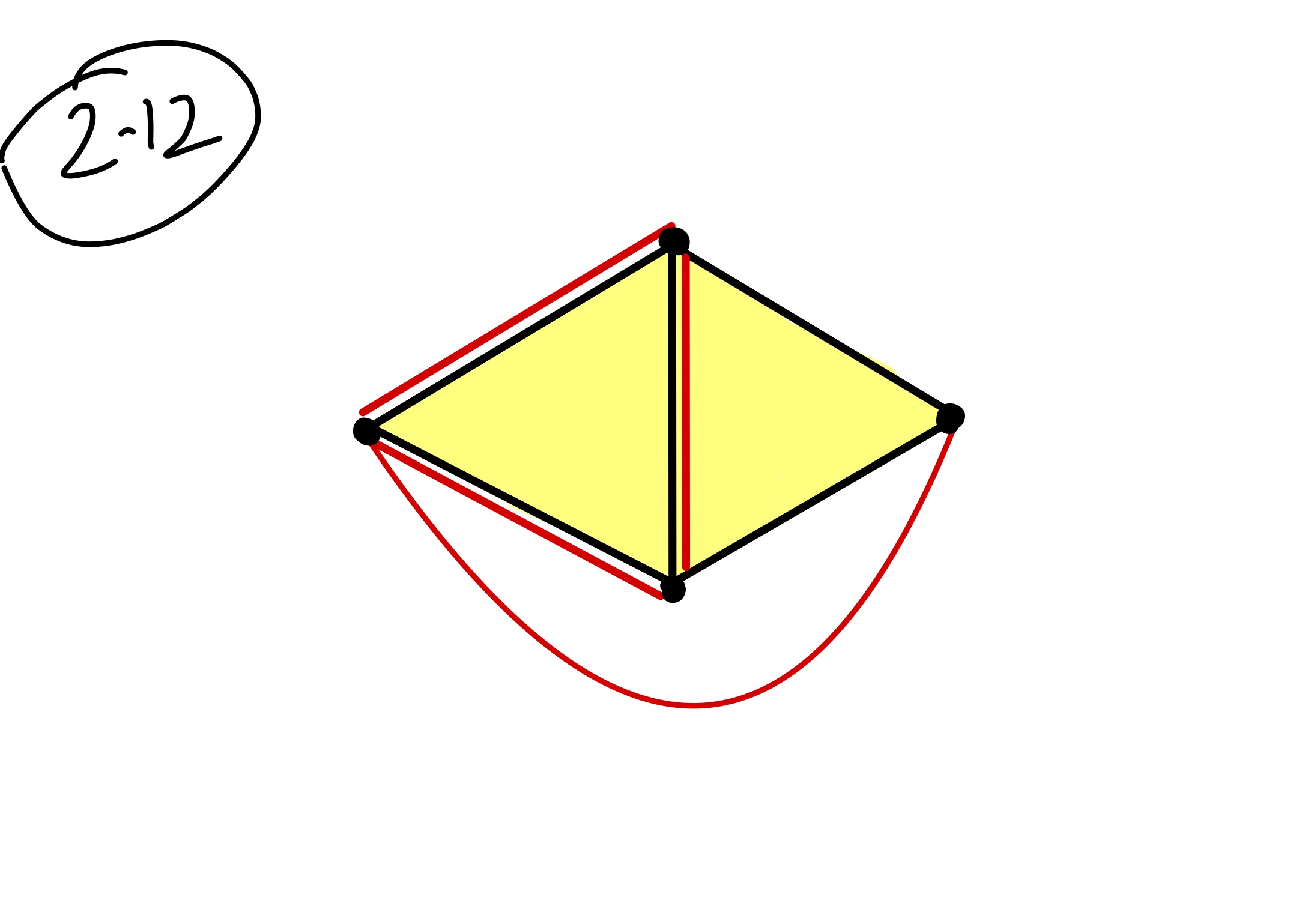}
\end{center}
\end{subfigure}

\begin{subfigure}{0.17\textwidth}
\begin{center}
\includegraphics[trim={22cm 13cm 17cm 11cm}, clip,width=2.2cm]{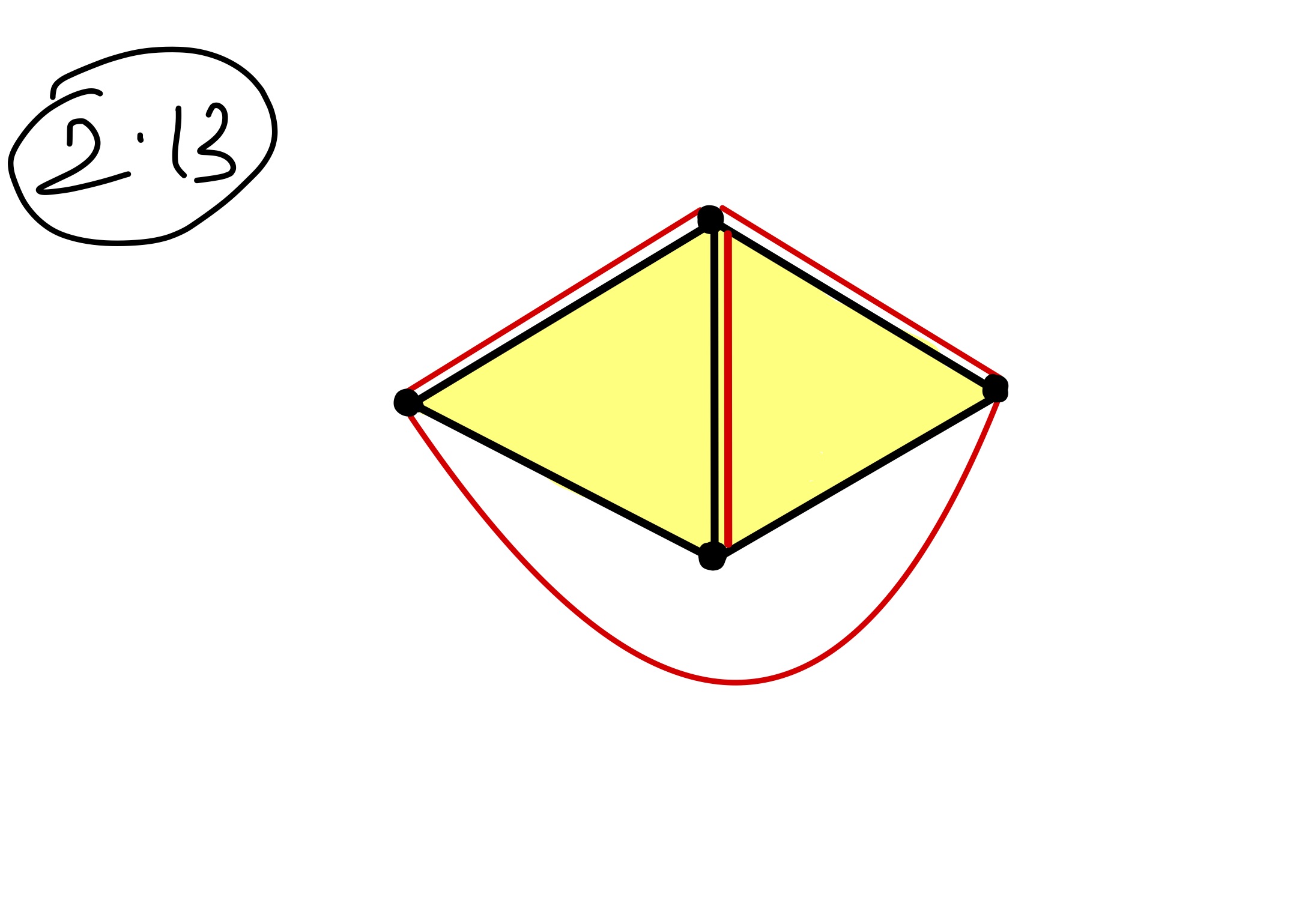}
\end{center}
\end{subfigure}
\begin{subfigure}{0.17\textwidth}
\begin{center}
\includegraphics[trim={21cm 9cm 17cm 11cm}, clip,width=2.2cm]{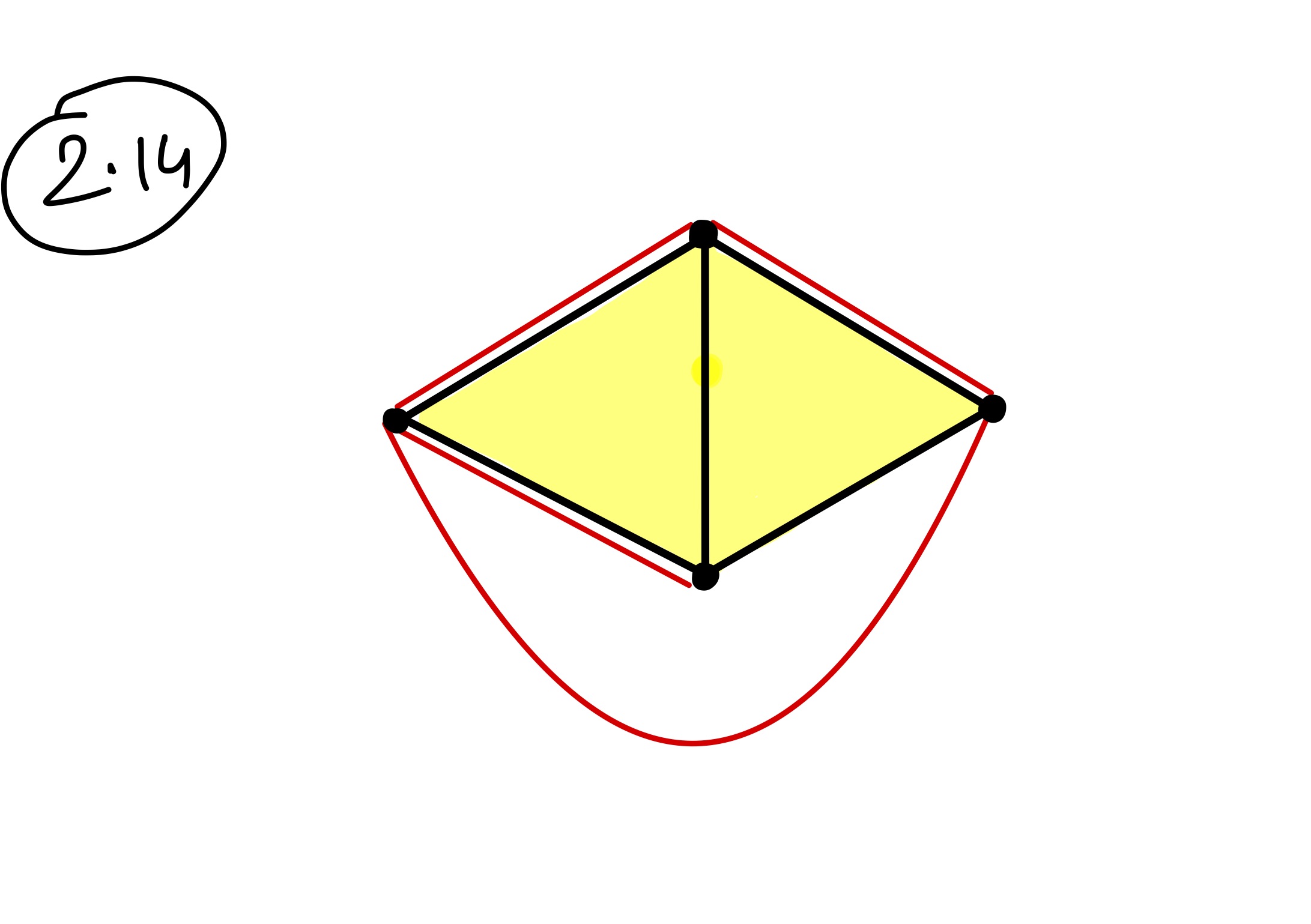}
\end{center}
\end{subfigure}
\begin{subfigure}{0.17\textwidth}
\begin{center}
\includegraphics[trim={20cm 11cm 24cm 17cm}, clip,width=2.2cm]{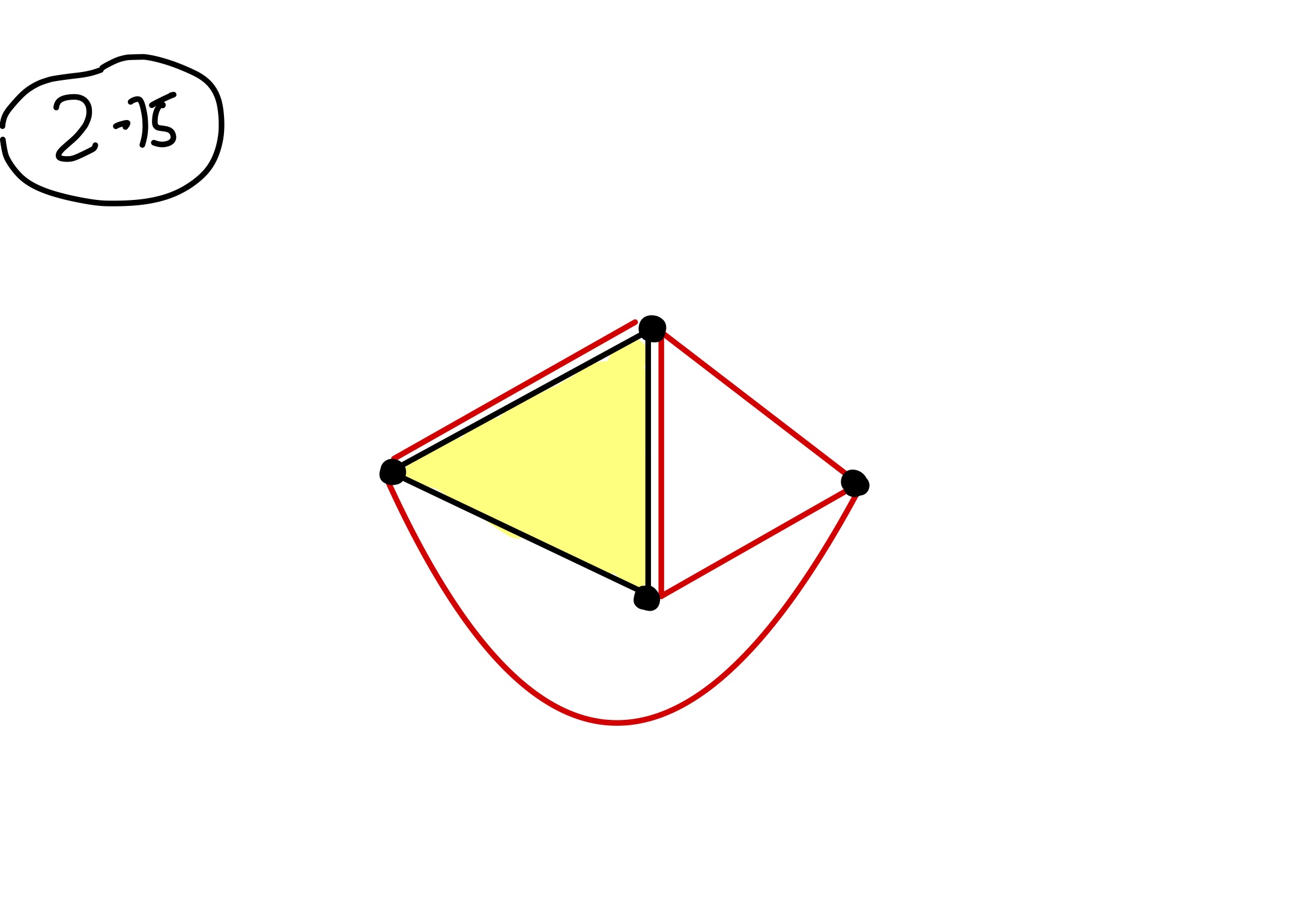}
\end{center}
\end{subfigure}
\begin{subfigure}{0.17\textwidth}
\begin{center}
\includegraphics[trim={21cm 5cm 17cm 15cm}, clip,width=2.2cm]{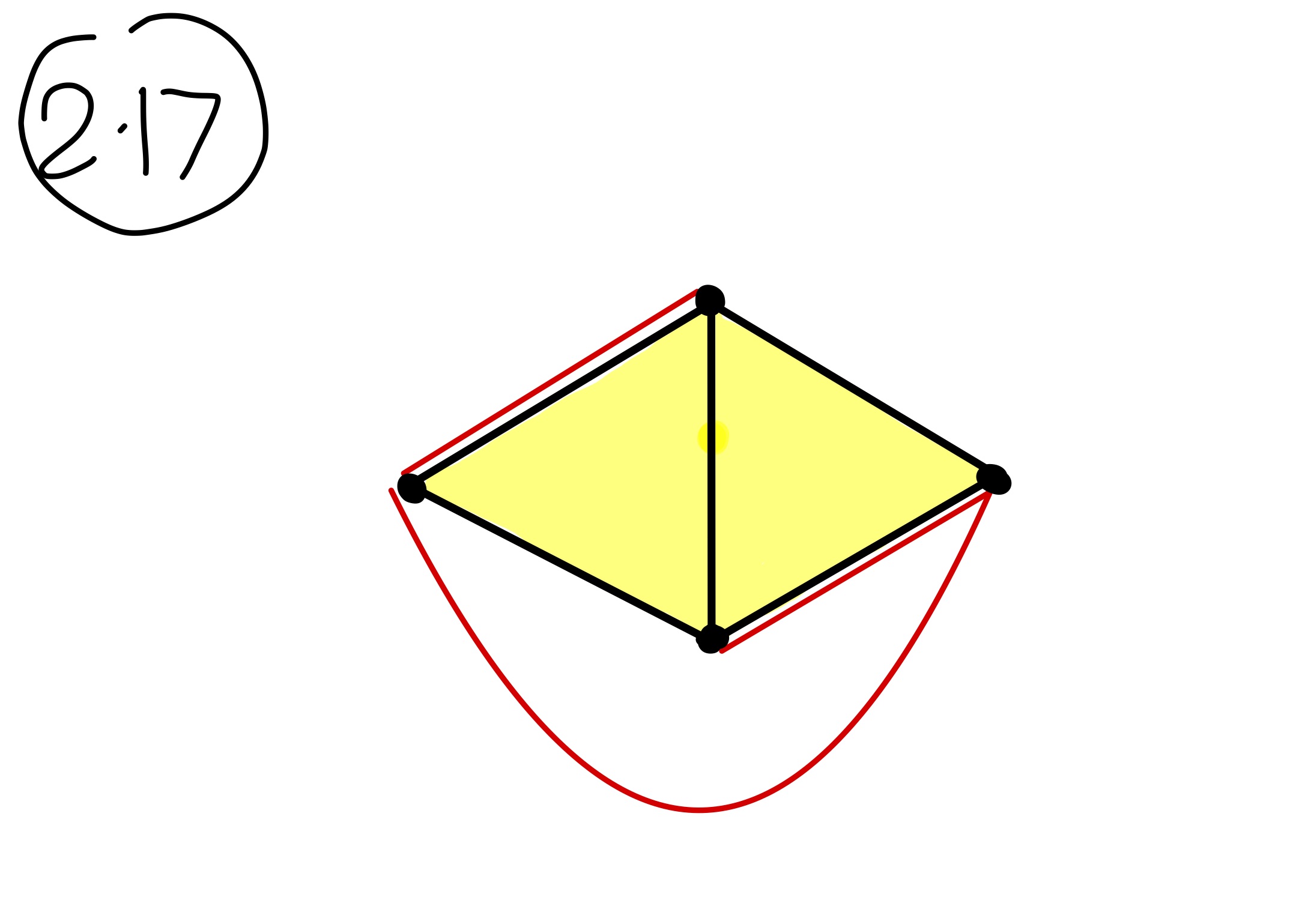}
\end{center}
\end{subfigure}
\begin{subfigure}{0.17\textwidth}
\begin{center}
\includegraphics[trim={20cm 20cm 18cm 11cm}, clip,width=2.2cm]{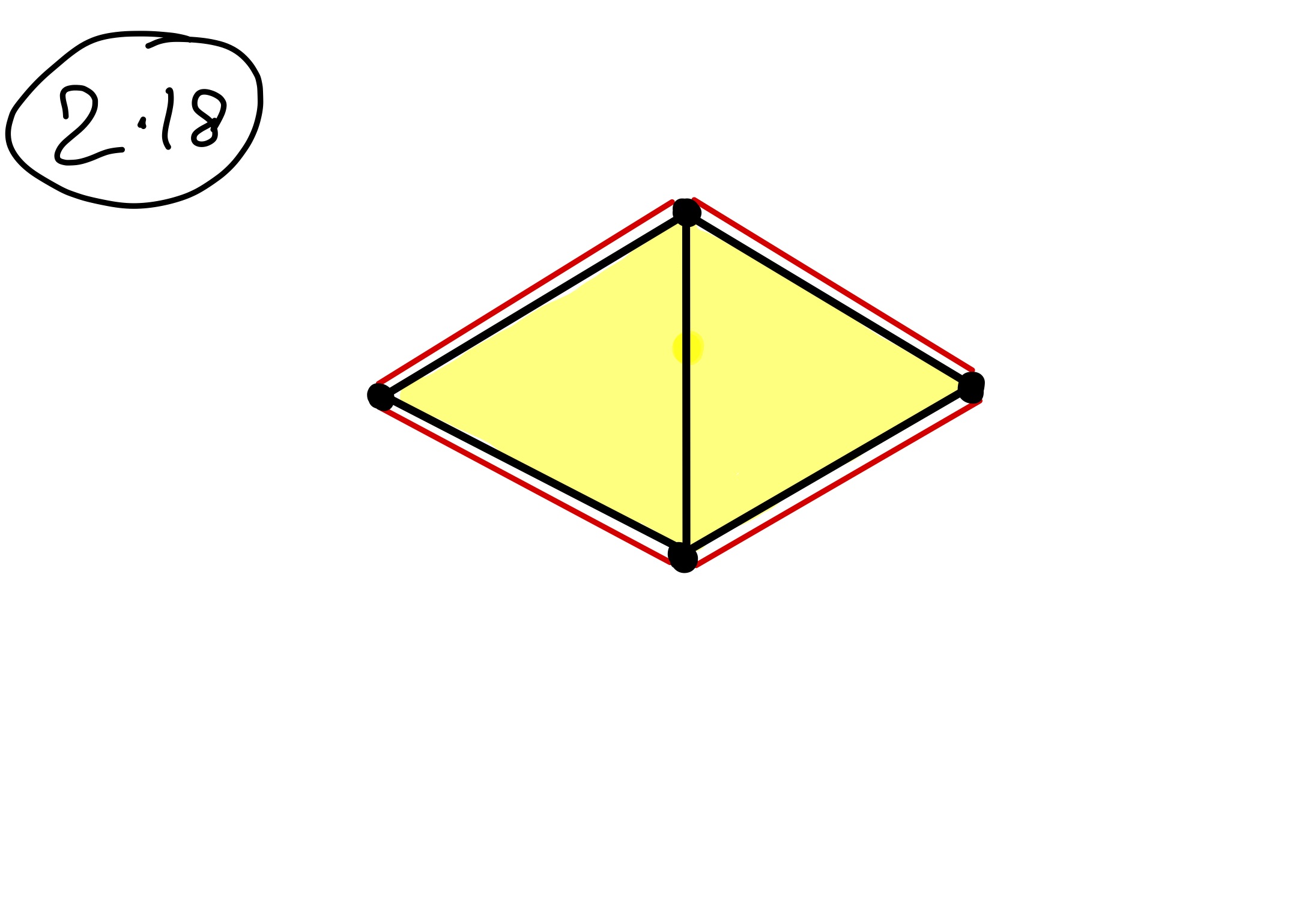}
\end{center}
\end{subfigure}

\begin{subfigure}{0.17\textwidth}
\begin{center}
\includegraphics[trim={15cm 17cm 27cm 9cm}, clip,width=2.2cm]{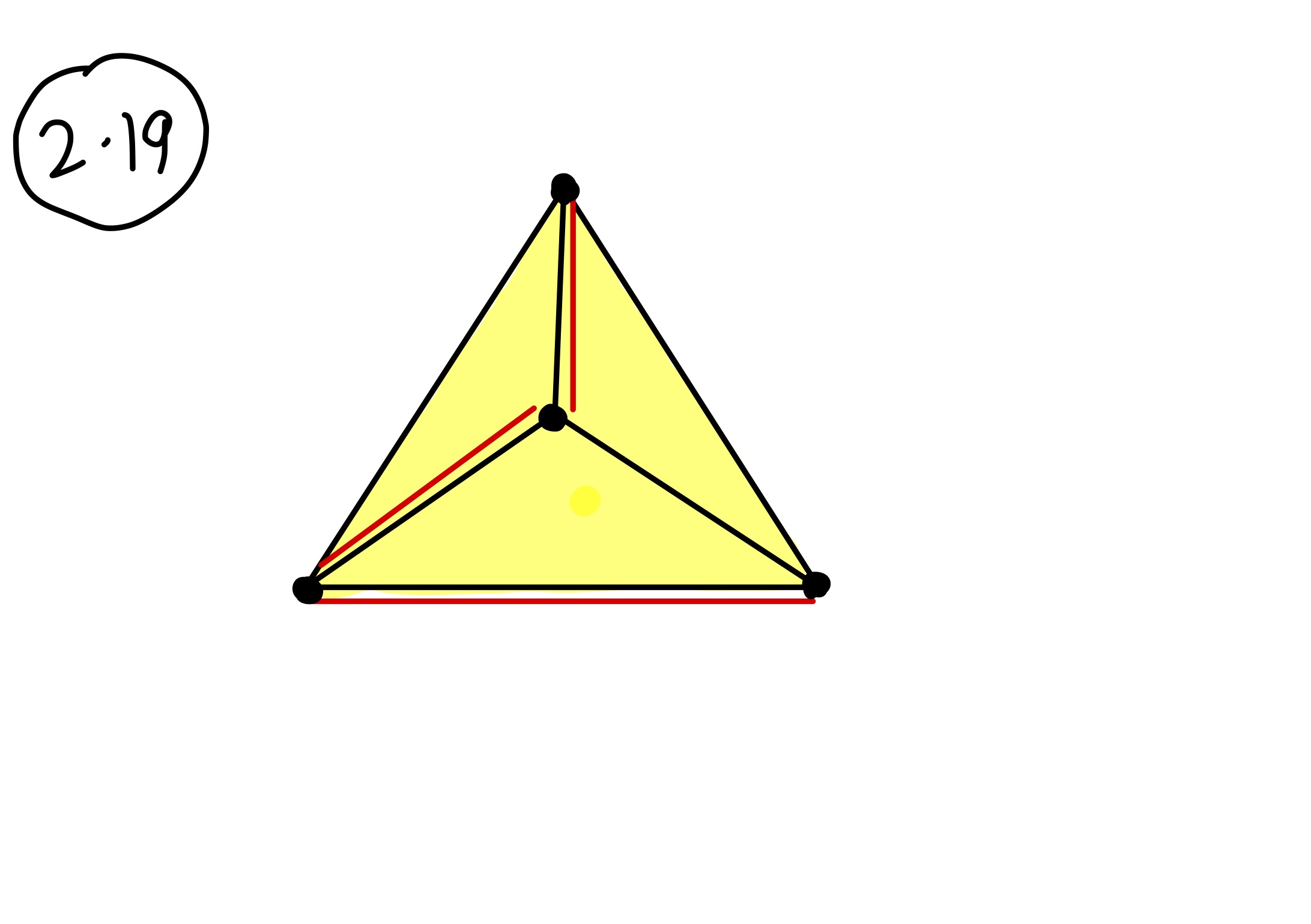}
\end{center}
\end{subfigure}
\begin{subfigure}{0.17\textwidth}
\begin{center}
\includegraphics[trim={19cm 11cm 16cm 13cm}, clip,width=2.2cm]{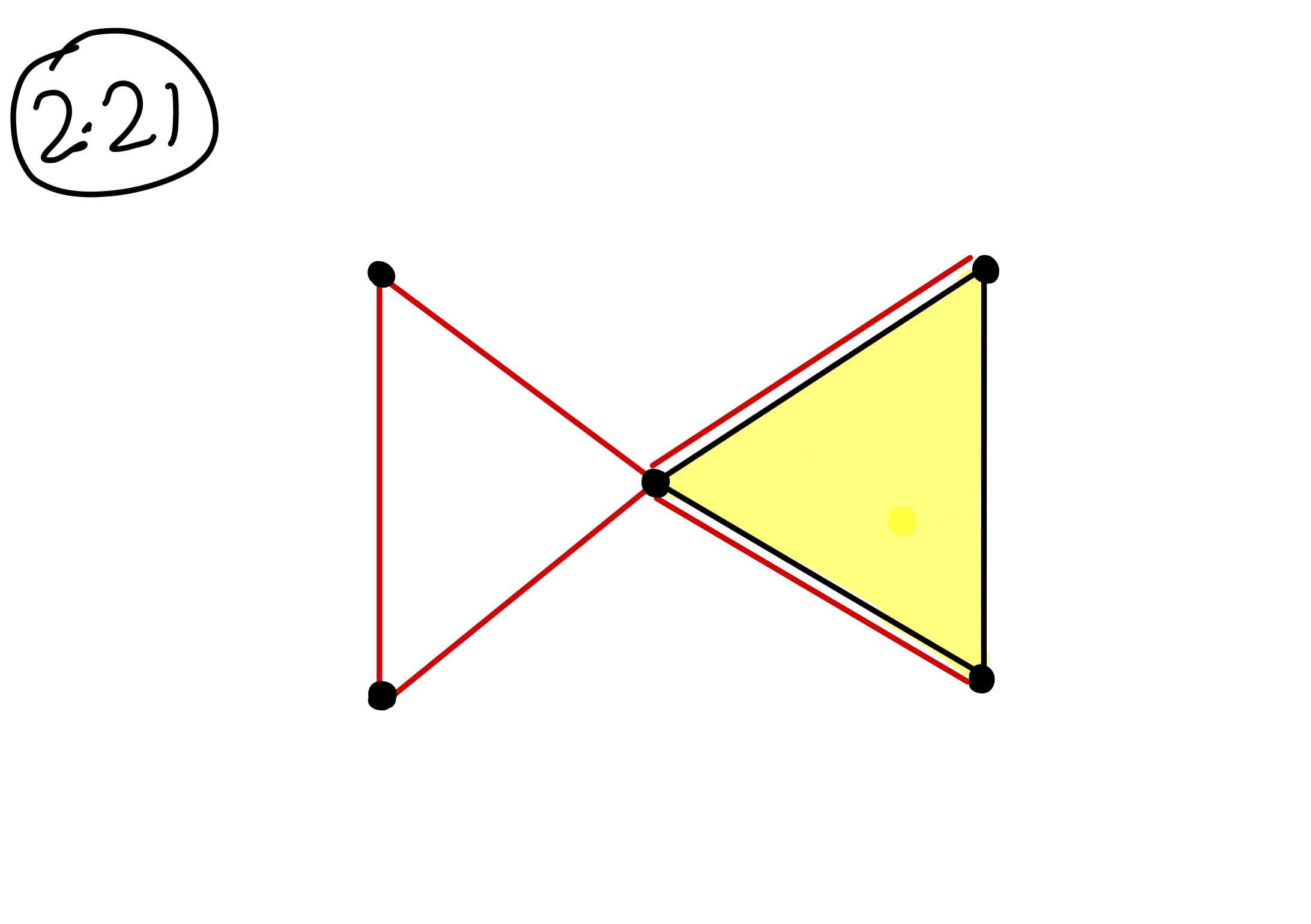}
\end{center}
\end{subfigure}
\begin{subfigure}{0.17\textwidth}
\begin{center}
\includegraphics[trim={15cm 13cm 23cm 11cm}, clip,width=2.2cm]{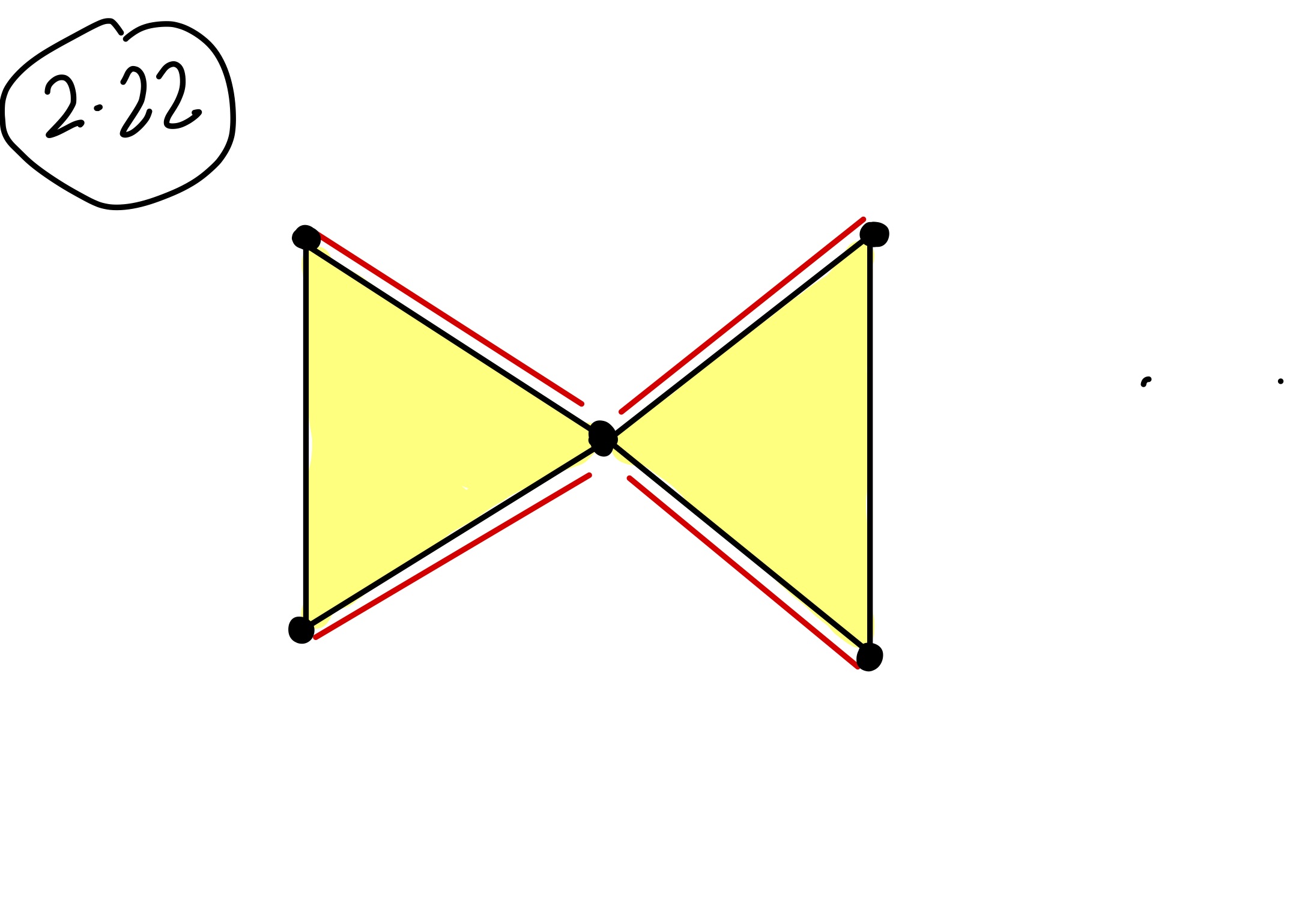}
\end{center}
\end{subfigure}
\begin{subfigure}{0.17\textwidth}
\begin{center}
\includegraphics[trim={23cm 22cm 28cm 10cm}, clip,width=2.2cm]{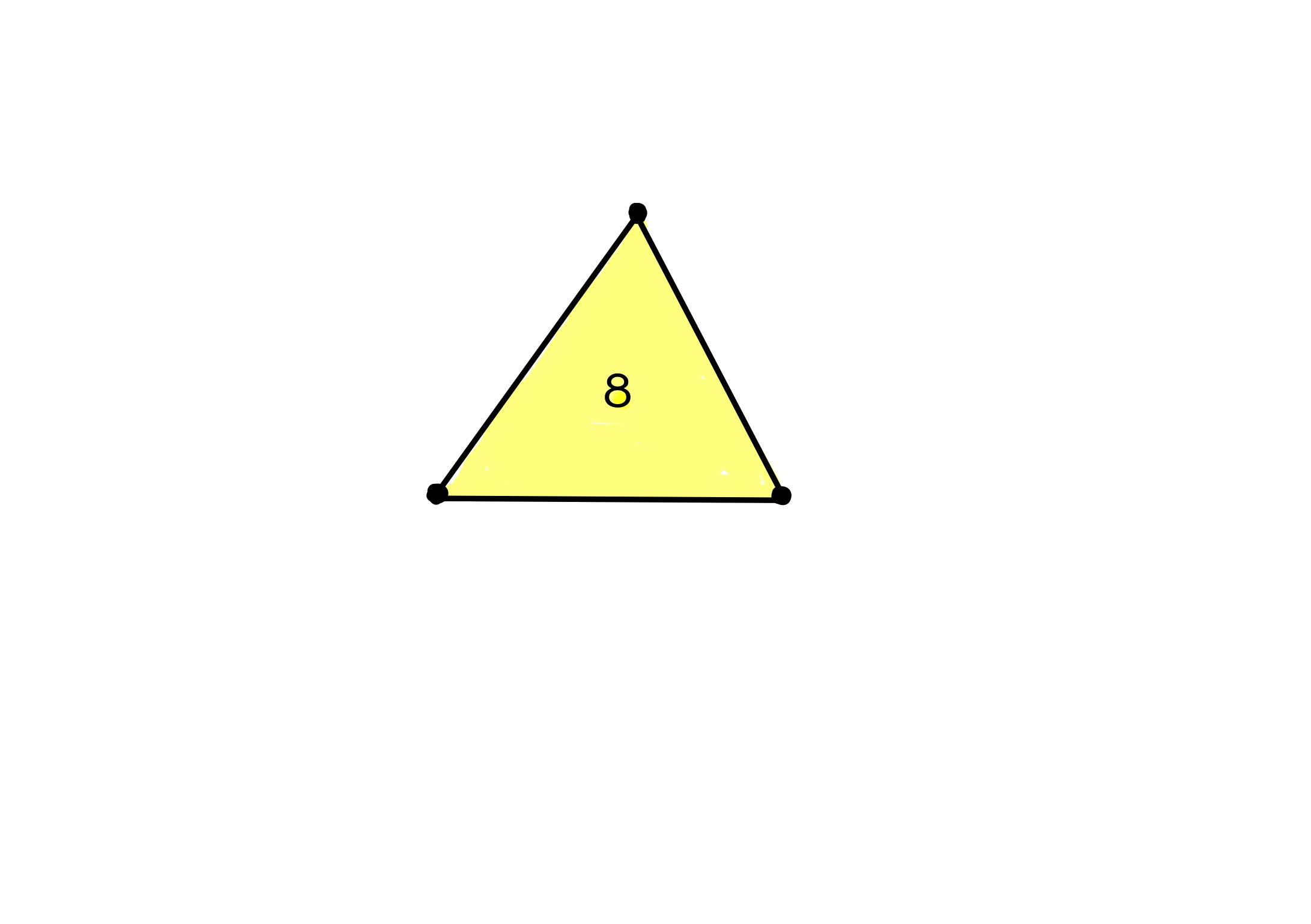}
\end{center}
\end{subfigure}
\begin{subfigure}{0.17\textwidth}
\begin{center}
\includegraphics[trim={21cm 22cm 22cm 13cm}, clip,width=2.2cm]{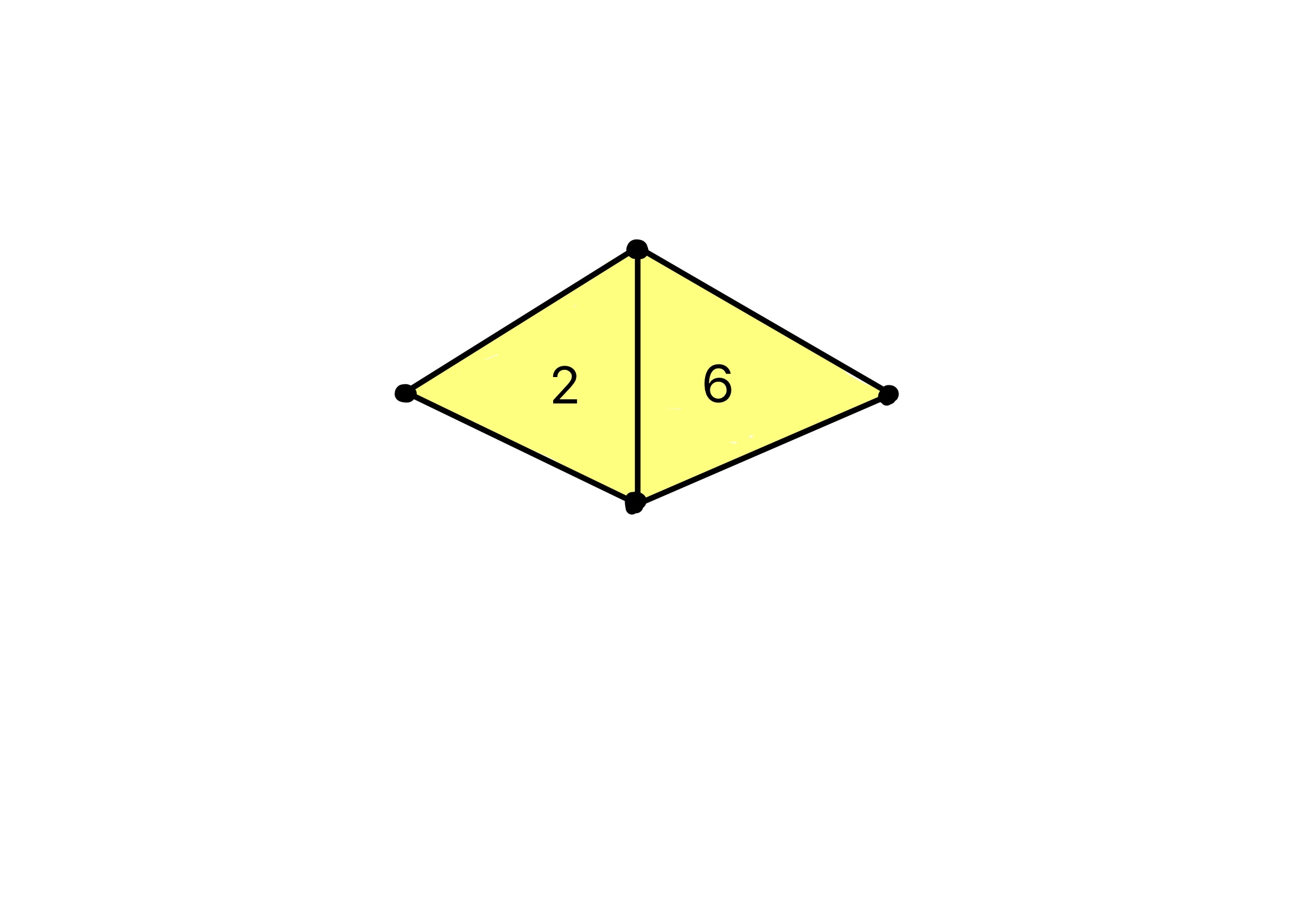}
\end{center}
\end{subfigure}

\begin{subfigure}{0.17\textwidth}
\begin{center}
\includegraphics[trim={21cm 19cm 22cm 16cm}, clip,width=2.2cm]{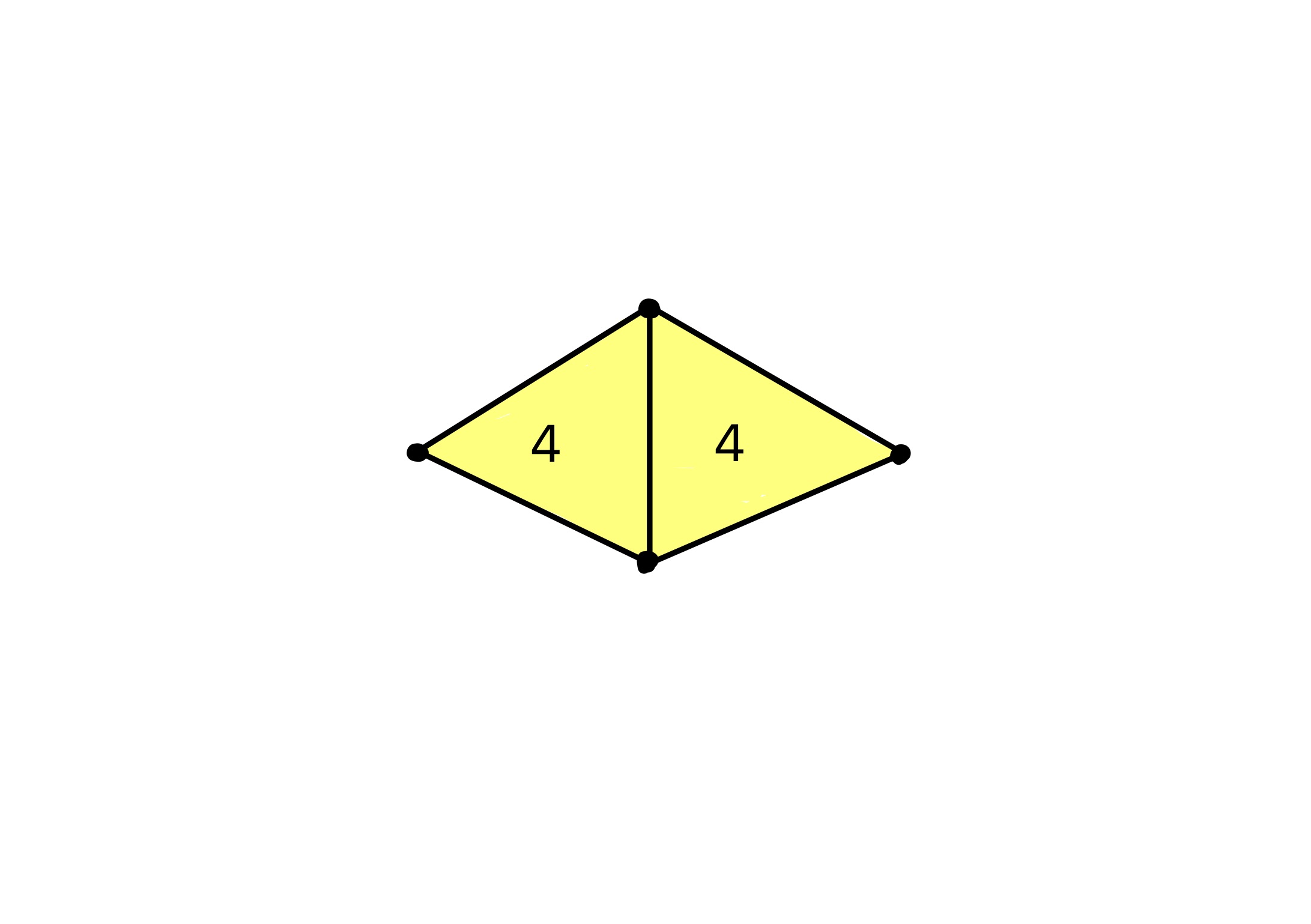}
\end{center}
\end{subfigure}
\begin{subfigure}{0.17\textwidth}
\begin{center}
\includegraphics[trim={23cm 22cm 24cm 12cm}, clip,width=2.2cm]{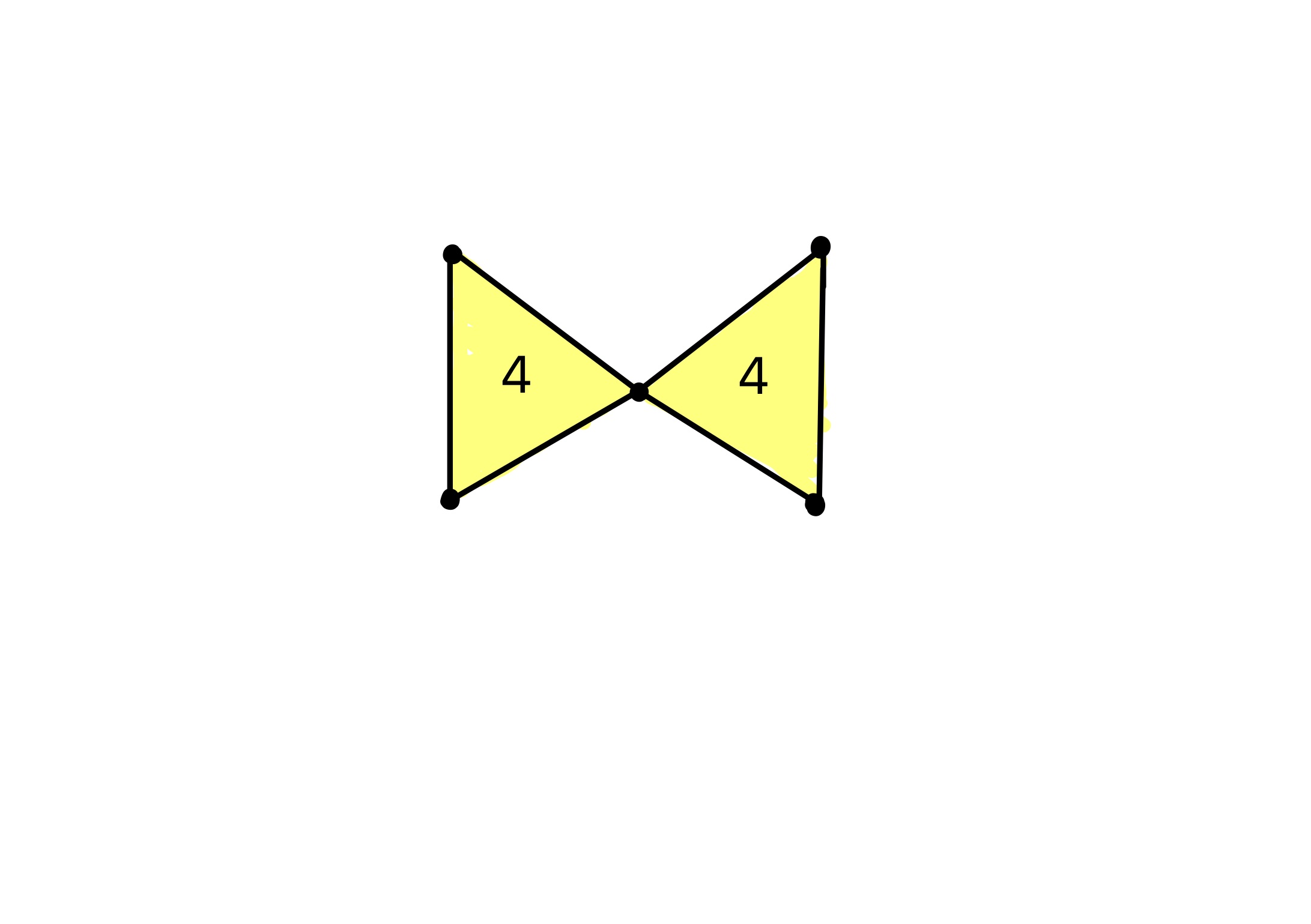}
\end{center}
\end{subfigure}
\begin{subfigure}{0.17\textwidth}
\begin{center}
\includegraphics[trim={19cm 16cm 26cm 10cm}, clip,width=2.2cm]{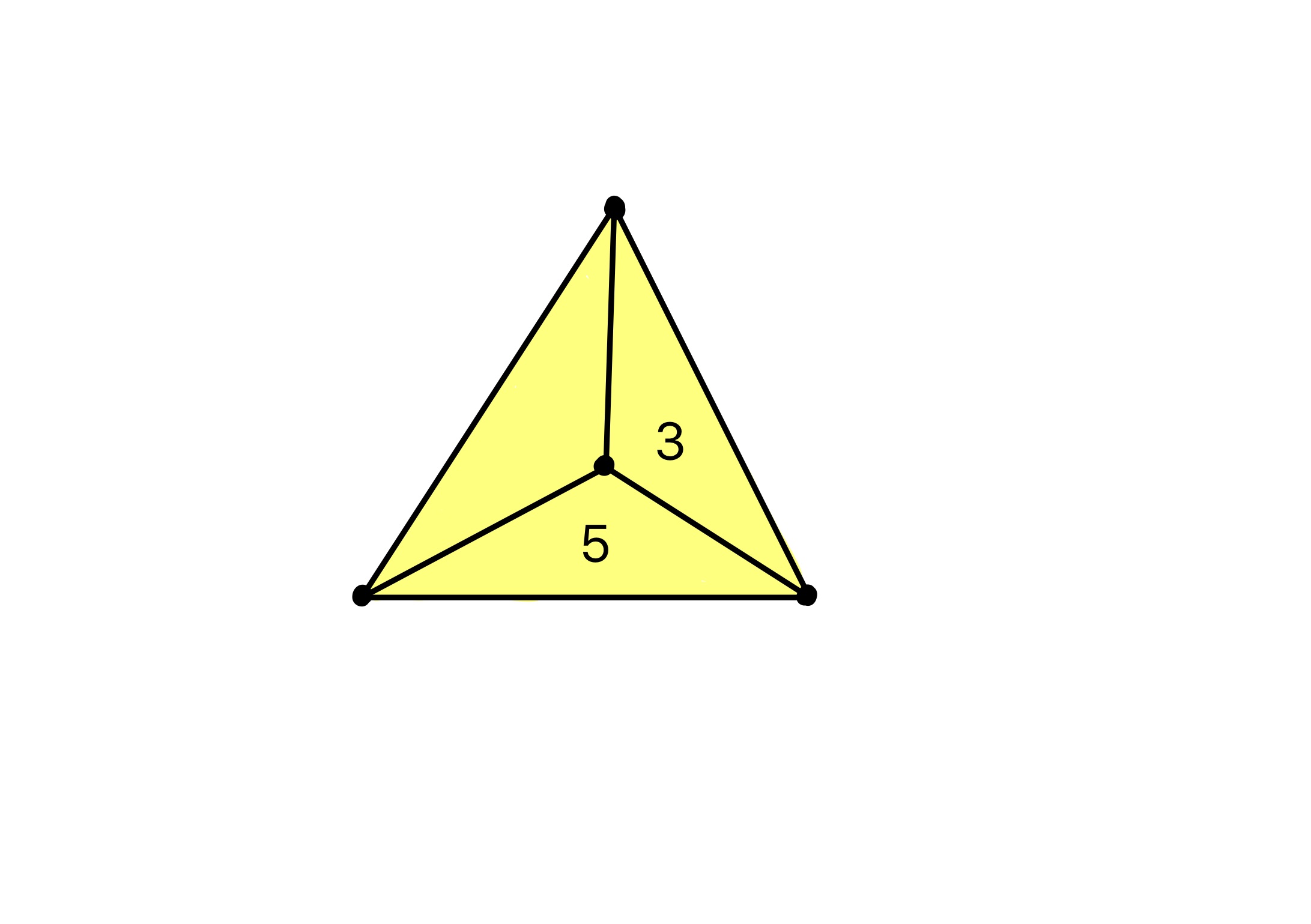}
\end{center}
\end{subfigure}
\caption{\label{fig-unsat-graphs}
A selection of known unsatisfiable looped-multi-hypergraphs.}
\end{figure}

\subsection{Minimality of unsatisfiable hypergraphs}
\label{sec:orgc6348d6}
In the case of multi-graphs, we showed in \cite{KaHi2020} that
satisfiability is invariant under homeomorphisms. We could then define a
minimal unsatisfiable multi-graph to be one which is unsatisfiable, with
every proper topological minor being totally satisfiable.

In the case of multi-hypergraphs, we have instead a list of reduction
rules. The reduction rules make it harder to define minimality due to
several reasons ---
\begin{enumerate}
\item Given a list of reduction rule, it is a computationally expensive to
check if any of these rules apply to a given graph.
\item For most reduction rules, the right side (which is what we obtain after
rewriting) is not a single graph --- it is instead a union of graphs.
This means that any notion of minimality for hypergraphs must
incorporate the effect of rewriting as a union of graphs instead of a
single graph.
\item Our list of reduction rules is not complete. We may find more reduction
rules by rewriting at higher degree vertices and carrying out longer
computations. Each additional reduction rule could make the minimality
criterion stricter and would shrink the size of any minimal set of
unsatisfiable hypergraphs.
\item Uniqueness of the minimal set of unsatisfiable hypergraphs is not
guaranteed owing to the complexity of the reduction rules.
\end{enumerate}

\subsection{Computational logistics}
\label{sec:org7cb9d3b}
In this section we discuss the logistical setup used for carrying out all
the calculations in this paper, along with a mention of the challenges
posed by continuing these computations on bigger graphs in the face of
exponentially more cases that need checking.

A computational procedure for finding all unsatisfiable
looped-multi-hypergraphs can be carried out as follows ---
\begin{description}
\item[{\textbf{Step 1.}}] Start with all looped-multi-hypergraphs sorted from smallest
to largest. This can be done by calling \texttt{nauty} from inside
SageMath. We use \texttt{nauty} to generate all graphs with the
following properties inside a specified vertex range ---
\begin{itemize}
\item the graph must be connected.
\item total number of vertices must lie within specified range.
\item edge sizes can be 1 (loops), 2 (simple edges), or 3 (hyperedges).
\item we disallow edges of size 4 or higher in order to keep the
computational task tractable.
\item we specify that the minimum vertex degree of the graph should be 2
(since leaf vertices are known to be reducible).
\item we allow an edge of size \(k\) to only have multiplicity less than
\(2^k\).
\end{itemize}
\item[{\textbf{Step 2.}}] Pick a graph and apply all known reduction rules to it.
\item[{\textbf{Step 3.}}] Sat-check the irreducible part of the graph left over from
Step 2 using brute-force strategy.
\item[{\textbf{Step 4.}}] If totally satisfiable, then pick the next graph and go back to
Step 2.
\item[{\textbf{Step 5.}}] If unsatisfiable, then add the irreducible part to the list
of known ``minimal criminals''.
\end{description}

While this procedure allows us to search for small unsatisfiable graphs, it
is clear that we have to contend with an exponential blowup in the number
of graphs as well as an exponential blowup in the number of Cnfs that need
to be sat-checked as we keep increasing the vertex count. Table
\ref{table-exponential-blowup-in-graphs} tabulates the number of graphs for
different vertex ranges.

\begin{table}[H]
\caption{Exponential blowup in the number of graphs with increasing vertex count.}
\begin{center}
\begin{tabular}{l l}
\hline
Number of connected simple graphs with less than 7 vertices & 143\\
Number of minimal unsatisfiable irreducible simple graphs & 4\\
 & \\
Number of connected looped-multi-hypergraphs with less than 6 vertices & 10080\\
Number of minimal unsatisfiable irreducible L-M-H-graphs & 202\\
 & \\
Number of connected looped-multi-hypergraphs with less than 7 vertices & 48,364,386\\
Number of minimal unsatisfiable irreducible L-M-H-graphs & unknown\\
\hline
\end{tabular}
\end{center}
\label{table-exponential-blowup-in-graphs}
\end{table}

\section{Satisfiability of triangulations}
\label{sec:org5b66ec3}
In this section we check a list of common/standard hypergraphs having edges
of size exactly \(3\). These hypergraphs can be drawn as triangulations of
various surfaces and are thus of interest when viewing \GraphSAT from a
topological point of view. We present computationally obtained
satisfiability and unsatisfiability results. The choice of structures we
study here is less systematic and more driven by ease of calculation.

\subsection{Thickening of graph edges}
\label{sec:org7e741ad}
We outline below a method to create triangulations starting with simple
graphs, such that the satisfiability status in going from the simple graph
to the triangulation remains unchanged.

The process involves a local rewrite of every simple edge \(\g{ab}\) in a
graph with the hyperedges \(\g{abc\wedge acd\wedge bcd}\), where \(\g{c}\) and
\(\g{d}\) are new vertices not previously appearing in the graph.

For example, since \(\g{ab^4}\) is an unsatisfiable graph, we can thicken
all its edges into hyperedges to form the triangulation
\[\g{ab1\wedge a12\wedge b12\wedge ab3\wedge a34\wedge b34\wedge ab5\wedge a56\wedge b56\wedge ab7\wedge a78\wedge b78}.\] This thickening
process is shown in Figure \ref{fig-thickening_of_ab4}.

\begin{figure}[h]
\centering
\begin{subfigure}{0.15\textwidth}
\begin{center}
\includegraphics[trim={22cm 30cm 34cm 17cm}, clip,width=2cm]{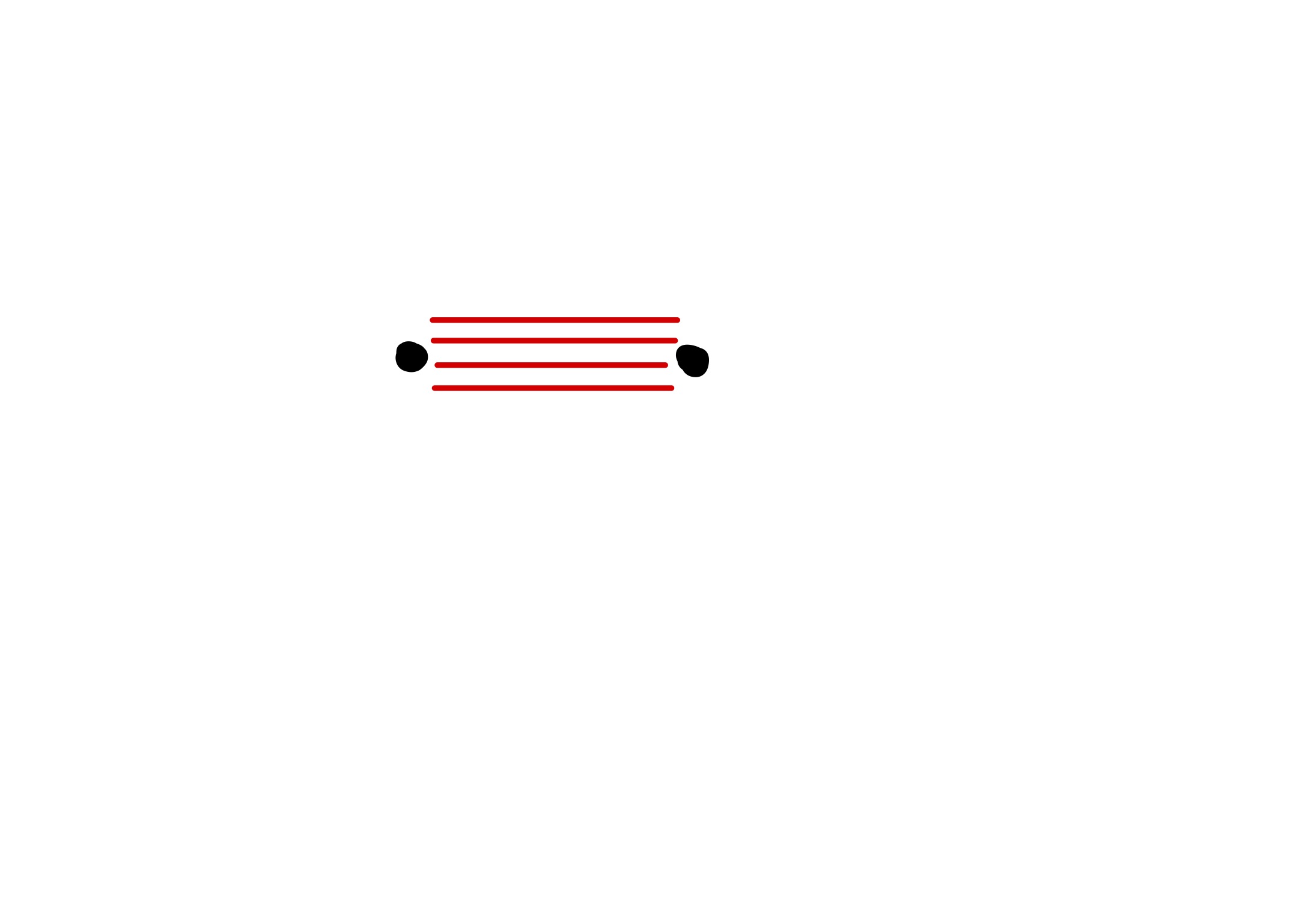}
\end{center}
\end{subfigure}
\begin{subfigure}{0.15\textwidth}
\begin{center}
\includegraphics[trim={21cm 20cm 36cm 08cm}, clip,width=2cm]{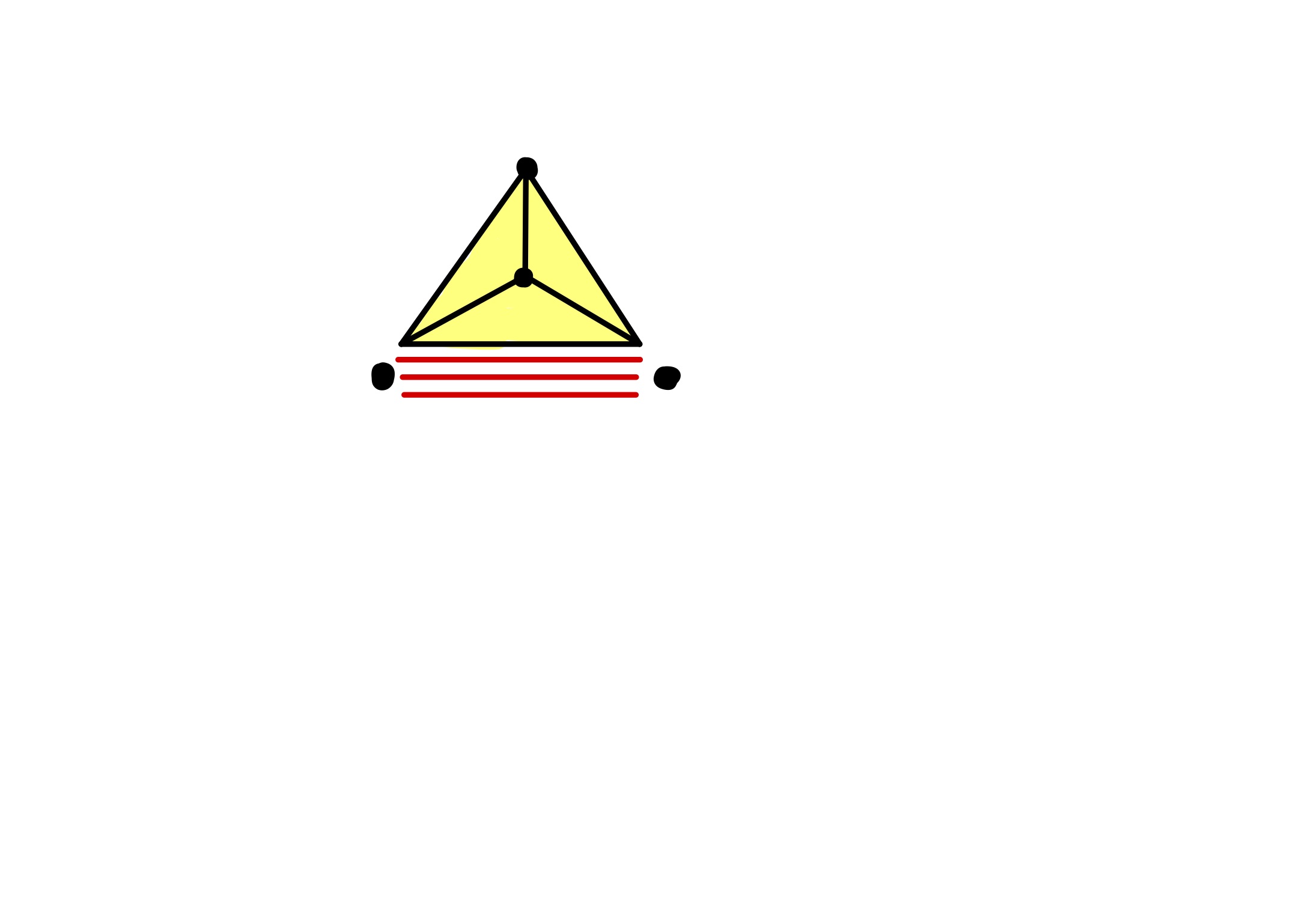}
\end{center}
\end{subfigure}
\begin{subfigure}{0.15\textwidth}
\begin{center}
\includegraphics[trim={29cm 20cm 29cm 7cm}, clip,width=2cm]{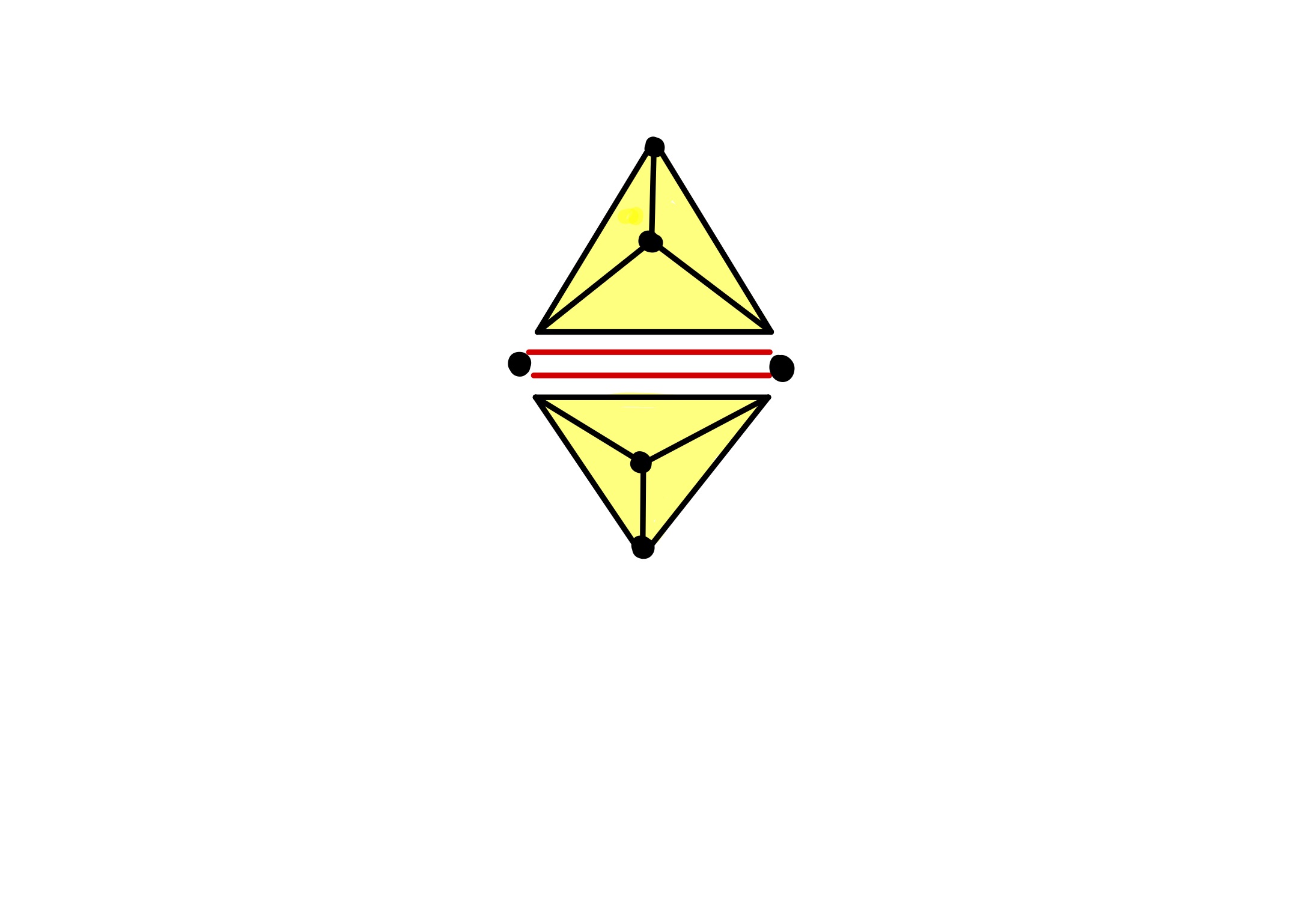}
\end{center}
\end{subfigure}
\begin{subfigure}{0.2\textwidth}
\begin{center}
\includegraphics[trim={11cm 15cm 36cm 11cm}, clip,width=3cm]{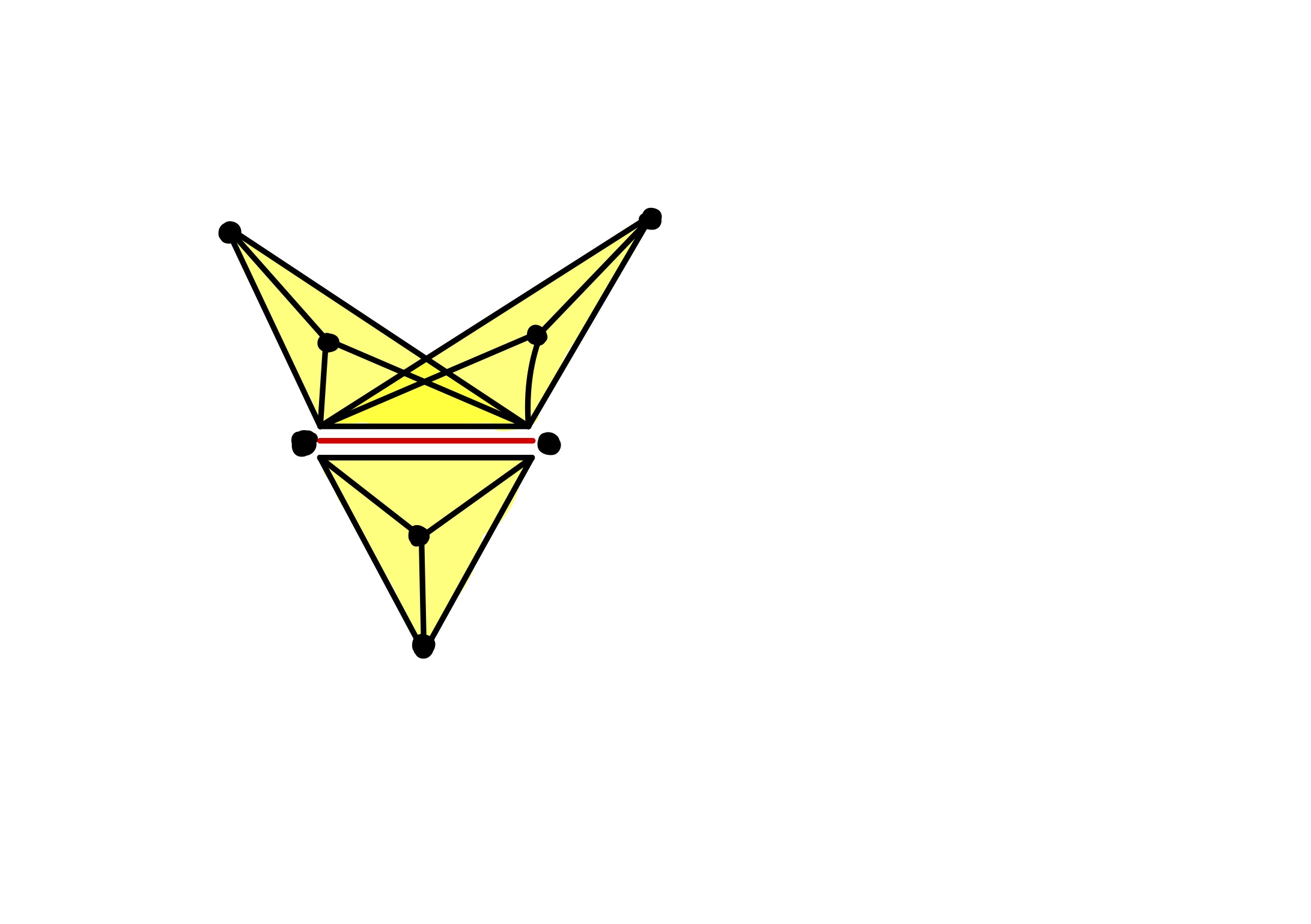}
\end{center}
\end{subfigure}
\begin{subfigure}{0.15\textwidth}
\begin{center}
\includegraphics[trim={15cm 14cm 20cm 5cm}, clip,width=3.5cm]{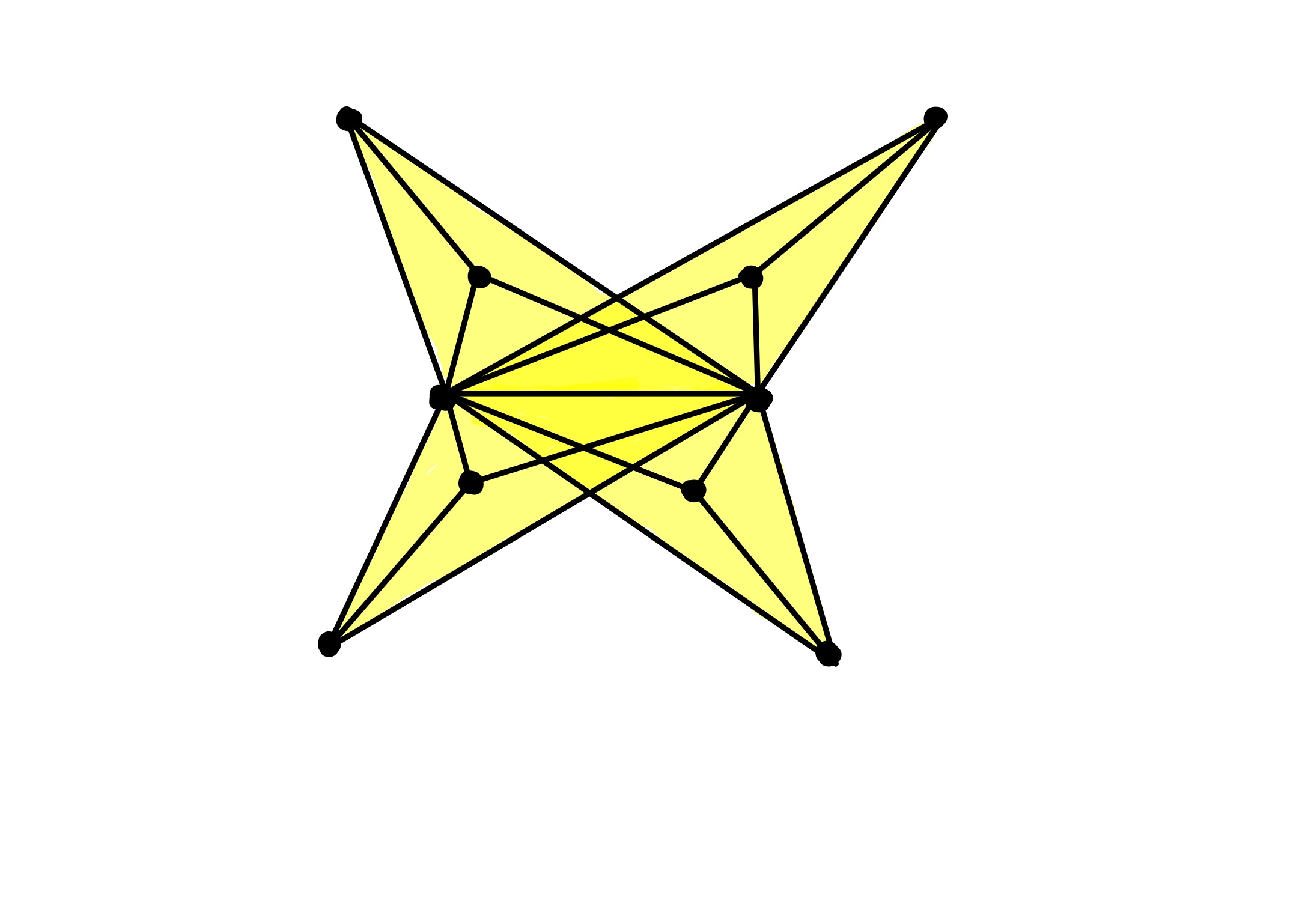}
\end{center}
\end{subfigure}
\caption{\label{fig-thickening_of_ab4}
Thickening of the edges of \(\g{ab^4}\) shown step-by-step.}
\end{figure}

Similarly, thickening of the unsatisfiable graph \(\g{ab^2\wedge bc^2}\)
results in the unsatisfiable triangulation shown in Figure
\ref{fig-thickening_of_ab2bc2}. This triangulation is planar and can therefore be
embedded in any surface. Thus, every surface has an unsatisfiable
triangulation.

\begin{figure}[h]
\centering
\begin{subfigure}{0.2\textwidth}
\begin{center}
\includegraphics[trim={17cm 29cm 30cm 18cm}, clip,width=2.5cm]{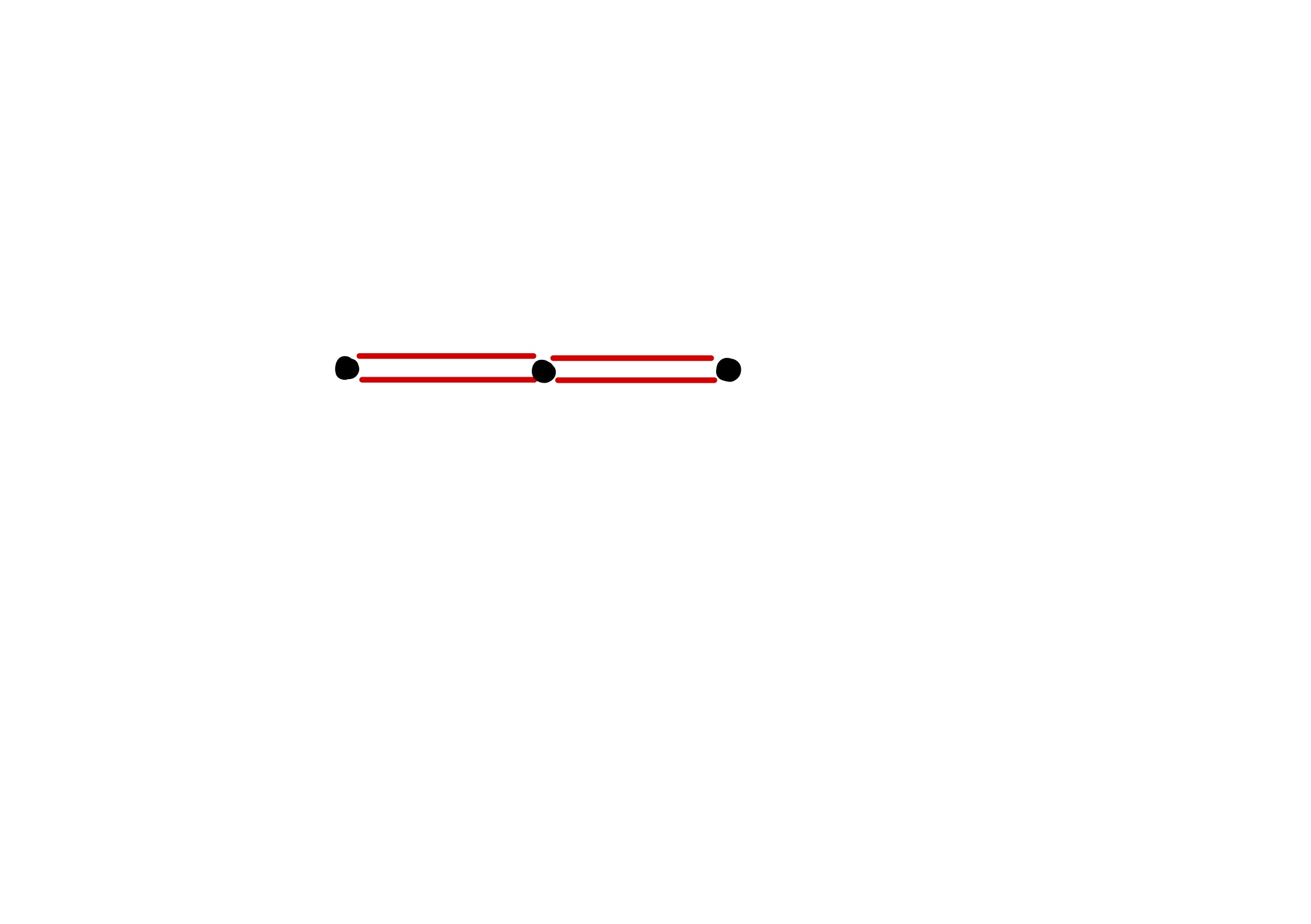}
\end{center}
\end{subfigure}
\begin{subfigure}{0.2\textwidth}
\begin{center}
\includegraphics[trim={14cm 17cm 25cm 12cm}, clip,width=3cm]{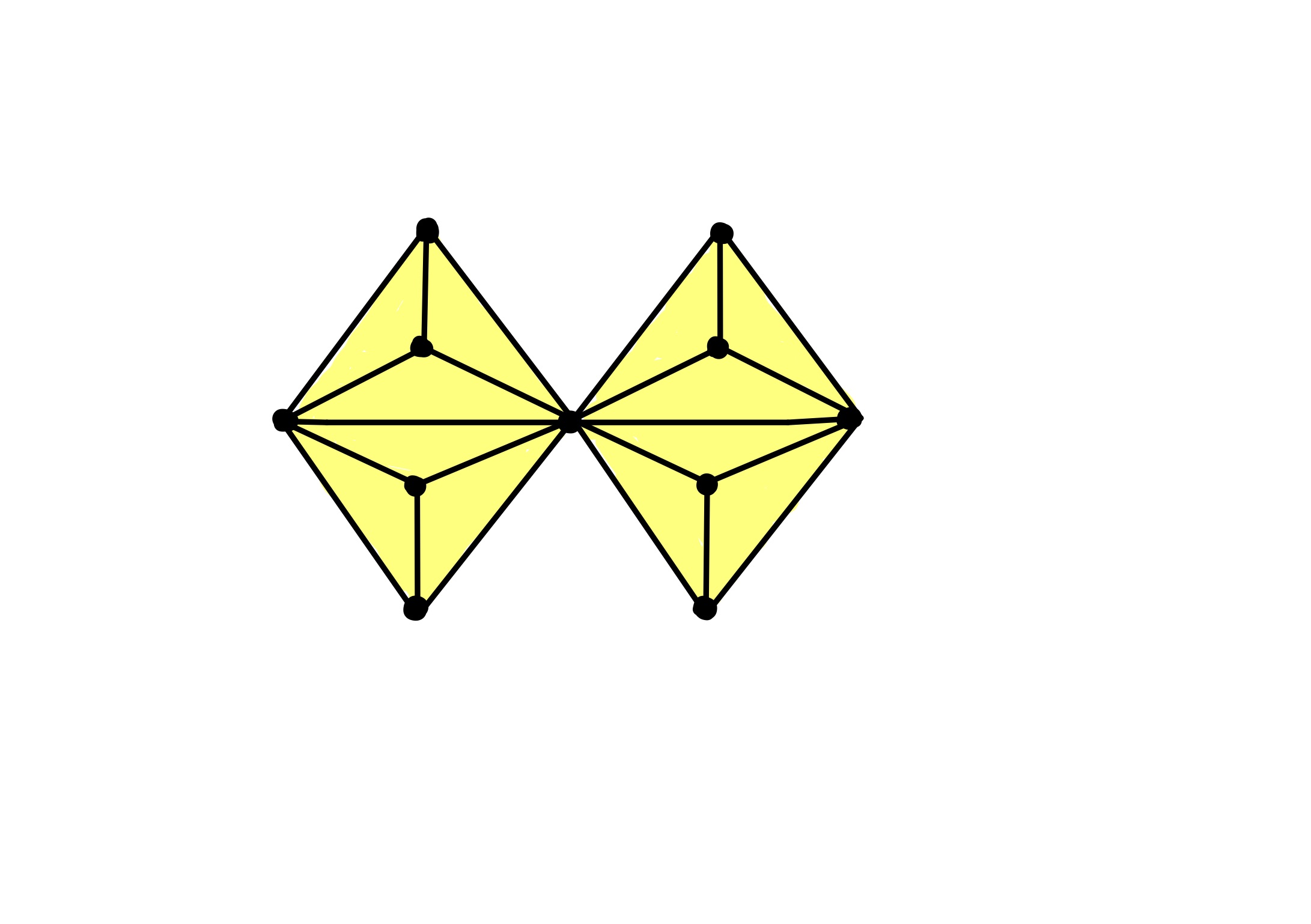}
\end{center}
\end{subfigure}
\caption{\label{fig-thickening_of_ab2bc2}
Thickening of the edges of \(\g{ab^2\wedge bc^2}\).}
\end{figure}

We also note that not all unsatisfiable triangulations are thickenings of
unsatisfiable simple graphs. To prove that graph satisfiability is
invariant under thickening of edges, we observe the following ---
\begin{align*}
  \g{s\wedge 123\wedge 134\wedge 234} &\bsim \g{s\wedge 123^2} &(\text{using edge-smoothing reduction rule})\\
                    &\bsim \g{s\wedge 12}    &(\text{using leaf vertex reduction rule})
\end{align*}
These reduction rules can be applied only because \(\g{s}\) does not have
any edges incident on vertices \(\g{3}\) and \(\g{4}\) since there vertices
are newly introduced by the thickening process.

\subsection{Tetrahedron and prisms}
\label{sec:orga829bde}
The tetrahedron's wire-frame structure, i.e. its edges form the graph
\(\g{K_4}\), a known unsatisfiable graph. The faces form the triangulation
\(\g{abc\wedge acd\wedge abd\wedge bcd}\). Using the reduction rule from \S \ref{sec:org6227f90} gives

\begin{align*}
\text{Tetrahedron}
    &\bsim\; \g{abc\wedge ab} \;\cup\; \g{abc\wedge ac} \;\cup\; \g{abc\wedge bc} \\
    &\bsim\; \g{ab} \;\cup\; \g{ac} \;\cup\; \g{bc} \\
    &\bsim\; \btop \;\cup\; \btop \;\cup\; \btop \\
    &=\; \btop
\end{align*}

Thus the tetrahedron is totally satisfiable.

On the other hand the triangular prism has two possible minimal
triangulations ---
\begin{enumerate}
\item The symmetric triangulation, given by
\[\g{123\wedge 125\wedge 134\wedge 145\wedge 236\wedge 256\wedge 346\wedge 456}.\]
\item The asymmetric triangulation, given by
\[\g{123\wedge 125\wedge 136\wedge 145\wedge 146\wedge 236\wedge 256\wedge 456}.\]
\end{enumerate}

These triangulations are shown in Figure \ref{fig-prisms}. Passing them to the
\texttt{decompose} function from the \texttt{graph\_rewite} module tells us that both
triangulations are totally satisfiable.

\begin{figure}[h]
\centering
\begin{subfigure}{0.4\textwidth}
\begin{center}
\includegraphics[trim={15cm 11cm 20cm 9cm}, clip,width=5cm]{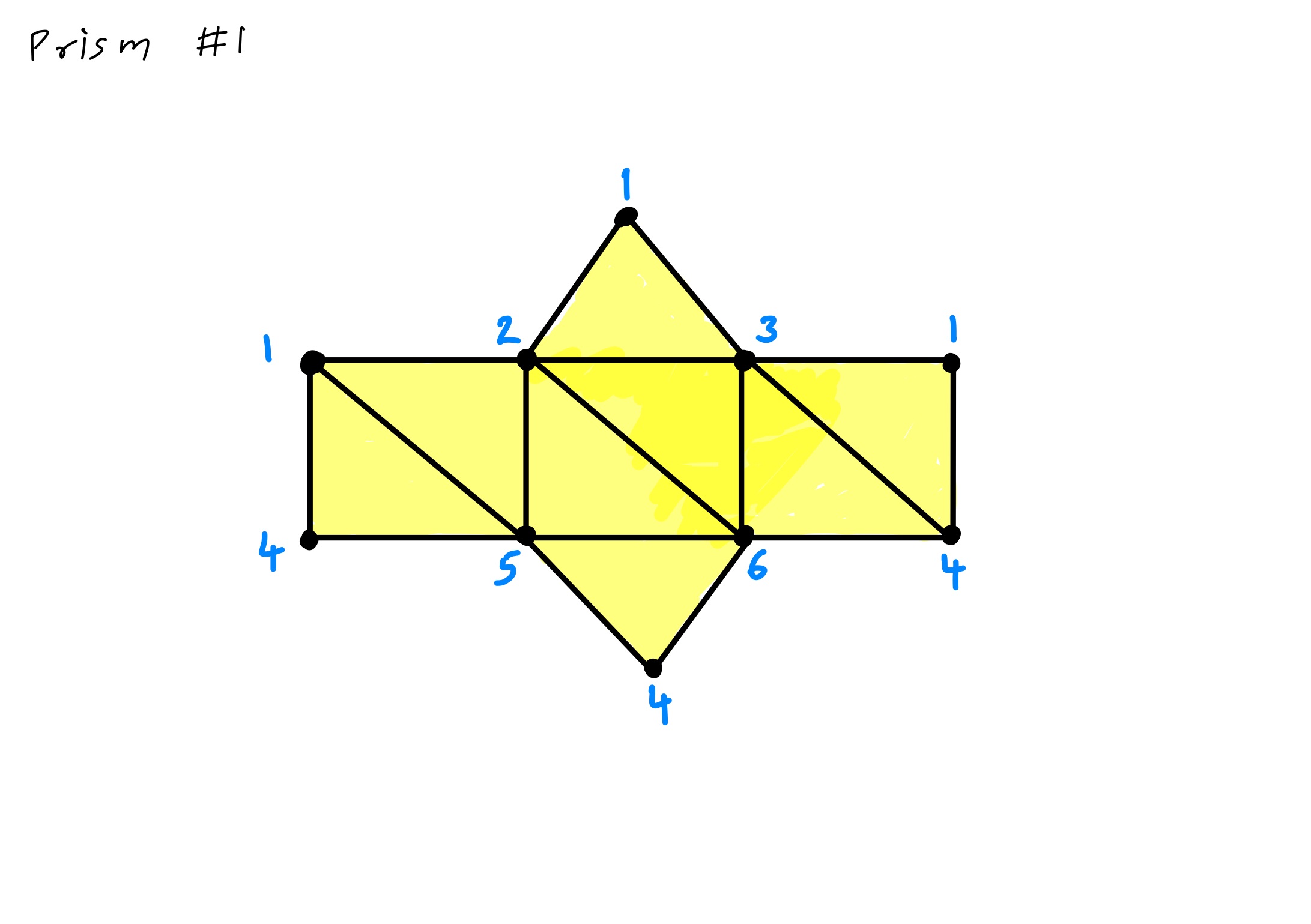}
\end{center}
\caption{The symmetric triangulation.}
\end{subfigure}
\begin{subfigure}{0.4\textwidth}
\begin{center}
\includegraphics[trim={13cm 8cm 21cm 10cm}, clip,width=5cm]{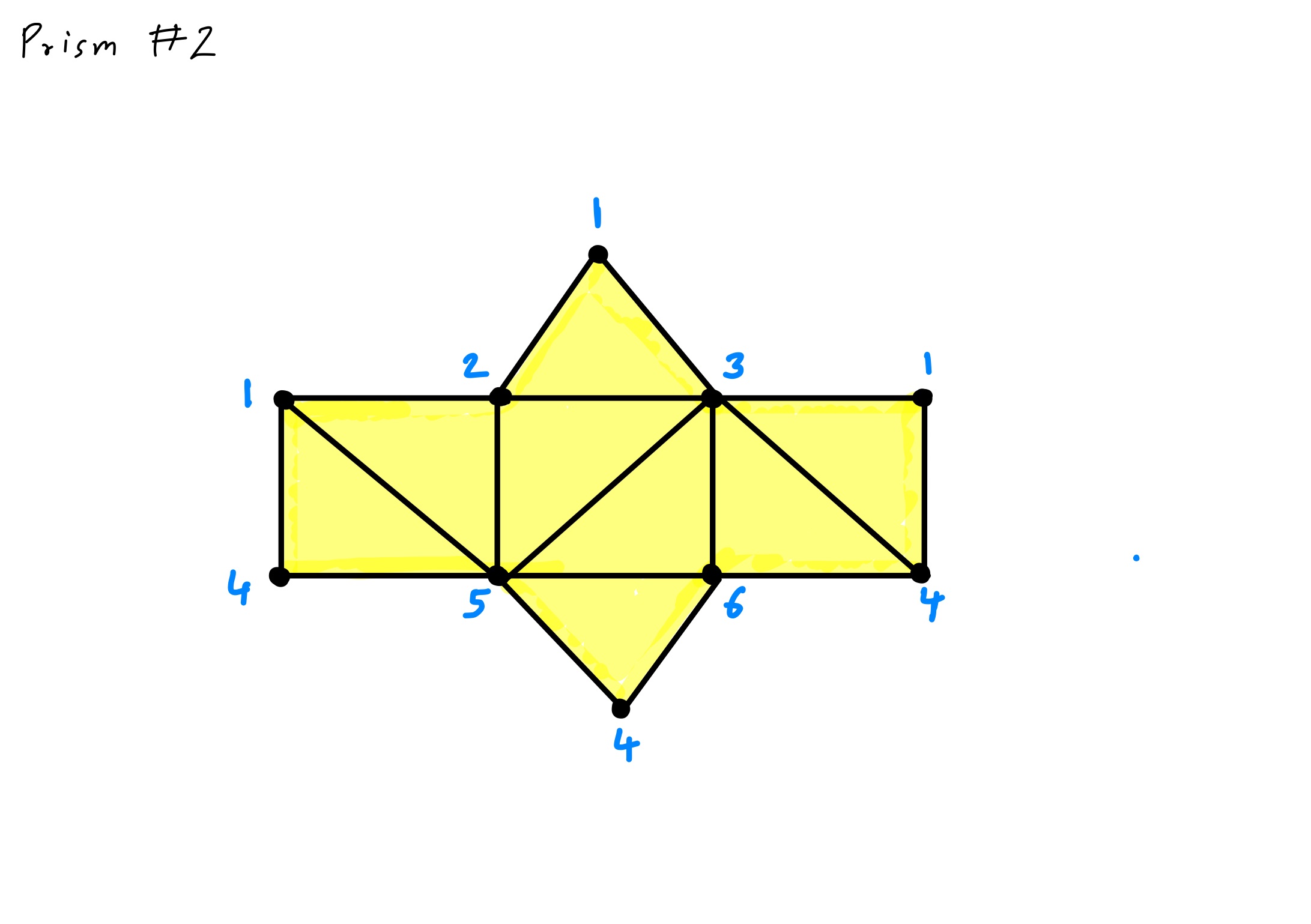}
\end{center}
\caption{The asymmetric triangulation.}
\end{subfigure}
\caption{\label{fig-prisms}
Minimal triangulations of a triangular prism}
\end{figure}

\begin{minted}[frame=lines,label= (python3.9) (scratch) <<prism-calculation>>]{python}
import mhgraph as mhg
import graph_rewrite as grw

prism1: mhg.MHGraph = mhg.mhgraph([[1,2,3], [1,2,5], [1,3,4], [1,4,5],
				   [2,3,6], [2,5,6], [3,4,6], [4,5,6]])
grw.decompose(prism1)

prism2: mhg.MHGraph = mhg.mhgraph([[1,2,3], [1,2,5], [1,3,6], [1,4,5],
				   [1,4,6], [2,3,6], [2,5,6], [4,5,6]])
grw.decompose(prism2)
\end{minted}
\footnotesize
\color{darkgray}
\uline{Output}:
\begin{verbatim}
(1, 2, 3)¹,(1, 2, 5)¹,(1, 3, 4)¹,(1, 4, 5)¹,(2, 3, 6)¹,(2, 5, 6)¹,(3, 4, 6)¹,(4, 5, 6)¹
    is SAT
(1, 2, 3)¹,(1, 2, 5)¹,(1, 3, 6)¹,(1, 4, 5)¹,(1, 4, 6)¹,(2, 3, 6)¹,(2, 5, 6)¹,(4, 5, 6)¹
    is SAT
\end{verbatim}
\normalsize
\color{black}

\subsection{Triangulation of a Möbius strip}
\label{sec:orga6c741c}
A Möbius strip can be triangulated as \[\g{124\wedge 146\wedge 235\wedge
245\wedge 346\wedge 356}\] (also shown in Figure \ref{fig-mobius-strip}). Using
\texttt{graph\_rewrite.decompose}, we conclude that this triangulation is totally
satisfiable.

\begin{figure}[h]
\centering
\begin{center}
\includegraphics[trim={8cm 14cm 5cm 12cm}, clip,width=7cm]{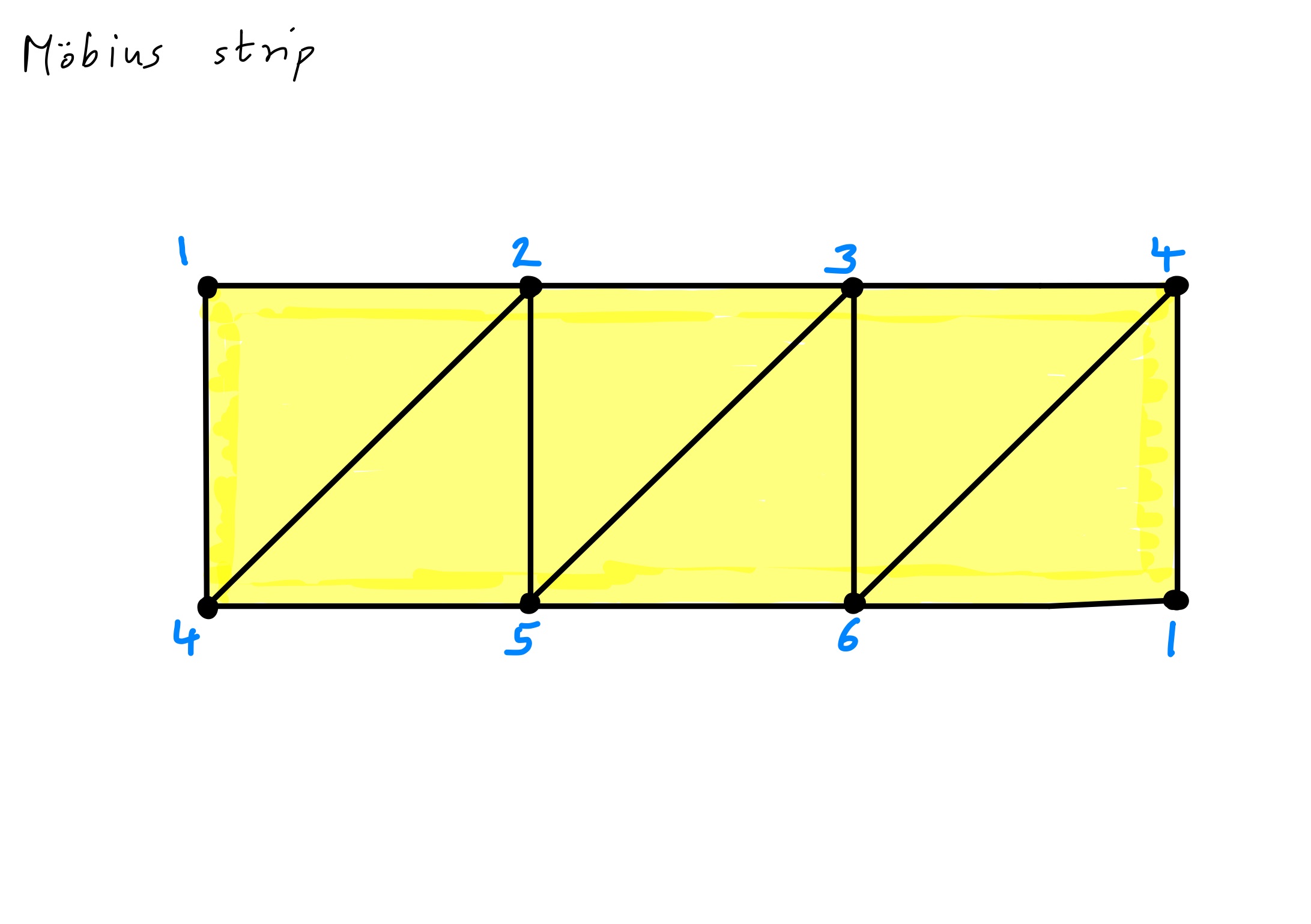}
\end{center}
\caption{\label{fig-mobius-strip}
Triangulation of a Möbius strip}
\end{figure}

\begin{minted}[frame=lines,label= (python3.9) (scratch) <<mobius>>]{python}
import mhgraph as mhg
import graph_rewrite as grw

mobius_strip: mhg.MHGraph = mhg.mhgraph([[1, 2, 4], [1, 4, 6], [2, 3, 5],
					 [2, 4, 5], [3, 4, 6], [3, 5, 6]])
grw.decompose(mobius_strip)
\end{minted}
\footnotesize
\color{darkgray}
\uline{Output}:
\begin{verbatim}
(1, 2, 4)¹,(1, 4, 6)¹,(2, 3, 5)¹,(2, 4, 5)¹,(3, 4, 6)¹,(3, 5, 6)¹ is SAT
\end{verbatim}
\normalsize
\color{black}

\subsection{Minimal triangulation of the real projective plane}
\label{sec:org249d63e}
The minimal triangulation of \(\mathbb{RP}^2\) has six vertices and is
unique up to relabeling of vertices. This triangulation is a classic result
and is often referred to in literature as \(\mathbb{RP}^2_6\). It can be
written as\\
\(\g{123\wedge 326\wedge 461\wedge 412\wedge 526\wedge 561\wedge 153\wedge 364\wedge 425\wedge 534}\), and is shown in Figure
\ref{fig-RP2}. Using \texttt{graph\_rewrite.decompose}, we conclude that this
triangulation is totally satisfiable.

\begin{figure}[h]
\centering
\begin{center}
\includegraphics[trim={10cm 6cm 6cm 4cm}, clip,width=7cm]{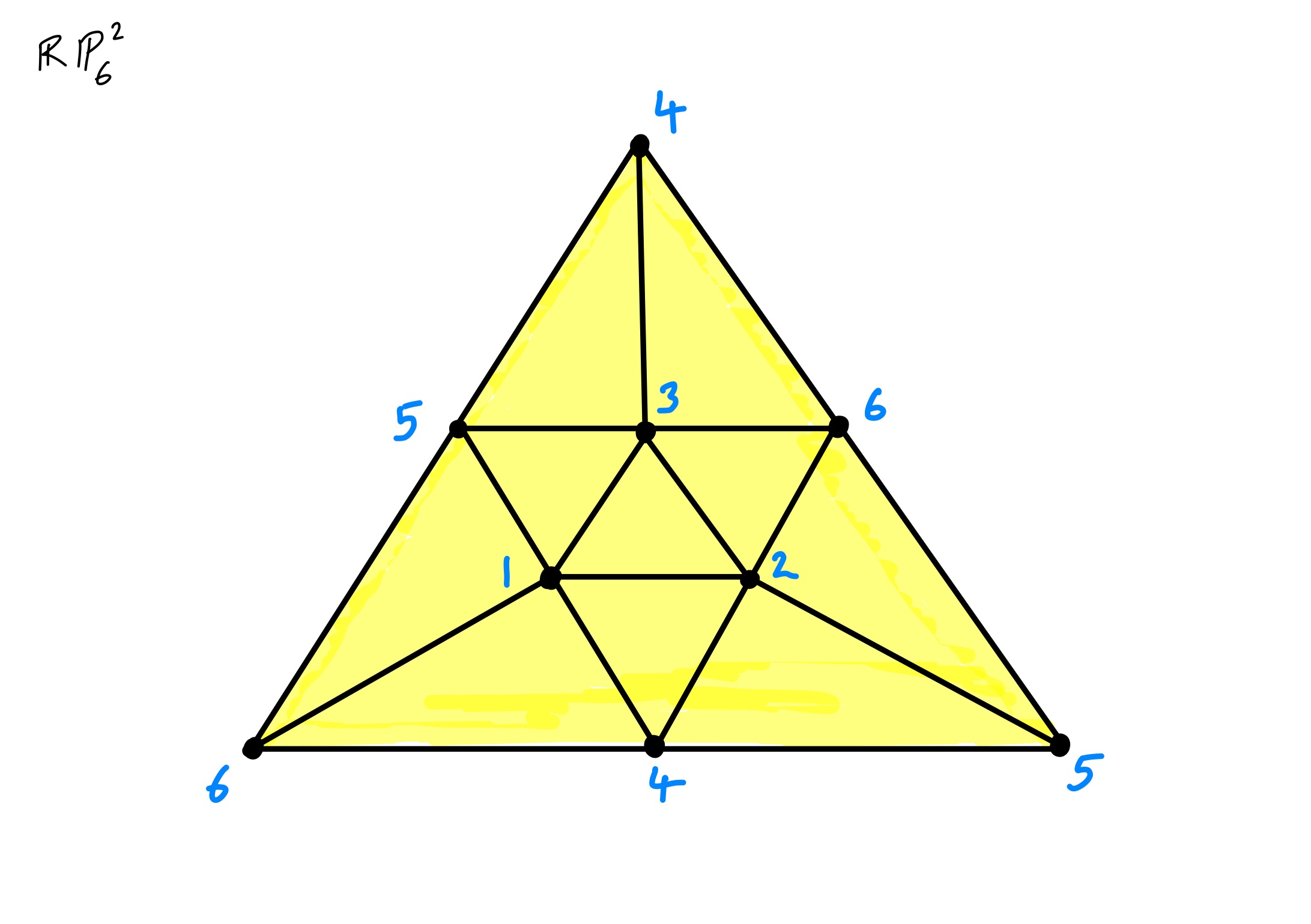}
\end{center}
\caption{\label{fig-RP2}
Minimal triangulation of the real projective plane, denoted \(\mathbb{RP}^2_6\).}
\end{figure}

\begin{minted}[frame=lines,label= (python3.9) (scratch) <<RP2>>]{python}
import mhgraph as mhg
import graph_rewrite as grw

rp2: mhg.MHGraph
rp2 = mhg.mhgraph([[1, 2, 3], [3, 2, 6], [4, 6, 1], [4, 1, 2], [5, 2, 6],
		   [5, 6, 1], [1, 5, 3], [3, 6, 4], [4, 2, 5], [5, 3, 4]])
grw.decompose(rp2)
\end{minted}
\footnotesize
\color{darkgray}
\uline{Output}:
\begin{verbatim}
(1, 2, 3)¹,(3, 2, 6)¹,(4, 6, 1)¹,(4, 1, 2)¹,(5, 2, 6)¹,(5, 6, 1)¹,(1, 5, 3)¹,
    (3, 6, 4)¹,(4, 2, 5)¹,(5, 3, 4)¹ is SAT
\end{verbatim}
\normalsize
\color{black}

\subsection{Minimal triangulation of a Klein bottle}
\label{sec:org81d17f5}
A Klein bottle has six distinct \(8\)-vertex triangulations
\cite{Cervone1994}. These triangulations all contain \(16\) distinct
hyperedges and have a minimum vertex degree of \(6\). These large numbers
make it difficult to determine the satisfiability status of these
triangulations without committing to significant computational resources.

Of the six distinct triangulations, we checked but one --- the ``242
triangulation'', given by the faces\\
\(\g{123\wedge 372\wedge 153\wedge 175\wedge 147\wedge 162\wedge 642\wedge
168\wedge 148\wedge 248\wedge 643}\)\\ \(\g{\wedge 374\wedge
685\wedge 653\wedge 825\wedge 275}\).

Passing it to \texttt{graph\_rewrite.decompose} and waiting for several hours of
computations results in the discovery that the \(242\) configuration is
unsatisfiable. In fact, it is unsatisfiable even if we remove the
\(\g{825\wedge 275}\) subgraph!

The \(242\) triangulation and its unsatisfiable subgraph are shown in
Figure \ref{fig-klein-bottle} for reference.

\begin{figure}[h]
\centering
\begin{subfigure}{0.49\textwidth}
\begin{center}
\includegraphics[trim={7cm 7cm 10cm 5cm}, clip,width=7cm]{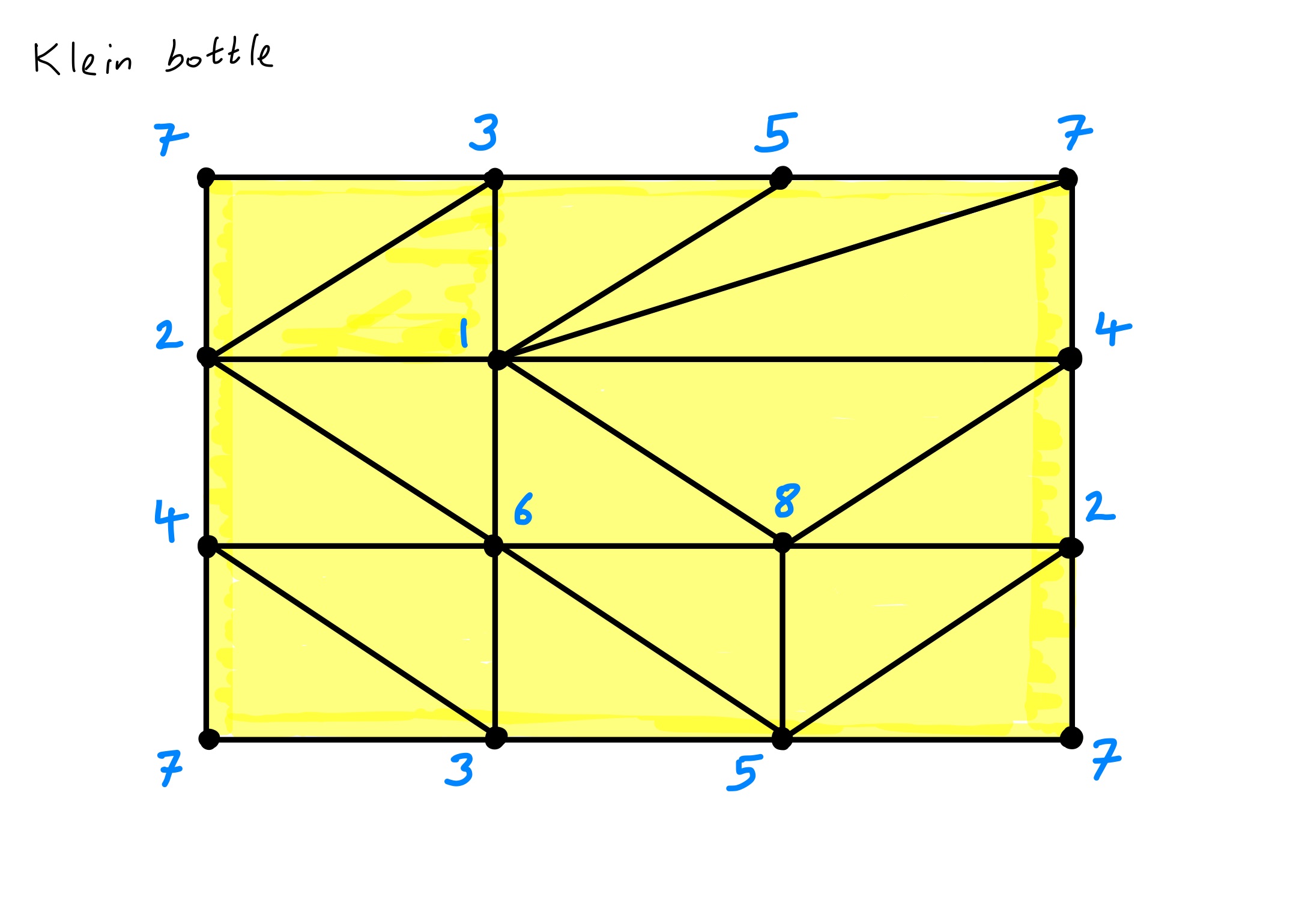}
\end{center}
\end{subfigure}
\begin{subfigure}{0.49\textwidth}
\begin{center}
\includegraphics[trim={5cm 7cm 10cm 7cm}, clip,width=7cm]{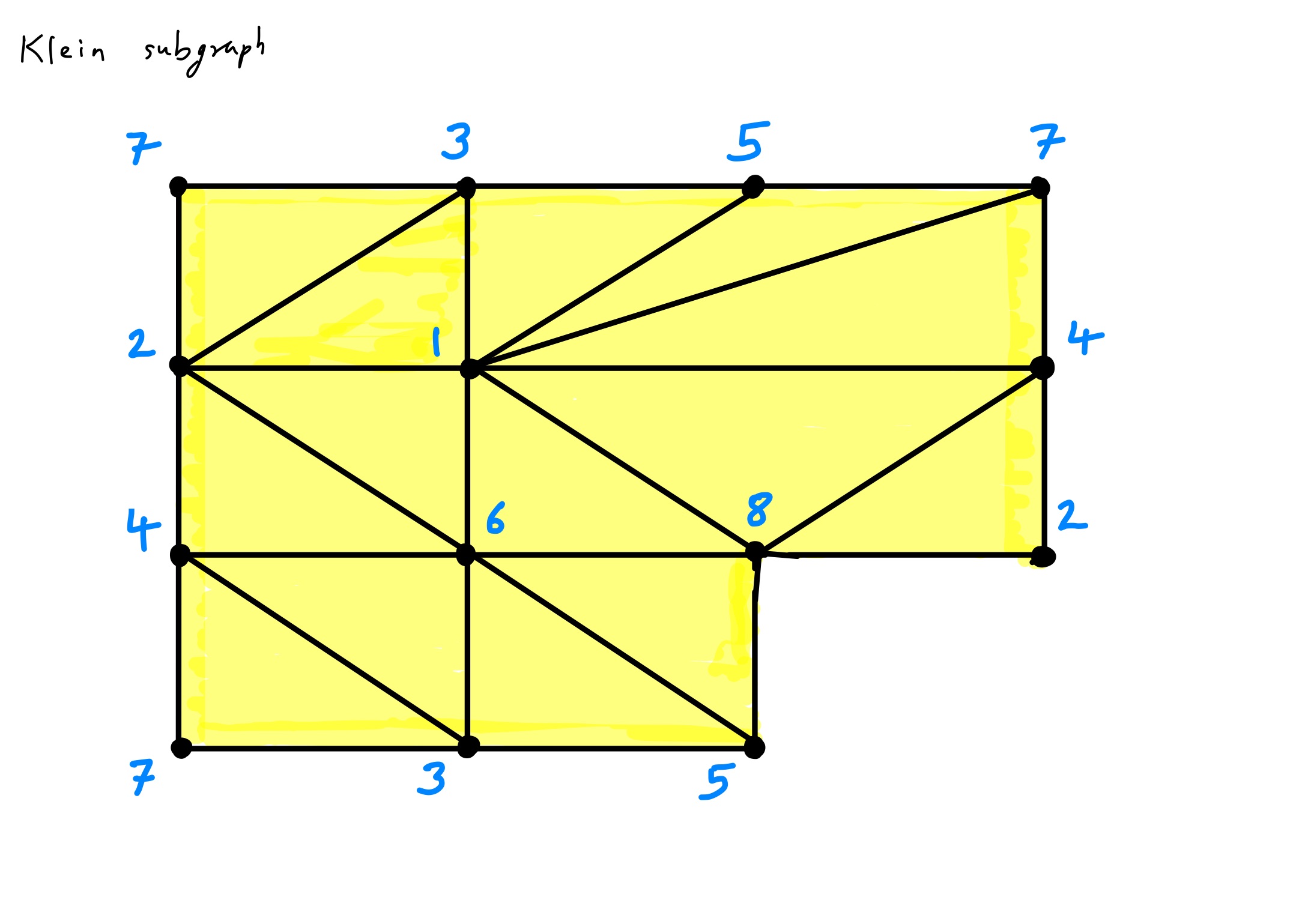}
\end{center}
\end{subfigure}
\caption{\label{fig-klein-bottle}
The ``242 triangulation'' of a Klein bottle and its unsatisfiable subgraph.}
\end{figure}

\subsection{Minimal triangulation of a torus}
\label{sec:org6c4410c}
The torus can be minimally triangulated \cite{Mobius1886} as --- \\
\(\g{126\wedge 267\wedge 237\wedge 371\wedge 674\wedge 745\wedge 715\wedge
156\wedge 412\wedge 452\wedge 523}\)\\ \(\g{\wedge 563\wedge
634\wedge 431}\).

This triangulation is shown in Figure \ref{fig-torus} and is found to be
unsatisfiable by the\\
\texttt{graph\_rewrite.decompose} function. In fact, it is unsatisfiable even if we
remove the \(\g{634\wedge 431}\) subgraph.

\begin{figure}[h]
\centering
\begin{subfigure}{0.4\textwidth}
\begin{center}
\includegraphics[trim={13cm 1cm 17cm 3cm}, clip,width=5cm]{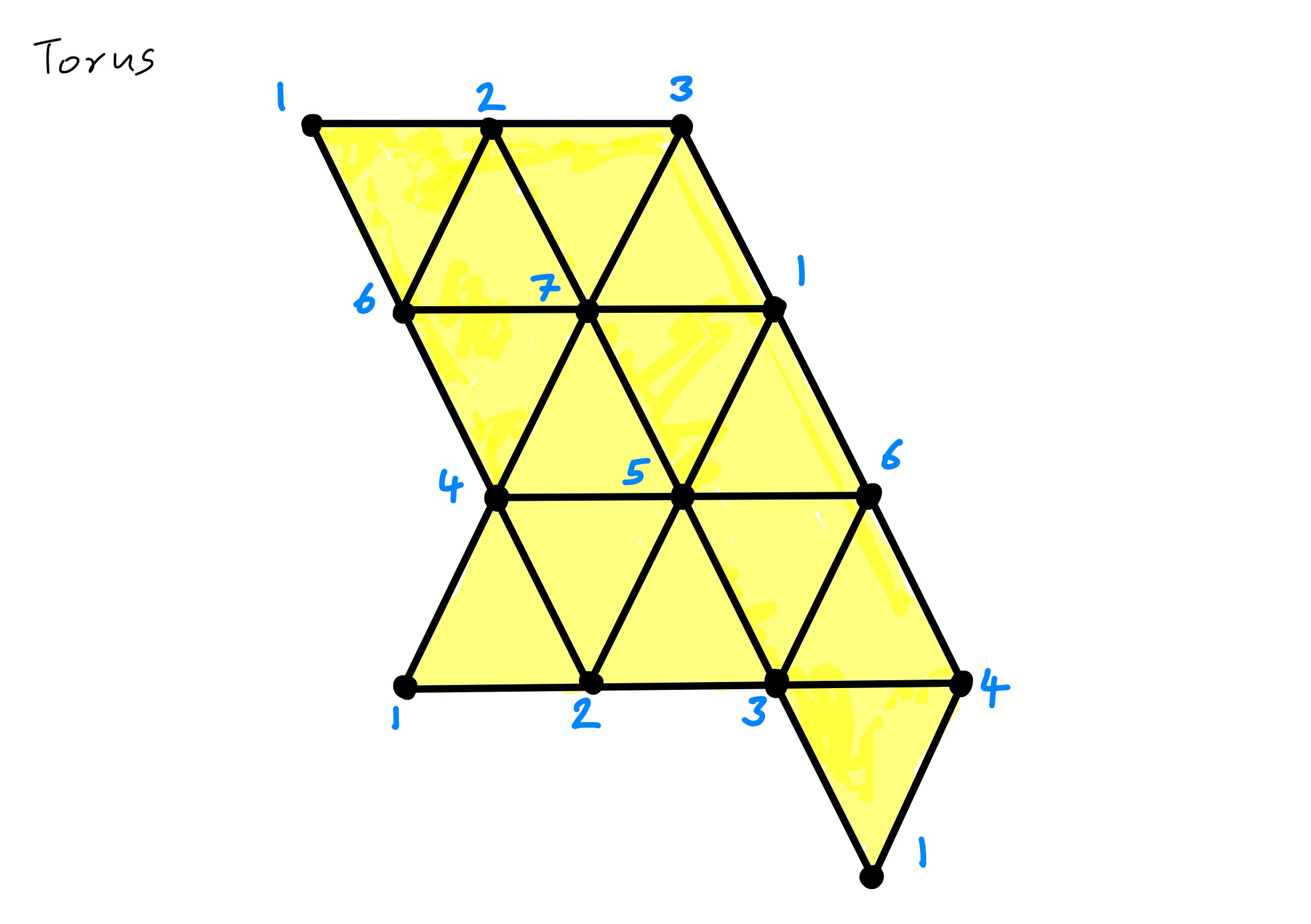}
\end{center}
\end{subfigure}
\begin{subfigure}{0.4\textwidth}
\begin{center}
\includegraphics[trim={16cm 0cm 15cm 9cm}, clip,width=5cm]{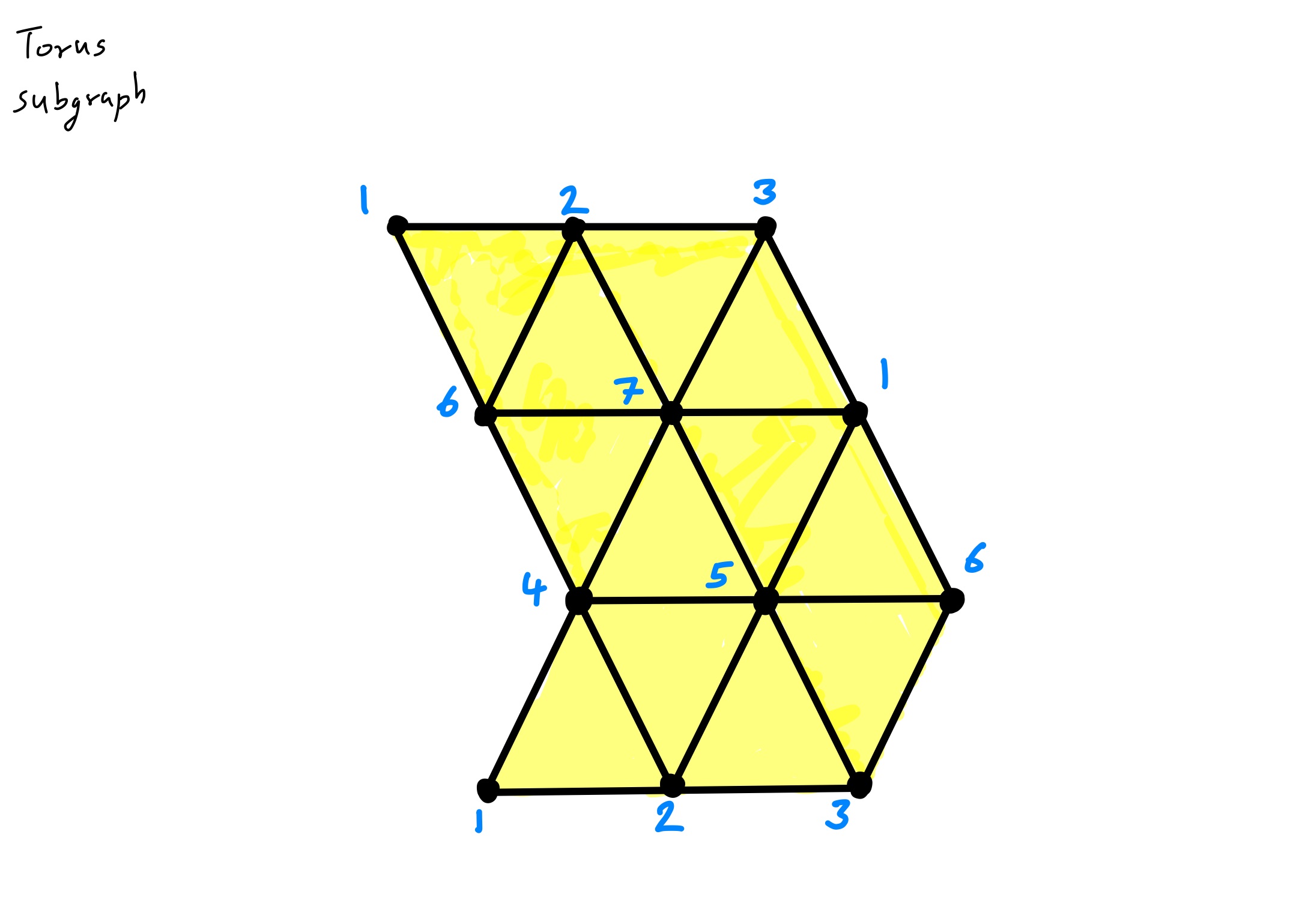}
\end{center}
\end{subfigure}
\caption{\label{fig-torus}
Minimal triangulation of a torus and its unsatisfiable subgraph.}
\end{figure}

\section{Satisfiability of infinite graphs}
\label{sec:org0d3ace9}
We recall from \S \ref{sec:org2ef3df6} that \(V\) is an arbitrary countable set, edges
are nonempty sets of vertices, and graphs are nonempty multisets of edges.
This definition does not exclude edges or graphs from being of countably
infinite size. A graph with infinite edges, or edges of infinite size is an
\emph{infinite graph}.

We note that Cnfs, in a similar vein, can also be infinite. \emph{Infinite Cnfs}
either have infinitely many clauses, or have clauses of infinite size.

The notions of assignment, satisfiability, unsatisfiability --- all carry
over to infinite graphs and infinite Cnfs. The only notions that do not
carry over are the questions of computational complexity since we cannot
talk of program run-time for infinite instances of \GraphSAT.

\subsection{Infinitely many disconnected loops}
\label{sec:orga57344c}
The graph \(\g{1\wedge 2\wedge 3\wedge \ldots}\) is a graph made of countably infinite
disconnected self-loops. This graph is totally satisfiable because every connected
component of it is.

\subsection{Uniform infinite trees}
\label{sec:org863e1f1}
For examples of totally satisfiable infinite graphs that are connected, we consider
a family of tree graphs. For positive integer \(n\), let \(\g{T_n}\) denote
an infinite tree graph with each vertex being connected to exactly \(n\)
different vertices via edges of size \(2\). These are also sometimes
referred to in the literature as infinite trees of uniform degree \(n\).

Each \(\g{T_n}\) is in fact totally satisfiable since we proved in
\cite{KaHi2020} that every tree is totally satisfiable and since the proof
did not depend on the finiteness of the graph, the theorem still hold for
infinite graphs .

Another intuitive way to see that \(\g{T_2}\), for example, is totally satisfiable
is that (at the level of Cnfs) each vertex can be used to satisfy its
adjacent clause (see Figure \ref{fig-infinite-line-and-ray}). This results in a
chain of assignments and each clause is satisfied in a style reminiscent of
Hilbert's famous infinite hotel.

\begin{figure}[h]
\centering
\begin{subfigure}{0.4\textwidth}
\begin{center}
\includegraphics[trim={16cm 27cm 20cm 19cm}, clip,width=7cm]{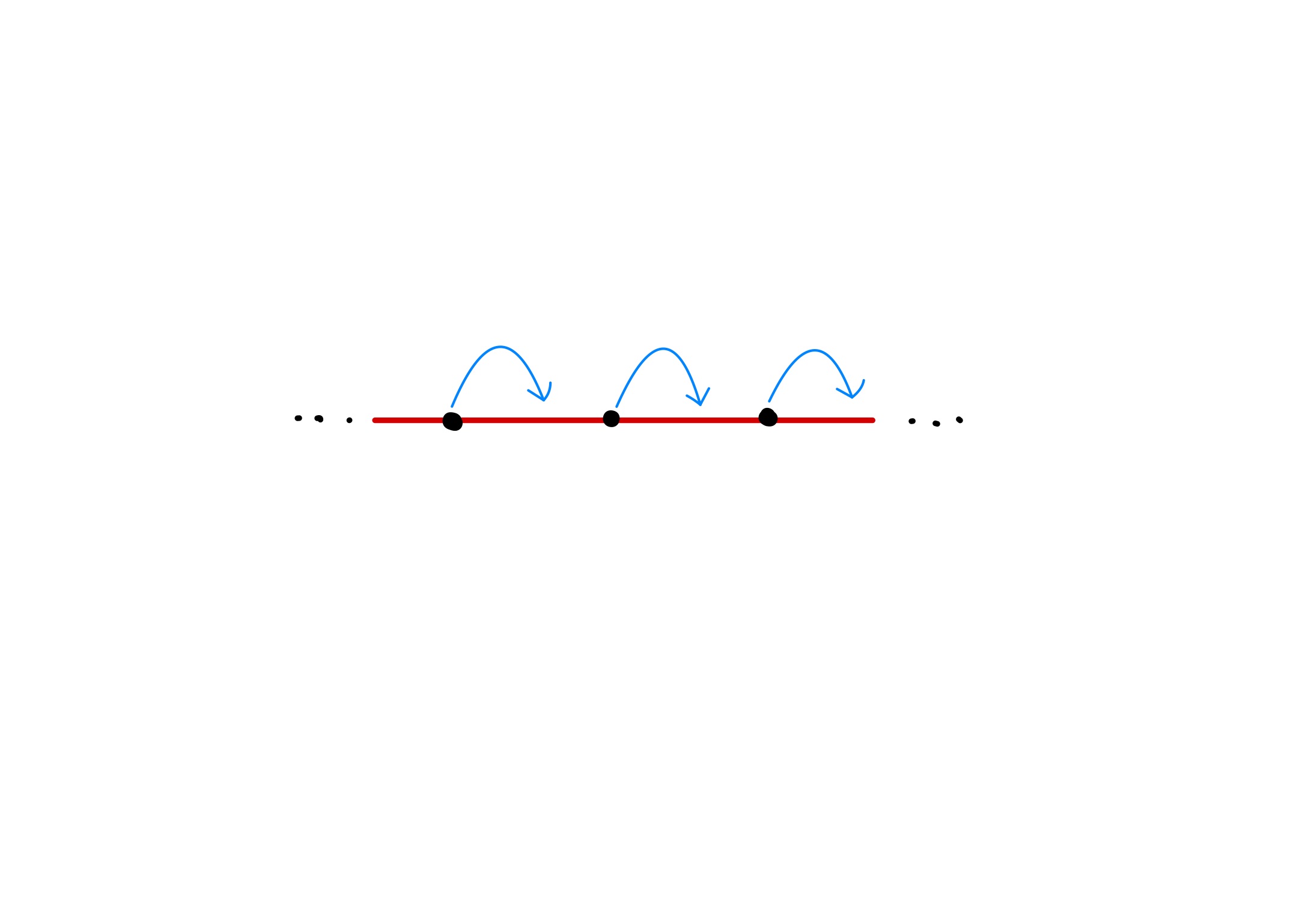}
\end{center}
\caption{\(\g{T_2}\) (the infinite line graph) can be satisfied by vertex assignments shown using blue arrows.}
\end{subfigure}
\hfill
\begin{subfigure}{0.4\textwidth}
\begin{center}
\includegraphics[trim={22cm 27cm 17cm 20cm}, clip,width=7cm]{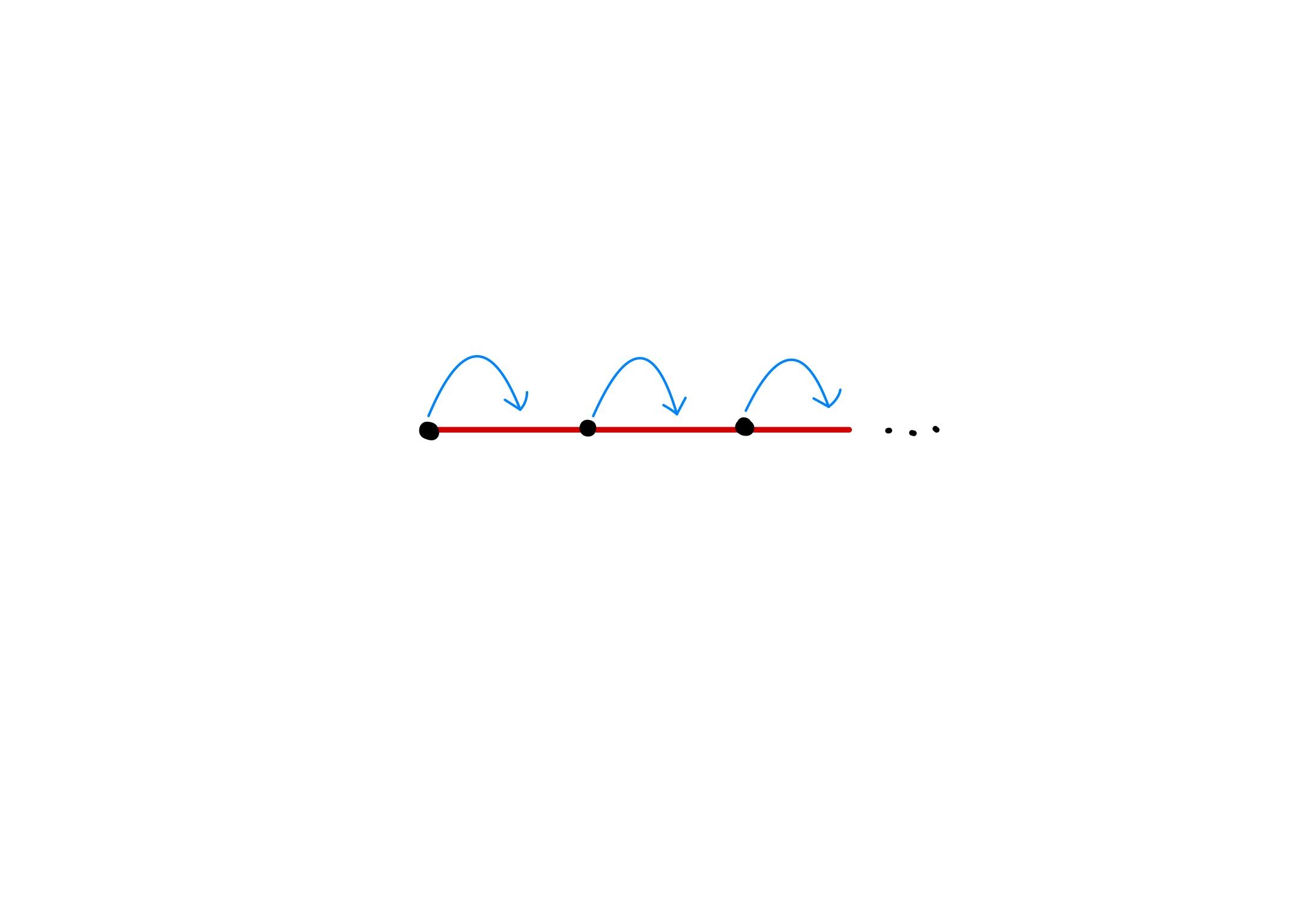}
\end{center}
\caption{The infinite ray graph is similarly totally satisfiable.}
\end{subfigure}

\vspace*{5mm}
\begin{subfigure}{0.7\textwidth}
\begin{center}
\includegraphics[trim={15cm 30cm 20cm 15cm}, clip,width=7cm]{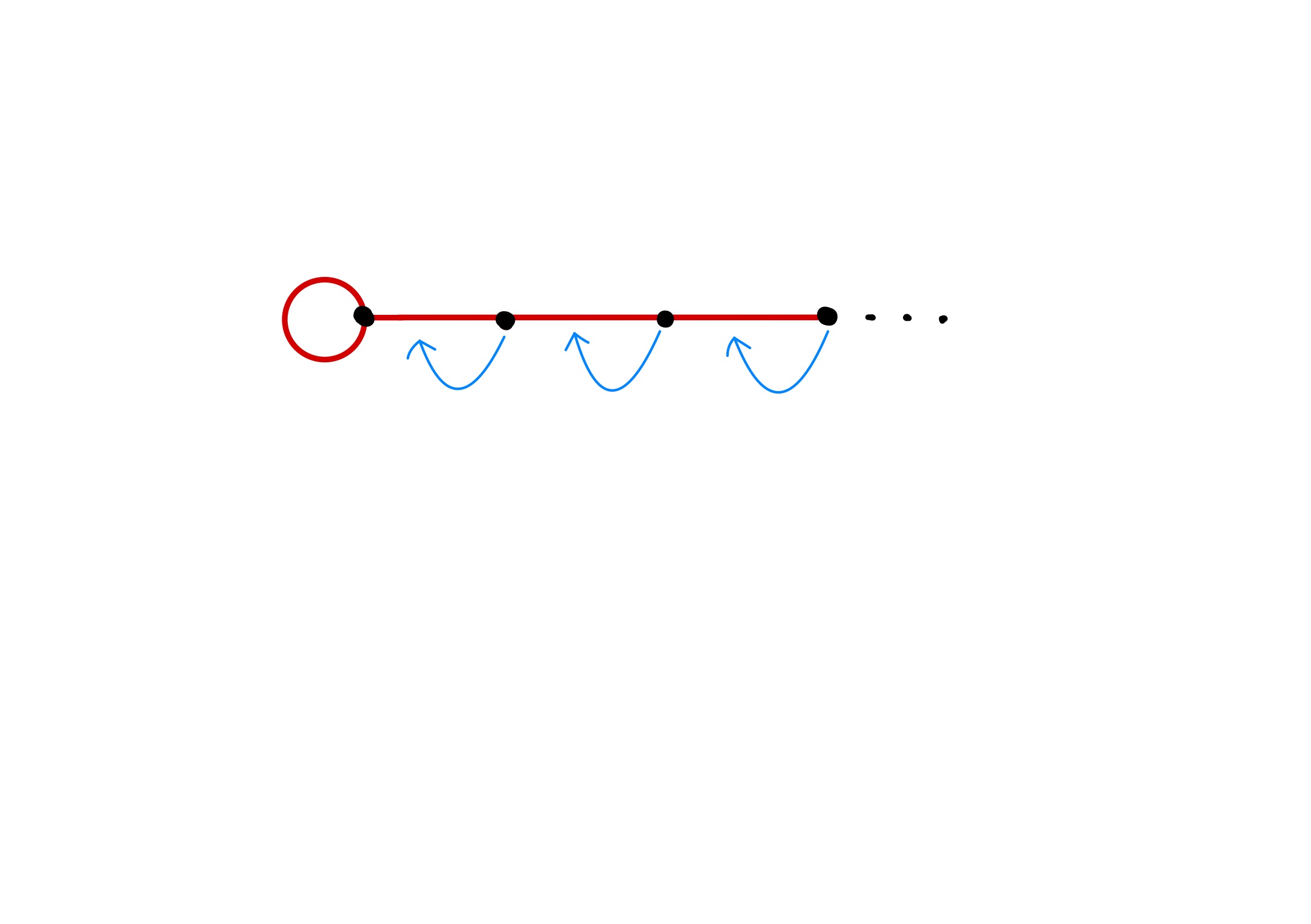}
\end{center}
\caption{The infinite ray graph with a looped on the tail is similarly totally satisfiable.}
\end{subfigure}
\caption{\label{fig-infinite-line-and-ray}
Infinite line and ray graphs are totally satisfiable using vertex assignment. Each edge is satisfied by a unique vertex.}
\end{figure}

\subsection{Infinite ray graph}
\label{sec:org3de146e}
Proof by demonstrating a vertex assignment also for other infinite graphs.
We should keep in mind that a valid vertex assignment can only help us
remove a single adjacent edge (or hyperedge) per vertex. Also, the
existence of a vertex assignment implies that the graph in question is
totally satisfiable, but its nonexistence does not prove that the graph is
unsatisfiable.

We use this technique to argue that the infinite ray graph with a looped
tail (see Figure \ref{fig-infinite-line-and-ray}) is totally satisfiable. At the level of
Cnfs, we can see that the tail vertex can be used to satisfy the loop. The
vertex next to the loop satisfies the last edge, the vertex after that
satisfies that last-but-one edge, and so on. This assignment is shown in
the figure using arrows. This proves that the infinite ray as well as the
infinite ray with looped tail are both totally satisfiable graphs.

\subsection{Bi-infinite strip}
\label{sec:org52e7457}
We next consider a thickened version of \(\g{T_2}\) made of hyperedges, as
shown in Figure \ref{fig-infinite-strip-and-tiling}. We call this is bi-infinite
strip and claim that it is totally satisfiable. The assignment that satisfies a
given Cnf in this graph can be derived by using the arrows shown in the
figure. Similarly, the mono-infinite strip shown in Figure
\ref{fig-infinite-strip-and-tiling} is also totally satisfiable by the vertex assignment
shown in the figure.

\begin{figure}[h]
\centering
\begin{subfigure}{0.4\textwidth}
\begin{center}
\includegraphics[trim={7cm 29cm 18cm 13cm}, clip,width=7cm]{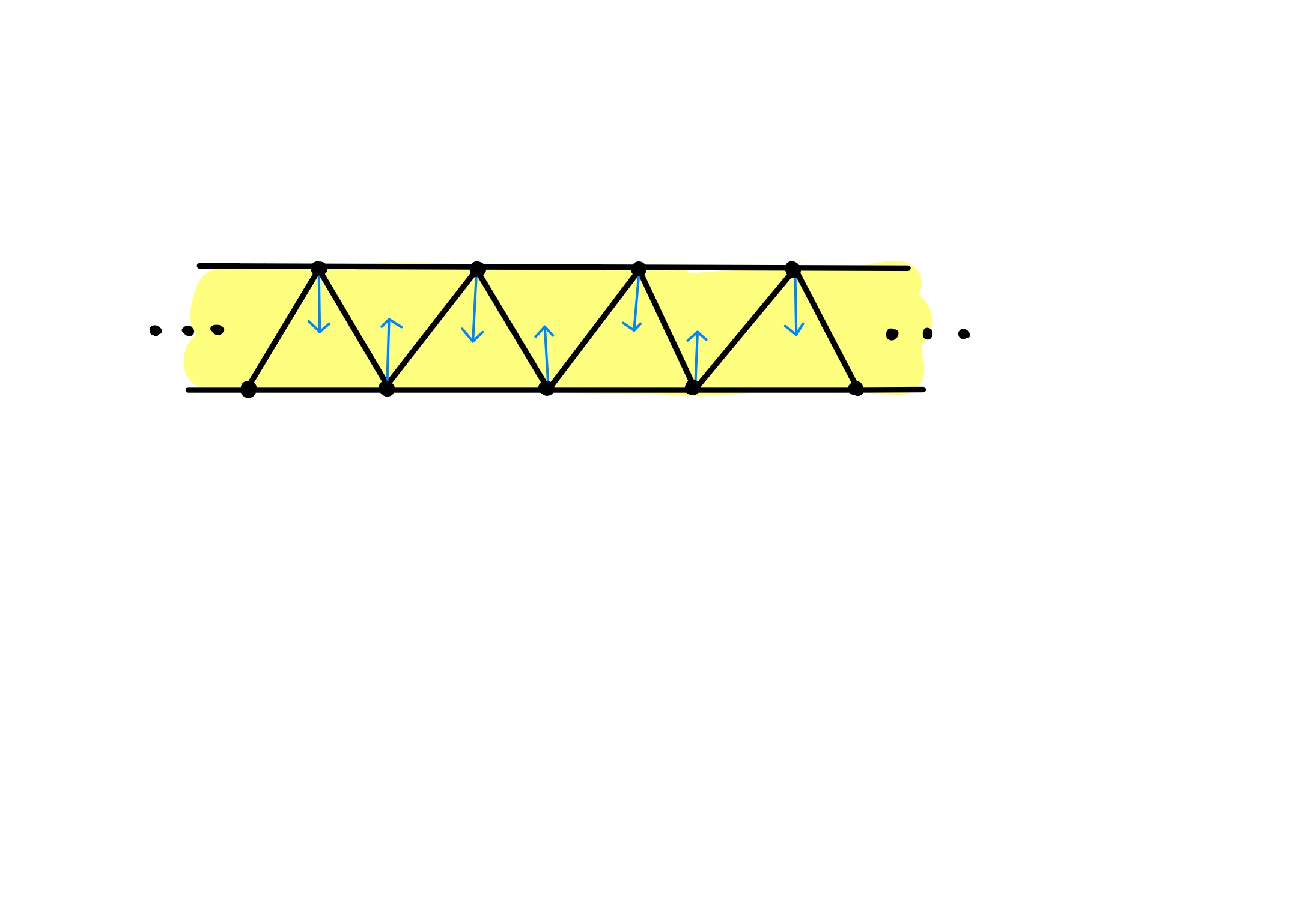}
\end{center}
\caption{The bi-infinite strip is totally satisfiable using the vertex assignments shown using blue arrows.}
\end{subfigure}
\hfill
\begin{subfigure}{0.4\textwidth}
\begin{center}
\includegraphics[trim={10cm 26cm 11cm 15cm}, clip,width=7cm]{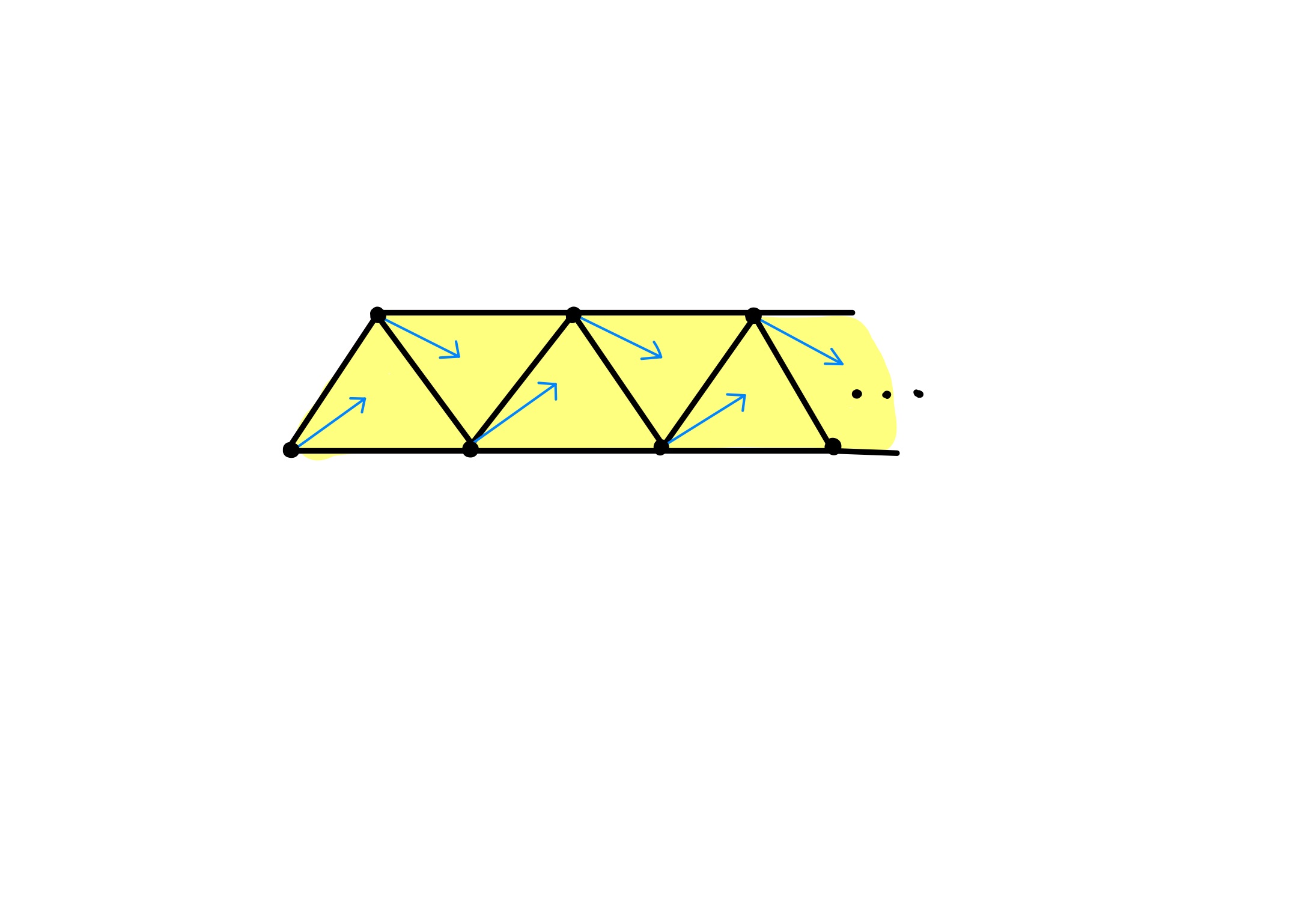}
\end{center}
\caption{The mono-infinite strip is similarly totally satisfiable.}
\end{subfigure}

\vspace*{5mm}
\begin{subfigure}{0.7\textwidth}
\begin{center}
\includegraphics[trim={5cm 10cm 20cm 7cm}, clip,width=8cm]{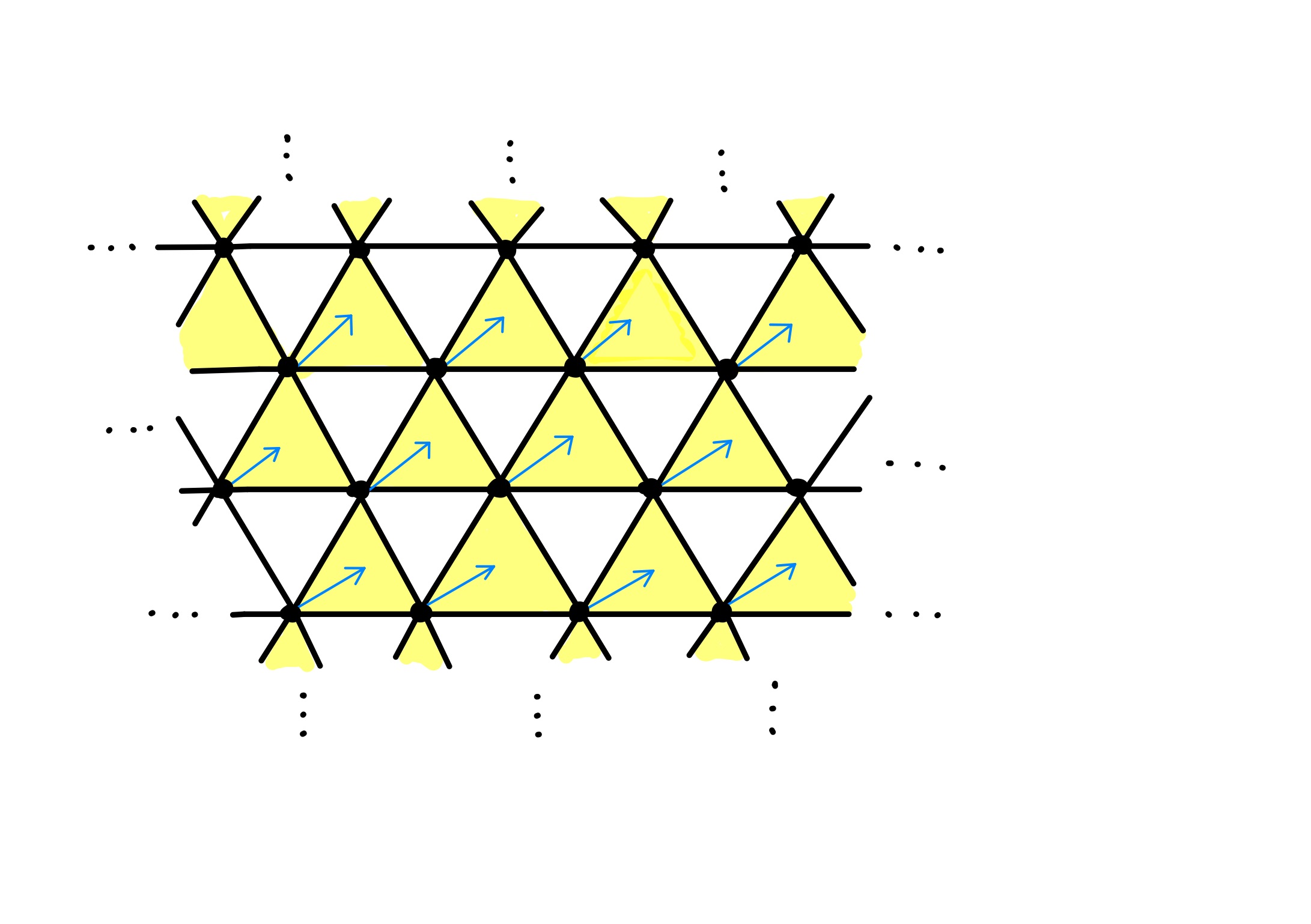}
\end{center}
\caption{This triangulation is formed by taking a triangular uniform plane tiling and removing alternate tiles. It is totally satisfiable using the vertex assignments shown.}
\end{subfigure}
\caption{\label{fig-infinite-strip-and-tiling}
Triangulations having vertex-assignments that use a single vertex to satisfy each hyperedge are totally satisfiable.}
\end{figure}

\subsection{Plane tiling with missing alternate tiles}
\label{sec:org572370c}
Lastly, we consider the tiling of the plane with alternate triangles and
holes as shown in Figure \ref{fig-infinite-strip-and-tiling}. This triangulation
can be satisfied by the vertex assignment shown in the same figure.

\subsection{Compactness theorem and infinite \GraphSAT}
\label{sec:orgdd07ba5}
A graph is totally satisfiable if and only if every Cnf in it is
satisfiable. The condition for every Cnf being satisfiable can itself be
translated into a large Cnf if we allow the introduction of new variables.

For example, consider the single-edge graph \(\g{ab}\). There is a set of
\(4\) Cnfs in the set \(\g{ab}\), and by relabeling the vertices, we can
write \(\bgamma(\g{ab}) = \sigma\big((a_1\vee b_1)\wedge(a_2\vee
\overline{b_2})\wedge(\overline{a_3}\vee
b_3)\wedge(\overline{a_4}\vee\overline{b_4})\big)\). We call this
translation map from Graphs to Cnfs \(\g{\tau}\) (short for translation).

We use this map \(\g{\tau}\) to change the total-satisfiability question of
a graph \(\g{g}\) from a universal quantification over all Cnfs in
\(\g{g}\) to an existential quantification over all truth-assignments for
the Cnf \(\g{\tau}(\g{g})\). This change to existential quantification
allows us to apply the Compactness Theorem.

In mathematical logic, the Compactness Theorem states that a set of
first-order sentences has a model if and only if every finite subset of it
has a model. In the context of \GraphSAT, this means that an infinite graph
is totally satisfiable if and only if every finite subgraph of it is
totally satisfiable. This means we can always restrict out attention to
studying only finite graphs. It also means that any unsatisfiable infinite
graph must have an unsatisfiable finite subgraph.

\section{Conclusion and future directions}
\label{sec:org23f1643}
An outcome of this work is the creation of a new graph decision problem ---
\GraphSAT. In \cite{KaHi2020} we showed that \(2\)\GraphSAT is in
complexity class P and has a finite obstruction set containing four simple
graphs \cite{KaHi2020}. The natural next step of exploring \(3\)\GraphSAT
gave rise to the local graph rewriting theorem (Theorem
\ref{thm-local_rewriting}), which leveraged the fact that taking a union over all
possible vertex-assignments preserves the satisfiability status of a graph.
Using this theorem, we were able to generate a list of graph reduction
rules and an incomplete list of obstructions to satisfiability of
multi-hypergraphs.

An incomplete list of known unsatisfiable looped-multi-hypergraphs
(pictured in Figure \ref{fig-unsat-graphs} and listed in Appendix \ref{sec:org9d7b11d}.

\subsection{Future directions}
\label{sec:org9131325}
We showed that the complexity class of 2\GraphSAT is P while the complexity
class for 3\GraphSAT is not known. Moreover, the effect of local graph
rewriting on 3\GraphSAT's complexity class is not known. Hence a key
research question that arises is whether local rewriting preserves
complexity, and whether it makes 3\GraphSAT easier in practice.

We have an incomplete list of unsatisfiable looped-multi-hypergraphs.
Questions that arise within this context are --- whether the number of
essential \SAT-invariant graph reduction rules is finite? Even if the
reduction rules are not finite, are they implementable in polynomial-time.
Even if they are not implementable in polynomial time, it is possible that
there is a polynomial-time check for demonstrating that none of the
reduction rules apply to a given graph.

It is also not known if the number of minimal unsatisfiable graphs under
these reduction rules is finite. So far we have found more than 200
distinct unsatisfiable and irreducible looped-multi-hypergraphs with less
than 7 vertices. If the complete list is infinite, it would imply that
3\GraphSAT is not in complexity class P.

If 3\GraphSAT is in P, this would give us an easy P-time heuristic check
for 3\SAT, simplifying some 3\SAT cases, while not directly affecting the
complexity class of 3\SAT.

Let \(\K{a}{b}\) denote the complete \(a\)-uniform hypergraph on \(b\)
vertices. To construct \(\K{a}{b}\), we can start with \(b\) vertices and
connect all \(\choose{b}{a}\) combinations with a hyperedge of size \(a\).
We can also think of this as the \((a-1)\)-skeleton of a \((b-1)\)-simplex.

The hypergraph \(\K{a}{b}\)'s satisfiability status is interesting because
it combines the extreme of having all possible hyperedges connected (which
can force unsatisfiability) with the extreme of each hyperedge being
incident on a large number of vertices (which can force satisfiability).

For example, we know that \(\K{2}{4}\) is \(K_4\), i.e. the complete simple
graph on \(4\) vertices and is known to be totally satisfiable. On the
other hand, the graph \(\K{2}{3}\) is \(C_3\) is known to be totally
satisfiable. Table \ref{table-complete-uniform-graphs} summarizes the known
satisfiability statuses of various \(\K{a}{b}\) graphs. As seen in the
table, the satisfiability-status of \(\K{4}{7}\) and \(\K{5}{7}\) are not
known.

\begin{table}[t]
\caption{A table showing the satisfiability statuses of complete uniform hypergraphs. Non-obvious results are shown in boldface.}
\begin{center}
\begin{tabular}{c l l l l l l l l}
\hline
 & b = 1 & b = 2 & b = 3 & b =  4 & b =  5 & b = 6 & b = 7 & \(\cdots\)\\
\hline
a = 1 & sat & sat & sat & sat & sat & sat & sat & \(\cdots\)\\
a = 2 & - & sat & sat & \textbf{unsat} & unsat & unsat & unsat & \(\cdots\)\\
a = 3 & - & - & sat & sat & \textbf{unsat} & unsat & unsat & \(\cdots\)\\
a = 4 & - & - & - & sat & sat & \textbf{sat} & \textbf{unknown} & \(\cdots\)\\
a = 5 & - & - & - & - & sat & sat & \textbf{unknown} & \(\cdots\)\\
a = 6 & - & - & - & - & - & sat & sat & \(\cdots\)\\
a = 7 & - & - & - & - & - & - & sat & \(\cdots\)\\
\(\vdots\) & - & - & - & - & - & - & - & \\
\end{tabular}
\end{center}
\label{table-complete-uniform-graphs}
\end{table}

The generalized rule for \(n\) triangular hyperedges meeting at a common
free vertex is not known. We do know the reduction rule only for \(n=3\)
--- \[\g{s\wedge 123\wedge 124\wedge 134} \bsim \g{s\wedge 23} \;\cup\;
\g{s\wedge 24} \;\cup\; \g{s\wedge 34}\] Reduction rules for \(n\geq 4\)
yield massive data-tables of resulting Cnfs, which we have so far been
unable to group into a convenient set of graphs.

\appendix

\section{Operator and notation summary}
\label{sec:org534de93}
All operators defined in \S \ref{sec:org93a5019} are summarized in
Table \ref{table-operator-summary}. These operators are written in increasing
order of binding-tightness. The order of binding-tightness can be used to
disambiguate expressions when multiple operators are used at the same time.

\begin{table}[H]
\scriptsize
\caption{Summary of all the operators.}
\begin{center}
\begin{tabular}{c l l p{2.5cm}}
\hline
Operator & Context & Meaning & Remarks\\
\hline
Invisible glue & between literals & boolean disjunction & binds tighter than all other operators\\
, & between clauses or Cnfs & boolean conjunction & also written as \(\wedge\)\\
\(\overline{x_1}\) & acts on literals & unary negation on literals & also written as \(\neg x_1\)\\
\(x[a]\) & action of assignment on Cnf & \(f_{a,\Cnf}(x)\) in \S \ref{sec:org5ecd0fd} & \\
\(\sim\) & between two Cnfs & equi-satisfiable Cnfs & equivalence relation\\
Invisible glue & between vertices & adjacency of vertices & \\
\(\g{e^n}\) & superscript for a (hyper)edge & edge-multiplicity & \\
\(\g{,}\) & between edges, or graphs & graph union & also called the adjacency of edges\\
\(\g{\vee}\) & between two sets of Cnfs & disjunction in \S \ref{sec:orgf3d1fdf} & \\
\(\g{\wedge}\) & between two sets of Cnfs & conjunction in \S \ref{sec:orgf3d1fdf} & \\
\(\g{g[v]}\) & action of vertex on a set of Cnfs & assignment in \S \ref{sec:org4e3aba8} & \\
\(\g{\le}\) & between two graphs & subgraph relation & \\
\(\g{\ll}\) & between two graphs & shaved version & \\
\(\bsim\) & between two sets of Cnfs & equi-satisfiable graphs/sets & binds looser than all other operators\\
\hline
\end{tabular}
\end{center}
\label{table-operator-summary}
\end{table}

\section{Standard graph disjunctions}
\label{sec:org7678103}
Here we list tables of standard graph disjunctions that one may encountered
when carrying out local graph rewriting calculations. These tables are all
generated using the \texttt{graph\_or} function in the \texttt{operations.py} module from
our \texttt{graphsat} Python package. In \S \ref{sec:org60b361b} we use these
disjunction results to derive graph reduction rules --- global rewrites that
do not affect the satisfiability of a graph. In \S \ref{sec:orgc6348d6}, we describe an ongoing effort to describe the
criterion for hypergraph minimality taking into account a growing list of
graph reduction rules.

The first two tables list graph disjunctions that can be written exactly as
a union of graphs; the third table lists graph disjunctions that can only
be listed as a subset-superset pair.

 \setlength{\tabcolsep}{0.5em}
\begin{minipage}[t]{0.35\textwidth}
\footnotesize
\begin{center}
\captionof{table}{\label{table-graph_disjunctions_size_2}
Graph disjunctions where size of \(\g{h_1}\) + size of \(\g{h_2}\) is \(2\).}
\begin{tabular}{r c l c l}
\hline
\(\g{h_1}\) &  & \(\g{h_2}\) &  & \(\g{h_1\vee h_2}\)\\
\hline
\(\g{a}\) & \(\g{\vee}\) & \(\g{a}\) & \(=\) & \(\btop \;\cup\; \g{a}\)\\
\(\g{a}\) & \(\g{\vee}\) & \(\g{b}\) & \(=\) & \(\g{ab}\)\\
\(\g{a}\) & \(\g{\vee}\) & \(\g{bc}\) & \(=\) & \(\g{abc}\)\\
\(\g{a}\) & \(\g{\vee}\) & \(\g{ab}\) & \(=\) & \(\btop \;\cup\; \g{ab}\)\\
\(\g{ab}\) & \(\g{\vee}\) & \(\g{cd}\) & \(=\) & \(\g{abcd}\)\\
\(\g{ab}\) & \(\g{\vee}\) & \(\g{ac}\) & \(=\) & \(\btop \;\cup\; \g{abc}\)\\
\(\g{ab}\) & \(\g{\vee}\) & \(\g{ab}\) & \(=\) & \(\btop \;\cup\; \g{ab}\)\\
\hline
\end{tabular}
\end{center}
\normalsize
\end{minipage}
\hspace*{0.1\textwidth}
\begin{minipage}[t]{0.4\textwidth}
\footnotesize
\begin{center}
\captionof{table}{\label{table-graph_disjunctions_size_3}
Graph disjunctions where size of \(\g{h_1}\) + size of \(\g{h_2}\) is \(3\).}
\begin{tabular}{r c l c l}
\hline
\(\g{h_1}\) &  & \(\g{h_2}\) &  & \(\g{h_1\vee h_2}\)\\
\hline
\(\g{a}\) & \(\g{\vee}\) & \(\g{a^2}\) & \(=\) & \(\g{a}\)\\
\(\g{a}\) & \(\g{\vee}\) & \(\g{b^2}\) & \(=\) & \(\g{a}\)\\
\(\g{a}\) & \(\g{\vee}\) & \(\g{ab^2}\) & \(=\) & \(\btop \;\cup\; \g{a} \;\cup\; \g{ab}\)\\
\(\g{a}\) & \(\g{\vee}\) & \((\g{b\wedge ab})\) & \(=\) & \(\g{a} \;\cup\; \g{ab}\)\\
\(\g{ab}\) & \(\g{\vee}\) & \(\g{c^2}\) & \(=\) & \(\g{ab}\)\\
\(\g{ab}\) & \(\g{\vee}\) & \((\g{a\wedge c})\) & \(=\) & \(\g{ab} \;\cup\; \g{abc}\)\\
\(\g{ab}\) & \(\g{\vee}\) & \(\g{a^2}\) & \(=\) & \(\g{ab}\)\\
\(\g{ab}\) & \(\g{\vee}\) & \((\g{a\wedge b})\) & \(=\) & \(\btop \;\cup\; \g{ab}\)\\
\(\g{ab}\) & \(\g{\vee}\) & \((\g{c\wedge ac})\) & \(=\) & \(\g{ab} \;\cup\; \g{abc}\)\\
\(\g{ab}\) & \(\g{\vee}\) & \((\g{c\wedge ab})\) & \(=\) & \(\g{ab} \;\cup\; \g{abc}\)\\
\(\g{ab}\) & \(\g{\vee}\) & \((\g{a\wedge cd})\) & \(=\) & \(\g{ab} \;\cup\; \g{abcd}\)\\
\(\g{ab}\) & \(\g{\vee}\) & \((\g{a\wedge ac})\) & \(=\) & \(\btop \;\cup\; \g{ab} \;\cup\; \g{abc}\)\\
\(\g{ab}\) & \(\g{\vee}\) & \((\g{a\wedge bc})\) & \(=\) & \(\btop \;\cup\; \g{ab} \;\cup\; \g{abc}\)\\
\(\g{ab}\) & \(\g{\vee}\) & \((\g{a\wedge ab})\) & \(=\) & \(\btop \;\cup\; \g{ab}\)\\
\(\g{ab}\) & \(\g{\vee}\) & \((\g{ab\wedge cd})\) & \(=\) & \(\g{ab} \;\cup\; \g{abcd}\)\\
\(\g{ab}\) & \(\g{\vee}\) & \((\g{ab\wedge ac})\) & \(=\) & \(\btop \;\cup\; \g{ab} \;\cup\; \g{abc}\)\\
\(\g{ab}\) & \(\g{\vee}\) & \(\g{ac^2}\) & \(=\) & \(\btop \;\cup\; \g{ab} \;\cup\; \g{abc}\)\\
\(\g{ab}\) & \(\g{\vee}\) & \((\g{ac\wedge bc})\) & \(=\) & \(\btop \;\cup\; \g{ab} \;\cup\; \g{abc}\)\\
\(\g{ab}\) & \(\g{\vee}\) & \(\g{ab^2}\) & \(=\) & \(\btop \;\cup\; \g{ab}\)\\
\hline
\end{tabular}
\end{center}
\normalsize
\vspace*{1ex}
\end{minipage}

 \setlength{\tabcolsep}{0.5em}
\begin{table}
\footnotesize
\begin{center}
\captionof{table}{\label{table-graph-disjunctions-incomplete}
Subset-superset pairs for graph disjunctions where size of \(\g{h_1}\) + size of \(\g{h_2}\) is at most \(3\).}
\begin{tabular}{r r r c l l l}
\hline
Subset &  & \(\g{h_1}\) &  & \(\g{h_2}\) &  & Superset\\
\hline
 &  & \(\g{a}\) & \(\g{\vee}\) & (\(\g{b\wedge c}\)) & \(\subset\) & \((\g{ab\wedge ac})\)\\
\(\g{ab}\) & \(\subset\) & \(\g{a}\) & \(\g{\vee}\) & (\(\g{a\wedge b}\)) & \(\subset\) & \(\g{ab}\;\cup\; (\g{a\wedge ab})\)\\
 &  & \(\g{a}\) & \(\g{\vee}\) & (\(\g{b\wedge cd}\)) & \(\subset\) & \((\g{ab\wedge acd})\)\\
\(\g{abc}\) & \(\subset\) & \(\g{a}\) & \(\g{\vee}\) & (\(\g{a\wedge bc}\)) & \(\subset\) & \(\g{abc} \;\cup\; (\g{a\wedge abc})\)\\
\(\g{ab}\) & \(\subset\) & \(\g{a}\) & \(\g{\vee}\) & (\(\g{b\wedge ac}\)) & \(\subset\) & \(\g{ab}\;\cup\; (\g{ab\wedge ac})\)\\
 &  & \(\g{a}\) & \(\g{\vee}\) & (\(\g{b\wedge bc}\)) & \(\subset\) & \((\g{ab\wedge abc})\)\\
\(\btop \;\cup\; \g{a} \;\cup\; \g{ab}\) & \(\subset\) & \(\g{a}\) & \(\g{\vee}\) & (\(\g{a\wedge ab}\)) & \(\subset\) & \(\btop \;\cup\; \g{a} \;\cup\; \g{ab} \;\cup\; (\g{a\wedge ab})\)\\
 &  & \(\g{a}\) & \(\g{\vee}\) & \((\g{bc\wedge de}\)) & \(\subset\) & \((\g{abc\wedge ade})\)\\
 &  & \(\g{a}\) & \(\g{\vee}\) & \((\g{bc\wedge bd}\)) & \(\subset\) & \((\g{abc\wedge abd})\)\\
 &  & \(\g{a}\) & \(\g{\vee}\) & \(\g{bc^2}\) & \(\subset\) & \(\g{abc^2}\)\\
\(\g{acd}\) & \(\subset\) & \(\g{a}\) & \(\g{\vee}\) & \((\g{ab\wedge cd})\) & \(\subset\) & \(\g{acd} \;\cup\; (\g{ab\wedge acd})\)\\
\(\g{ab}\;\cup\; \g{abc}\) & \(\subset\) & \(\g{a}\) & \(\g{\vee}\) & \((\g{ab\wedge bc})\) & \(\subset\) & \(\g{ab} \;\cup\; \g{abc} \;\cup\; (\g{ab\wedge abc})\)\\
\(\btop \;\cup\; \g{ab} \;\cup\; \g{ac}\) & \(\subset\) & \(\g{a}\) & \(\g{\vee}\) & \((\g{ab\wedge ac})\) & \(\subset\) & \(\btop \;\cup\; \g{ab} \;\cup\; \g{ac} \;\cup\; (\g{ab\wedge ac})\)\\
 &  & \(\g{ab}\) & \(\g{\vee}\) & \((\g{c\wedge d})\) & \(\subset\) & \((\g{abc\wedge abd})\)\\
 &  & \(\g{ab}\) & \(\g{\vee}\) & \((\g{c\wedge de})\) & \(\subset\) & \((\g{abc\wedge abde})\)\\
\(\g{abc}\) & \(\subset\) & \(\g{ab}\) & \(\g{\vee}\) & \((\g{c\wedge cd})\) & \(\subset\) & \(\g{abc} \;\cup\; (\g{abc\wedge abcd})\)\\
\(\g{abc}\) & \(\subset\) & \(\g{ab}\) & \(\g{\vee}\) & \((\g{c\wedge ad})\) & \(\subset\) & \(\g{abc} \;\cup\; (\g{abc\wedge abd})\)\\
 &  & \(\g{ab}\) & \(\g{\vee}\) & \((\g{cd\wedge ef})\) & \(\subset\) & \((\g{abcd\wedge abef})\)\\
 &  & \(\g{ab}\) & \(\g{\vee}\) & \((\g{cd\wedge ce})\) & \(\subset\) & \((\g{abcd\wedge abce})\)\\
 &  & \(\g{ab}\) & \(\g{\vee}\) & \(\g{cd^2}\) & \(\subset\) & \(\g{abcd^2}\)\\
\(\g{abcd}\) & \(\subset\) & \(\g{ab}\) & \(\g{\vee}\) & \((\g{cd\wedge ac})\) & \(\subset\) & \(\g{abcd} \;\cup\; (\g{abc\wedge abcd})\)\\
\(\btop \;\cup\; \g{abc} \;\cup\; \g{abd}\) & \(\subset\) & \(\g{ab}\) & \(\g{\vee}\) & \((\g{ac\wedge ad})\) & \(\subset\) & \(\btop \;\cup\; \g{abc} \;\cup\; \g{abd} \;\cup\; (\g{abc\wedge abd})\)\\
\(\btop \;\cup\; \g{abc} \;\cup\; \g{abd}\) & \(\subset\) & \(\g{ab}\) & \(\g{\vee}\) & \((\g{ac\wedge bd})\) & \(\subset\) & \(\btop \;\cup\; \g{abc} \;\cup\; \g{abd} \;\cup\; (\g{abc\wedge abd})\)\\
\(\g{abcd}\) & \(\subset\) & \(\g{ab}\) & \(\g{\vee}\) & \((\g{ac\wedge cd})\) & \(\subset\) & \(\g{abcd} \;\cup\; (\g{abc\wedge abcd})\)\\
\hline
\end{tabular}
\end{center}
\end{table}
\normalsize

The above tables show that the possible graph disjunctions grow quickly
with the edges participating in the disjunction. This is why we stop at a
maximum of three edges. For calculating the graph disjunction of more
edges, we can always use the \texttt{graph\_or} function from the \texttt{operations}
module on each individual disjunction.

\newpage
\section{List of known unsatisfiable graphs}
\label{sec:org9d7b11d}
Presented below is a list of known unsatisfiable hypergraphs. This list was
generated using SageMath's \texttt{nauty} module and then filtering for
unsatisfiable graphs.

\footnotesize
\begin{verbatim}
 1  (1)²
 2  (1), (2), (1,2)
 3  (1,2), (1,3), (1,4), (2,3), (2,4), (3,4)
 4  (1,2), (1,3), (2,3), (1,2,3), (1,2,4), (1,2,5), (3,4,5)
 5  (1,2), (1,3), (1,4), (2,3), (2,4), (1,2,5), (3,4,5)
 6  (1,2), (1,3), (1,4), (2,3), (2,4), (1,3,4)
 7  (1,2), (1,3), (1,4), (2,3), (2,4), (1,3,5), (2,4,5)
 8  (1,2)², (1,4), (1,2,4)²
 9  (1,2), (1,3), (1,4), (2,3), (1,2,4), (1,3,4)
10  (1,2), (1,3), (1,4), (2,3), (1,2,4), (1,4,5), (2,3,5)
11  (1,2), (1,3), (1,4), (2,3), (1,2,4), (1,3,5), (2,4,5)
12  (1,2), (1,3), (1,4), (2,3), (1,2,4), (2,3,4)
13  (1,2), (1,3), (1,4), (2,3), (1,2,4), (1,2,5), (3,4,5)
14  (1,2), (1,3), (1,4), (2,3), (1,2,5), (1,4,5), (3,4,5)
15  (1,2), (1,3), (1,4), (2,3), (1,4,5), (2,3,5), (2,4,5)
16  (1,2), (1,3), (1,4), (2,3), (1,4,5), (2,4,5), (3,4,5)
17  (1)², (1,3)
18  (1,2)², (1,4)², (1,2,4)
19  (1,2), (1,3), (1,4)², (2,3), (2,3,4)
20  (1,2)², (1,3)², (2,3)
21  (1,2)², (1,4), (1,5), (1,4,5), (2,4,5)
22  (1,2), (1,3), (1,4), (1,5), (2,3), (2,4,5), (3,4,5)
23  (1,2), (1,3)², (2,4), (1,2,4), (2,3,4)
24  (1,2), (1,3), (1,5), (2,4), (3,5), (2,3,4), (2,4,5)
25  (1,2), (1,3), (2,4), (1,2,3), (1,2,4), (1,3,4)
26  (1), (1,3), (1,5), (1,3,5), (3,4,5)
27  (1,2), (1,3), (1,5), (2,4), (1,2,4), (2,3,5), (3,4,5)
28  (1), (1,3)², (1,3,4)
29  (1,2), (1,3), (2,4), (1,2,4), (2,3,4), (1,3,4)
30  (1,2)², (1,3), (2,4), (1,3,4), (2,3,4)
31  (1,2), (1,3), (2,4), (3,4), (1,2,3), (1,2,4)
32  (1,2), (1,3), (2,4), (3,4), (1,2,5), (1,3,4), (3,4,5)
33  (1,2), (1,3), (2,4), (3,4), (1,2,4), (1,4,5), (2,3,5)
34  (1), (1,3)², (3,4)
35  (1,2), (1,3), (1,5), (2,4), (3,4), (1,2,5), (3,4,5)
36  (1,2), (1,3), (2,4), (1,2,3), (1,2,4), (1,3,5), (2,4,5)
37  (1,2), (1,3), (2,4), (1,2,3), (1,2,4), (1,4,5), (2,3,5)
38  (1,2), (1,3), (2,4), (1,2,3), (1,2,4), (1,2,5), (3,4,5)
39  (1,2), (1,3), (2,4), (1,2,3), (1,2,5), (2,4,5), (3,4,5)
40  (1,2), (1,3), (2,4), (1,2,3), (1,4,5), (2,3,4), (2,3,5)
41  (1,2), (1,3), (2,3), (2,4), (3,5), (4,5), (1,4,5)
42  (1,2), (1,3), (2,4), (3,5), (4,5), (1,2,3), (1,4,5)
43  (2)², (3,5)
44  (1,2), (1,3), (2,4), (3,5), (4,5), (1,2,5), (1,3,4)
45  (1,2), (1,3), (2,4), (4,5), (1,2,3), (1,3,5), (1,4,5)
46  (1,2), (1,3), (2,4), (4,5), (1,2,3), (1,3,5), (3,4,5)
47  (2)², (2,3,5)
48  (1,2), (1,3), (2,4), (4,5), (1,2,3), (1,4,5), (2,3,5)
49  (1,2), (1,3), (2,4), (4,5), (1,2,3), (1,4,5), (3,4,5)
50  (2), (1,2), (1,3), (1,2,3), (1,3,5)
51  (1,2), (1,3), (2,4), (1,2,3), (1,3,5), (1,4,5), (2,4,5)
52  (1,2), (1,3), (2,4), (1,2,3), (1,4,5), (2,3,5), (2,4,5)
53  (1,2), (1,3), (2,4), (1,3,5), (1,4,5), (2,3,5), (2,4,5)
54  (1,2), (1,3), (2,4), (1,2,3), (1,3,5), (2,4,5), (3,4,5)
55  (1,2), (1,3), (4,5), (1,2,4), (1,2,5), (1,3,4), (1,3,5)
56  (1,2), (1,3), (4,5), (1,2,4), (1,3,5), (2,4,5), (3,4,5)
57  (1,2), (1,3), (4,5), (1,2,4), (1,3,4), (1,3,5), (2,4,5)
58  (1,2), (1,3), (4,5), (1,2,4), (1,3,4), (2,3,5), (2,4,5)
59  (1,2), (1,3), (4,5), (1,2,4), (1,3,4), (2,4,5), (3,4,5)
60  (1,2), (1,3), (2,4)², (1,3,4), (2,3,4)
61  (1,2), (1,3), (2,4), (4,5), (1,2,5), (1,3,4), (3,4,5)
62  (1,2), (1,3), (2,4), (4,5), (1,2,5), (1,3,4), (1,3,5)
63  (1,2), (1,3), (2,4), (4,5), (1,3,4), (2,3,5), (3,4,5)
64  (1,2), (1,3), (2,4), (4,5), (1,2,5), (1,3,5), (2,3,4)
65  (1,2), (1,3), (2,4), (4,5), (1,3,5), (2,3,4), (3,4,5)
66  (1,2), (1,3), (2,4), (4,5), (1,2,4), (1,3,5), (2,3,5)
67  (1,2), (1,3), (2,4), (1,2,3), (1,4,5), (2,4,5), (3,4,5)
68  (1,2), (1,3), (2,4), (1,2,3), (1,3,5), (2,3,4), (2,4,5)
69  (1,2), (1,3), (2,4), (1,2,4), (1,2,5), (1,3,4), (3,4,5)
\end{verbatim}
\normalsize

\section{Implementation of local graph rewriting}
\label{sec:org322849c}
Below we include the docstring of the function, informing us what exactly
the function does, followed by its implementation as a code-snippet. The
implementation uses other functions defined in the package like
\texttt{operations.graph\_or}, \texttt{compute\_all\_two\_partitions\_of\_link}, and
\texttt{mhgraph.rest}. We will not detail each of these subsidiary functions here.
We leave it instead to the interested reader to look at \texttt{graphsat}'s source
code for more details.

\begin{minted}[frame=lines,label= (python3.9) (graph\_rewrite.py) <<local-rewrite-docstring>>]{python}
"""Locally rewrite at ``vertex`` assuming that the graph is only partially known.

This function only affects edges incident on ``vertex``, assuming that ``mhg`` only represents a part of the
full graph. The result is a dictionary of Cnfs grouped by their MHGraphs.
"""
\end{minted}

\begin{minted}[frame=lines,label= (python3.9) (graph\_rewrite.py) <<local\_rewrite>>]{python}
def local_rewrite(mhg: mhgraph.MHGraph, vertex: mhgraph.Vertex) -> dict[mhgraph.MHGraph, set[cnf.Cnf]]:
    <<local-rewrite-docstring>>                                                      # Add function docstring.
    rest: mhgraph.MHGraph = mhgraph.rest(mhg, vertex)                     # Compute the rest of mhg at vertex.

    part = list[mhgraph.MHGraph]                                    # The type of a part of link(mhg, vertex).
    two_partitions: Iterator[tuple[part, part]]              # Generate all 2-partitions of link(mhg, vertex).
    two_partitions = compute_all_two_partitions_of_link(mhg, vertex)

    resultant_cnfs: set[cnf.Cnf] = set()                                       # Initialize resultant Cnf set.

    for h1, h2 in two_partitions:                                    # Loop over all 2-partitions of the link.
	h1_or_h2: set[cnf.Cnf] = op.graph_or(h1, h2)                               # Disjunction of the parts.
	h12_and_rest: set[cnf.Cnf] = op.graph_and(h1_or_h2, rest)                 # Conjunction with the rest.
	resultant_cnfs |= h12_and_rest                                  # Add result to set of resultant Cnfs.

    return graph_collapse.create_grouping(resultant_cnfs)     # Group Cnfs by the Graphs to which they belong.
\end{minted}

\section{Implementation of graph disjunction}
\label{sec:org4527993}
We present an implementation of graph disjunction as a Python function.
This function can be found under the name \texttt{graph\_or} in the \texttt{operations.py}
module in our \texttt{graphsat} package. It computes the pairwise disjunction of
the Cartesian product of two sets of Cnfs. Since the disjunction of two
Cnfs is not a Cnf, we can bring it back into normal form using the function
\texttt{cnf\_or\_cnf} which is outlined below as a helper function and is part of
the \texttt{prop.py} module.

\begin{minted}[frame=lines,label= (python3.9) (operations.py) <<graph\_or>>]{python}
def graph_or(graph1: Union[MHGraph, set[cnf.Cnf]], graph2: Union[MHGraph, set[cnf.Cnf]]) -> set[cnf.Cnf]:
    """Disjunction of the Cartesian product of Cnfs."""
    if not isinstance(graph1, set):                 # Convert graph1 to its underlying set of Cnfs, if needed.
	graph1 = set(sat.cnfs_from_mhgraph(mhgraph(graph1)))
    if not isinstance(graph2, set):                 # Convert graph2 to its underlying set of Cnfs, if needed.
	graph2 = set(sat.cnfs_from_mhgraph(mhgraph(graph2)))

    product = it.product(graph1, graph2)                                          # Cartesian product of Cnfs.
    disjunction = it.starmap(prop.cnf_or_cnf, product)    # Use distributivity to bring back into normal form.
    disjunction_reduced = map(cnf.tautologically_reduce_cnf, disjunction)          # Simplify Cnf is possible.
    return set(disjunction_reduced)                                                    # Return a set of Cnfs.
\end{minted}

\begin{minted}[frame=lines,label= (python3.9) (prop.py) <<cnf\_or\_cnf>>]{python}
def clause_or_clause(clause1: Clause, clause2: Clause) -> Clause:             # Union of clause1 with clause2.
    return clause(clause1 | clause2)

def cnf_or_clause(cnf1: Cnf, clause_: Clause) -> Cnf:                 # Distribute each clause across the Cnf.
    return cnf([clause_or_clause(clause1, clause_) for clause1 in cnf1])

def cnf_or_cnf(cnf1: Cnf, cnf2: Cnf) -> Cnf:       # Distribute each clasuse across and then fold the result.
    return ft.reduce(cnf_and_cnf, [cnf_or_clause(cnf1, clause) for clause in cnf2])
\end{minted}

\bibliography{karve}
\bibliographystyle{acmdoi}
\end{document}

%% file: header.tex
\usepackage{setspace}
\usepackage{epsfig}  
\usepackage{graphicx}  
\usepackage{url}  
\usepackage{lscape}  
\usepackage[justification=raggedright]{caption}	

\usepackage{hyperref}
\hypersetup{colorlinks,
  linkcolor={blue},
  citecolor={blue},
  urlcolor={blue}}
\usepackage{amsthm}
\usepackage{todonotes}
\usepackage{mathtools} 
\usepackage{subcaption}
\usepackage{xspace}  
\usepackage{adjustbox}

\usepackage[frozencache=true,cachedir=minted-cache]{minted}

\setminted{fontsize=\scriptsize}
\setminted{stripnl=false}
\setminted{python3=true}
\setminted{labelposition=topline}
\newcommand{\Cnf}{\text{Cnf}}
\newcommand{\Variable}{\text{Variable}}
\newcommand{\Assignment}{\text{Assignment}}
\newcommand{\Literal}{\text{Literal}}
\newcommand{\Clause}{\text{Clause}}
\newcommand{\Vertex}{\text{Vertex}}
\newcommand{\Edge}{\text{Edge}}
\newcommand{\Graph}{\text{Graph}}

\newcommand{\Set}{\text{Set }}
\newcommand{\g}[1]{\boldsymbol{#1}}
\newcommand{\bsim}{\boldsymbol\sim}
\newcommand{\btop}{\boldsymbol\top}
\newcommand{\bbot}{\boldsymbol\bot}
\newcommand{\rest}{\mathrm{rest}}
\renewcommand{\star}{\mathrm{star}}
\newcommand{\link}{\mathrm{link}}
\newcommand{\bgamma}{\boldsymbol\gamma}

\newcommand{\SAT}{\textsc{sat}\xspace}
\newcommand{\GraphSAT}{\textsc{GraphSAT}\xspace}
\newcommand{\K}[2]{\boldsymbol{\mathcal{K}_{#1,#2}}}

\theoremstyle{plain}                 
\newtheorem{thm}{Theorem}[section]   
\newtheorem*{theorem*}{Theorem}
\newtheorem{lemma}[thm]{Lemma}       
\newtheorem{prop}[thm]{Proposition}  
\newtheorem{cor}{Corollary}[thm]     

\theoremstyle{definition}                   
\newtheorem{remark}{Remark}[section]        
\newtheorem{procedure}{Procedure}[section]  

\usepackage{tikz}
\usetikzlibrary{cd}  
\usetikzlibrary{calc}
\usetikzlibrary{arrows.meta}
\usetikzlibrary{backgrounds}
\usetikzlibrary{decorations.pathmorphing}



%% file: tikz_figures.tex
\tikzstyle{vertex}=[circle, draw, fill=black!100,
inner sep=1pt, minimum width=2pt]

\tikzstyle{vertexsmall}=[circle, draw, fill=black!100,
inner sep=0pt, minimum width=1.5pt]

\tikzset{triangle/.pic={
    \coordinate (A) at (0,0);
    \coordinate (B) at (150:1);
    \coordinate (C) at (210:1);
    \draw[fill=yellow!00!white] (A)--(B)--(C)--cycle;
    \foreach \x in {(A), (B), (C)}
    \node[vertex] (\x) at \x {};
}}

\newcommand{\Kfour}{
  \begin{tikzpicture}[trim left = 2, scale=0.45, baseline=0ex]
  \coordinate (A) at (0,0);
  \coordinate (B) at (1,0);
  \coordinate (C) at (60:1);
  \coordinate (D) at ($(C)-(0,0.52)$);
  \draw[fill=yellow!00!white] (A) -- (B) -- (C) -- cycle;
  \draw (A) -- (D) -- (C);
  \draw (B) -- (D);
  \foreach \x in {(A), (B), (C), (D)}
  \node[vertexsmall] (\x) at \x {};
  \end{tikzpicture}\;}

\newcommand{\Book}{
  \begin{tikzpicture}[trim left = 2, scale=0.3, baseline=0ex]
  \coordinate (A) at (0,0);
  \coordinate (C) at (25:1);
  \coordinate (D) at ($(C)+(0,0.5)$);
  \coordinate (E) at ($(C)+(0,1)$);
  \coordinate (B) at ($(30:2)+(0,-1)$);
  \draw[fill=yellow!00!white, fill opacity=0.3] (A) -- (B) -- (E) -- cycle;
  \draw[fill=yellow!00!white, fill opacity=0.6] (A) -- (D) -- (B) -- cycle;
  \draw[fill=yellow!00!white, fill opacity=1] (A) -- (C) -- (B) -- cycle;
  \foreach \x in {(A), (B), (C), (D), (E)}
  \node[vertexsmall] (\x) at \x {};
  \end{tikzpicture}\;}

\newcommand{\Butterfly}{
\begin{tikzpicture}[trim left = 1.5, scale=0.35, baseline=0ex]
  \coordinate (A) at (0,0);
  \coordinate (B) at (0,1);
  \coordinate (C) at (30:1);
  \coordinate (D) at (30:2);
  \coordinate (E) at ($(30:2)+(0,-1)$);
  \draw[fill=yellow!00!white] (A) -- (B) -- (C) -- cycle;
  \draw[fill=yellow!00!white] (C) -- (D) -- (E) -- cycle;
  \foreach \x in {(A), (B), (C), (D), (E)}
  \node[vertexsmall] (\x) at \x {};
  \end{tikzpicture}\;}

\newcommand{\Bowtiesmall}{
  \begin{tikzpicture}[trim left = 1.5, scale=0.35, baseline=0ex]
  \coordinate (A) at (0,0);
  \coordinate (B) at (0,1);
  \coordinate (C) at (30:1);
  \coordinate (D) at ($(30:1) + (1,0)$);
  \coordinate (E) at ($(30:2) + (1,0)$);
  \coordinate (F) at ($(30:2) + (1,-1)$);
  \draw[fill=yellow!00!white] (A) -- (B) -- (C) -- cycle;
  \draw[fill=yellow!00!white] (C) -- (D);
  \draw[fill=yellow!00!white] (D) -- (E) -- (F) -- cycle;
  \foreach \x in {(A), (B), (C), (D), (E), (F)}
  \node[vertexsmall] (\x) at \x {};
  \end{tikzpicture}\;}

\newcommand{\FlexibleButterfly}{
  \begin{tikzpicture}[scale=0.4, baseline=-1ex]
  \coordinate (A) at (0,0);
  \coordinate (TL) at (-1,0.5);
  \coordinate (BR) at (1,-0.5);
  \clip (TL) rectangle (BR);
  \draw (A) to [out=120, in=240, loop, min distance=20mm] (A);
  \draw (A) to [out=60, in=300, loop, min distance=20mm] (A);
  \foreach \x in {(A)}
  \node[vertex] (\x) at \x {};
  \end{tikzpicture}\;}

\newcommand{\FlexibleBowtie}{
  \begin{tikzpicture}[scale=0.4, baseline=-1ex]
  \coordinate (A) at (0,0);
  \coordinate (B) at (0.75,0);
  \coordinate (TL) at (-1.0,0.5);
  \coordinate (BR) at (1.75,-0.5);
  \clip (TL) rectangle (BR);
  \draw (A) to [out=120, in=240, loop, min distance=20mm] (A);
  \draw (B) to [out=60, in=300, loop, min distance=20mm] (B);
  \draw (A) -- (B);
  \foreach \x in {(A), (B)}
  \node[vertex] (\x) at \x {};
  \end{tikzpicture}\;}

\newcommand{\FlexibleKfour}{
  \begin{tikzpicture}[scale=0.4, baseline=-1ex]
  \coordinate (A) at (0,0);
  \coordinate (B) at (0,0.5);
  \coordinate (C) at ($(0,0) + (-30:0.5)$);
  \coordinate (D) at ($(0,0) + (-150:0.5)$);
  \coordinate (TL) at (-1,0.7);
  \coordinate (BR) at (1,-0.9);
  \clip (TL) rectangle (BR);
  \draw (A) -- (B);
  \draw (A) -- (C);
  \draw (A) -- (D);
  \draw (B) to [out=30, in=30, min distance=7mm] (C);
  \draw (B) to [out=150, in=150, min distance=7mm] (D);
  \draw (D) to [out=270, in=270, min distance=7mm] (C);

  \foreach \x in {(A), (B), (C), (D)}
  \node[vertex] (\x) at \x {};
  \end{tikzpicture}\;}

\newcommand{\FlexibleBook}{
  \begin{tikzpicture}[scale=0.4, baseline=0ex]
  \coordinate (A) at (0,0);
  \coordinate (B) at (1.25,0);
  \coordinate (TL) at (-0.25,1.25);
  \coordinate (BR) at (1.5,-0.35);
  \clip (TL) rectangle (BR);
  \draw (A) to [out=90, in=90, min distance=15mm] (B);
  \draw (A) to [out=70, in=110, min distance=12mm] (B);
  \draw (A) to [out=50, in=130, min distance=9mm] (B);
  \draw (A) -- (B);
  \foreach \x in {(A), (B)}
  \node[vertex] (\x) at \x {};
  \end{tikzpicture}\;}